\documentclass[12pt,leqno,amscd,amssymb,verbatim, url]{amsart}
\usepackage{amsfonts,amssymb,graphicx}
\usepackage{amsmath,amscd}
\usepackage{hyperref}
\usepackage{mathrsfs}
\let\oldsection=

\renewcommand{\subsection}[1]{\par\vspace{.18in}\noindent\addtocounter{subsection}{1}\setcounter{equation}{0}{\bf\thesubsection\hspace{9pt}#1}}

\usepackage[english]{babel}
\usepackage[latin1]{inputenc}
\usepackage{times}
\usepackage[T1]{fontenc}
\usepackage{amsthm, amsmath, amssymb}
\usepackage{tikz}

\usetikzlibrary{arrows,shapes}
\usetikzlibrary{decorations.pathmorphing}
\usetikzlibrary{calc}
\newtheorem{thm}{Theorem}[subsection]

\makeatletter\let\c@fact\c@theorem\makeatother\newtheorem{lem}[thm]{Lemma}
\newtheorem{cor}[thm]{Corollary}

\newtheorem{prop}[thm]{Proposition}

\theoremstyle{definition}

\newtheorem{defn}[thm]{Definition}

\newtheorem{rem}[thm]{Remark}
\newtheorem{rems}[thm]{Remarks}

\numberwithin{equation}{subsection}
\numberwithin{thm}{section}
\newcommand{\hd}{{\text{\rm hd}}}
\newcommand{\wnabla}{{\widetilde\nabla}}
\newcommand{\wrnabla}{{{\widetilde\nabla}}_{\text{\rm red}}}
\newcommand{\wF}{{\widetilde F}}

\newcommand{\wOmega}{{\widetilde\Omega}}

\newcommand{\grHom}{\text{\rm hom}}
\newcommand{\grExt}{\text{\rm ext}}

\newcommand{\opH}{{\text{\rm H}}}

\newcommand{\wfa}{\widetilde{\mathfrak a}}
\newcommand{\wrDelta}{{\widetilde{\Delta}}^{\text{\rm red}}}

\newcommand{\wX}{{\widetilde{X}}}

\newcommand{\ck}{{\text{\rm coker}}}
\newcommand{\wXi}{{\widetilde\Xi}}
\newcommand{\wgr}{\widetilde{\text{\rm gr}}\,}

\newcommand{\coker}{{\text{\rm coker}}}

\newcommand{\wrad}{\widetilde{\text{\rm rad}}\,}

\newcommand{\sU}{{\mathscr U}}

\newcommand{\whQ}{\widehat Q}

\newcommand{\gr}{\text{\rm gr}}

\newcommand{\wGr}{{\widetilde{\text{\rm Gr}}\,}}

\newcommand{\sA}{{\mathscr A}}

\newcommand{\Ext}{{\text{\rm Ext}}}

\newcommand{\fa}{{\mathfrak a}}

\newcommand{\Amod}{A\mbox{--mod}}

\newcommand{\Hom}{\text{\rm Hom}}

\newcommand{\sE}{\operatorname{{\mathscr E}}}

\newcommand{\soc}{\operatorname{soc}}
\newcommand{\sO}{{\mathscr{O}}}

\newcommand{\rad}{\operatorname{rad}}

\newcommand{\Dist}{\operatorname{Dist}}

\newcommand{\rDelta}{\Delta^{\text{\rm red}}}
\newcommand{\rnabla}{\nabla_{\text{\rm red}}}
\newcommand{\wM}{{\widetilde{M}}}
\newcommand{\wA}{{\widetilde{A}}}
\newcommand{\wB}{{\widetilde{B}}}
\newcommand{\wR}{{\widetilde{R}}}
\newcommand{\wU}{{\widetilde U}}
\newcommand{\wDelta}{{\widetilde{\Delta}}}

\newcommand{\wN}{{\widetilde{N}}}

\newcommand{\wP}{{\widetilde{P}}}
\newcommand{\wQ}{{\widetilde{Q}}}
\newcommand{\wJ}{{\widetilde{J}}}
\newcommand{\Jan}{\Gamma_{\text{\rm Jan}}}

\newcommand{\wL}{{\widetilde{L}}}
\newcommand{\sD}{{\mathscr{D}}}

\newcommand{\Tor}{{\text{\rm Tor}}}

\newcommand{\Gmod}{{G{\text{\rm --mod}}}}

\newcommand{\blist}{\begin{list}{\rom{(\roman{enumi})}}{\setlength
{\leftmarg in}{0em} \setlength{\itemindent}{7ex}
\setlength{\labelsep}{2ex}\setlength{\listparindent}{\parindent}
\usecounter{enumi}}}
\newcommand{\elist}{\end{list}}

\begin{document}
\begin{abstract} Let $G$ be a semisimple algebraic group over an algebraically closed field $k$ of
positive characteristic $p$. Under some restrictions on the size of $p$ (which in some cases require
validity of the Lusztig character formula), the present paper establishes
new results on the $G$-module structure of $\Ext^\bullet_{G_1}(V,W)$ when $V,W$ belong to several
important classes of rational $G$-modules, and $G_1$ denotes the first Frobenius kernel of $G$. For example,
it is proved that, if $L,L'$ are ($p$-regular) irreducible $G_1$-modules, then $\Ext^n_{G_1}(L,L')^{[-1]}$ has a good
filtration with computable multiplicities. This and many other results depend on the entirely new technique of using methods
of what we call forced gradings
 in the representation theory of $G$, as developed by the authors in   \cite{PS9}, \cite{PS10} and \cite{PS11}, and extended here.

 In addition to providing proofs, these methods lead effectively to a new conceptual framework for the study
 of rational $G$-modules, and, in this context, to the introduction of a new class of graded finite
 dimensional algebras, which we call Q-Koszul algebras. These algebras are similar to Koszul algebras, but
 are quasi-hereditary, rather than semisimple, in grade 0. \end{abstract}

 \title[New graded methods in the homological algebra of semisimple groups]{New graded methods in the homological algebra of semisimple algebraic groups}\author{Brian J. Parshall}
\address{Department of Mathematics \\
University of Virginia\\
Charlottesville, VA 22903} \email{bjp8w@virginia.edu {\text{\rm
(Parshall)}}}
\author{Leonard L. Scott}
\address{Department of Mathematics \\
University of Virginia\\
Charlottesville, VA 22903} \email{lls2l@virginia.edu {\text{\rm
(Scott)}}}

\thanks{Research supported in part by the National Science
Foundation}
\maketitle
\section{Introduction}
Let $G$ be a semisimple, simply connected algebraic group over an algebraically closed field $k$ of
positive characteristic $p$. The irreducible rational $G$-modules $L(\lambda)$ are indexed by the
set $X(T)_+$ of dominant weights. When $\lambda\in X(T)_+$, $L(\lambda)$ occurs as the
head (resp., socle) of the Weyl module $\Delta(\lambda)$ (resp., dual Weyl module $\nabla(\lambda)$).
The structure and cohomology of the modules $\Delta(\lambda)$ and $\nabla(\lambda)$, for all $\lambda\in X(T)_+$,
occupy a central place in the modular representation theory of semisimple groups. To give a recent
example, write $\lambda=\lambda_0+p\lambda_1$,
where $\lambda_0$ is a restricted dominant weight and $\lambda_1$ is dominant, and define $\Delta^p(\lambda):=L(\lambda_0)
\otimes \Delta(\lambda_1)^{[1]}$, where $\Delta(\lambda_1)^{[1]}$ denotes the Frobenius ``twist" of $\Delta(\lambda_1)$.  In 1980, Jantzen \cite{Jan} asked
if any Weyl module $\Delta(\lambda)$ has a $\Delta^p$-filtration, i.~e., a filtration by $G$-submodules with sections $\Delta^p(\gamma)$, for various $\gamma\in X(T)_+$.  In \cite{PS11}, the authors answered positively Jantzen's question under the hypothesis that the Lusztig character formula (LCF) holds, and $p\geq 2h-2$ is odd, where
$h$ is the Coxeter number of $G$.
The LCF is known to hold for very large $p$
depending on the root system of $G$ (see \cite{AJS} and \cite{F}).\footnote{Williamson \cite{Wil} has recently posted results stating that the original Lusztig conjecture with its proposed
bound 
$p\geq h$ can fail for primes $p$ of this size---i.~e., without stronger conditions on $p$. 
 Williamson has stated that $p\geq f(h)$
is insufficient when $f(h)$ is linear in $h$, and he even proposes that a sufficient $f(h)$ must be
exponential in $h$. To put this in perspective, however, the Weyl group order of $SL_n$ is $n!=h!$
which is exponential in $h$, but not ``huge" in the sense of the sufficient bounds on $p$ given by Fiebig \cite{F}.} When
it holds, the modules $\Delta^p(\lambda)$ (resp., $\nabla_p(\lambda)$), $\lambda\in X(T)_+$, identify with certain
modules $\rDelta(\lambda)$ (resp., $\rnabla(\lambda)$) arising from ``reduction mod $p$" of the quantum
enveloping algebra at a $p$th root of unity associated to $G$; see \S2.4. This connection with quantum enveloping algebras plays an essential role in \cite{PS11}, fitting in well with the new forced-graded methods developed by the authors there and in \cite{PS9} and \cite{PS10}.

The present paper builds on
 these methods, extending their scope from the module structure theory to the study of homological
 resolutions. Many new results for the homological algebra of rational $G$-modules emerge, as well as some promising forced-graded structures. Before elaborating on the latter,  we briefly mention three specific new results.

  First,
  let $G_1$ be the first Frobenius kernel
of $G$. The representation theory of $G_1$ coincides with the representation theory of the restricted enveloping algebra $u=u({\mathfrak g})$ of $G$. Given rational $G$-modules $V,W$, the
spaces $\Ext^n_{G_1}(V,W)$, $n\geq 0$, carry the natural structure of twisted $G$-modules, that is,
the natural action of $G_1$ through its containment $G_1\subset G$ is trivial.
Except
in special cases, e.~g., $n=1$, little is known about the structure of the untwisted $G$-modules $\Ext^n_{G_1}(V,W)^{[-1]}$,
even when $V$ and $W$ are taken to be irreducible, Weyl or dual Weyl modules. Now assume that
$p\geq 2h-2$ is odd and that the LCF holds for $G$. Let $V=L$ and $W=L'$  be
irreducible $G_1$-modules. Then $L\cong L(\lambda)|_{G_1}$ and $L'\cong L(\mu)|_{G_1}$ for restricted dominant weights $\lambda,\mu$ which we assume are $p$-regular.  Theorem \ref{nextmainresult} establishes that $\Ext^n_{G_1}(L(\lambda),L(\mu))^{[-1]}$ has a ``good" or  $\nabla$-filtration---that is, a filtration by
 $G$-submodules with sections of the form $\nabla(\gamma)$, for  various $\gamma\in X(T)_+$. In addition,
 the multiplicity of any $\nabla(\gamma)$
as a section in this filtration can be combinatorially determined in terms of coefficients of Kazhdan-Lusztig polynomials for the affine Coxeter group of $G$; see Theorem \ref{calcthm}.\footnote{The conclusion of Theorem \ref{nextmainresult}  may fail if $p$ is small.
 For example, if $G$ has type $F_4$
and $p=2$, then according to \cite[4.11]{Sin94},
$$\Ext^1_{G_1}(L(0), L(\varpi_2))^{[-1]}
\cong L(0)\oplus L(\varpi_1),$$
does not have a $\nabla$-filtration, since $L(\varpi_1)\not=\nabla(\varpi_1)$. We thank Peter Sin for
pointing out his paper to us. See also David Stewart \cite{Stew}, which largely extends Sin's $F_4$ calculations
to twisted $F_4$.} To our knowledge,  Theorem \ref{nextmainresult} and
Theorem \ref{calcthm} give the first general results in the literature on the $G$-module structure of $\Ext_{G_1}$-groups between
irreducible modules.

Second, Theorem \ref{Jantzen} proves, under the same assumptions about $p$, that, given any $p$-regular weight
$\lambda\in X(T)_+$, restricted dominant  weight $\mu$,  and integer $n\geq 0$, the rational $G$-modules $\Ext^n_{G_1}(\Delta(\lambda),L(\mu))^{[-1]}$ and  $\Ext^n_{G_1}(L(\mu),\nabla(\lambda))^{[-1]}$ both have $\nabla$-filtrations. Again, the multiplicities of any $\nabla(\gamma)$ can be determined in terms
of Kazhdan-Lusztig polynomial coefficients; see Theorem \ref{calcthm2}.\footnote{Theorem \ref{Jantzen} is suggested by the work \cite{KLT} of Kumar, Lauritzen, and Thomsen (improving
earlier work \cite{AJ} of Andersen and Jantzen), showing that, if $p>h$, then $\opH^n(G_1,\nabla(\lambda))^{[-1]}=\Ext^n_{G_1}(k,\nabla(\tau))$ always has a
$\nabla$-filtration. Our result, although it presently requires much larger values of $p$, considerably extends this result and rests on entirely different methods. The general question asking, given a rational $G$-module $V$, whether $H^n(G_r,V)^{[-r]}$
has a $\nabla$-filtration goes back at least to Donkin's paper \cite{Don87} who conjectured a positive answer
if $V$ has
a $\nabla$-filtration. Counterexamples were later given by van der Kallen \cite{vdK93}. }

Third, let $\fa$ be the sum of the $p$-regular blocks in the restricted enveloping algebra $u$ of $G$.
When $p>h$ and the LCF holds, an important result proved in \cite{AJS} establishes that $\fa$ is a Koszul algebra, and so, in particular, it has a natural positive grading. The positive grading exists, inherited from 
the quantum analogue of $\fa$, without the LCF assumption; see \cite{PS10}. Theorem \ref{DeltaKoszul}
proves that, given any $\nu\in X_{\text{\rm reg}}(T)_+$, the Weyl module $\Delta(\nu)$ has the structure of a graded $\fa$-module, provided  $p\geq 2h-2$ is odd. If, in addition, the LCF holds, this graded structure is
{\it linear}. In other words, if $P_\bullet\twoheadrightarrow \Delta(\lambda)$ is a minimal
graded $\fa$-projective resolution, then $\ker(P_{i+1}\to P_i)$ is generated by its terms of grade $i+2$. This
fact plays an important role in other results in this paper on the structure of $G$-module categories (see
Corollary 3.8). In part, it grows out of a related result \cite[Thm. 10.9]{PS9}, establishing $\fa$-gradings
(but not linearity) for dual Weyl modules in some cases. But, surprisingly, it is the quantum version of \cite[Thm. 8.7]{PS9} of these results which we apply to study the $G$-module case here.

Our results on $G$-modules are modeled on (and require in a strong way) similar results for quantum enveloping algebras
at roots of unity established by the authors in \cite{PS9}.

Underlying these results are gradings forced upon the algebras controlling
the representation theory of the modules we study.  Our philosophy has been
that it is likely to be extremely difficult, if not impossible, to impose
actual positive gradings\footnote{By a slight abuse of terminology, a positively graded algebra has, by
definition,
nonzero grades only in grades $n\geq 0$.} on all of these algebras,  although (as noted above) they do exist,
under various assumptions, in the restricted enveloping algebra case. So, using filtrations related to the restricted enveloping
algebra gradings, we pass to the associated graded algebras in all cases,
thus forcing a grading.  Once this is done, we do not immediately know,
if any of the nice properties, e.~g.,  quasi-heredity, carry over to the new
graded algebras, or even if it is possible to work with them as a
substitute for the original algebras in any meaningful way. However, from the start, a recent
goal,
continued in this paper, is to show that this is possible, thereby giving a
genuinely viable alternative to finding from the start a positive grading.
Indeed, because  forced-graded structures come with built-in compatibility
properties among the different algebras used, there are
advantages to using them over actual gradings on the original algebras.
There are, of course, some disadvantages. In particular, except in the case
of the restricted enveloping algebra, there is no general ``forget the
grading" functor that allows passage back to  the original algebras and
modules.  However,  such a forgetful functor does exist for some graded
modules,  including all those which are completely reducible for the
restricted enveloping algebra. We are able to use this
functor to communicate from the forced-graded setting back to the original module
categories. The results discussed above, as well as our $p$-Weyl filtration
result \cite{PS11}, demonstrate the success of this approach, providing
genuinely new advances in the structure and homological
algebra of algebraic groups through proofs relying on forced-graded
constructions.\footnote{For a survey of some of the literature in graded representation theory, see the introduction to \cite{PS9}
as well as \cite{PS9a}.}

The three results discussed above were chosen because the statements involve only the
classical language of algebraic groups. But once the forced-graded
 framework is in
place, many further results may be stated. Immediately,  we observe from \cite{PS10} that the new
graded algebras $\wgr A$ associated to standard quasi-hereditary algebras $A=A_\Gamma$ (associated to
a finite  ideal $\Gamma$ of $p$-regular dominant weights) are themselves quasi-hereditary. Their Weyl
modules    $\wgr \Delta(\lambda)$ are forced-graded versions of the Weyl modules $\Delta(\lambda)$ for $A$.  Moreover, the
present paper shows in  Theorem 5.3(b) and Theorem \ref{GExt} ({\it both not} assuming the LCF) that there are graded isomorphisms
$$\begin{cases}\Ext^\bullet_{G}(\Delta(\lambda),\rnabla(\mu))\cong\Ext^\bullet_A(\Delta(\lambda),\rnabla(\mu))\cong
\Ext^\bullet_{\wgr A}(\wgr \Delta(\lambda),\rnabla(\mu)),\\
\Ext^\bullet_G(\rDelta(\lambda),\rnabla(\mu))\cong\Ext^\bullet_A(\rDelta(\lambda),\rnabla(\mu))
\cong\Ext^\bullet_{\wgr A}(\rDelta(\lambda),\rnabla(\mu)),\\
\Ext^\bullet_G(\rDelta(\lambda),\nabla(\mu))\cong\Ext^\bullet_A(\rDelta(\lambda),\nabla(\mu))
\cong\Ext^\bullet_{\wgr A}(\rDelta(\lambda),\wgr^\diamond\nabla(\mu)),\end{cases}
$$
  where $\wgr^\diamond\nabla(\mu)$ denotes the
dual Weyl module for $\wgr A$ of highest weight $\mu$.  {\it Accordingly, the homological algebra of important
classes of rational $G$-modules is placed in the setting of forced-graded algebras.}
 These results just assume that $p\geq 2h-2$ is odd, and regard $\wgr A$ is an ungraded algebra.
 When the LCF is assumed in addition, 
 we prove two new results in this paper
of a graded nature. First, if the ideal $\Gamma$ of
$p$-regular weights is
contained in the Jantzen region, then $\wgr A$ is a Koszul algebra (and even has a graded Kazhdan-Lusztig
theory in the sense of \cite{CPS1a};
see Corollary \ref{KLtheory}).\footnote{A weaker result was established in \cite[Thm. 10.6]{PS9}, and the present result was
promised there in Remark 10.7(a). }

Second,  with $\Gamma$ an arbitrary ideal of $p$-regular weights,
Theorem \ref{QKoszultheorem} shows that $\wgr A$ is a ``Q-Koszul algebra,"
an algebra with a new Koszul-like property defined and studied in this paper; see Definition \ref{QKoszul}. 
In addition, there is the stronger, companion notion of a ``starndard Q-Koszul algebra, also introduced
in \S3, and Theorem \ref{standardtheorem} shows that $\wgr A$ is even standard Q-Koszul under the
same hypotheses.  In part, these results rely on the observation that the algebra $(\wgr A)_0$ (the grade 0 term of
$\wgr A$) is itself quasi-hereditary with Weyl (resp., dual Weyl) modules $\rDelta(\lambda)$ (resp.,
$\rnabla(\lambda)$), $\lambda\in\Gamma$. The algebra $(\wgr A)_0$ replaces the semisimple grade 0
term in the Koszul algebra case.

The Q-Koszul and standard Q-Koszul structures
have been proved for forced-graded algebras $\wgr A$ associated to $G$ only when the characteristic
$p$ is large enough that the LCF holds. Thus, these results (as well as the good filtration results in \S5) 
are presently {\it generic} in nature. However, the authors believe that, independently of the validity of
the LCF, there are tractable versions of the algebras $\wgr A$ for smaller $p$ (including some $p<h$) which
are likely Q-Koszul.\footnote{For example, this appears to be the case for $p=2$ and with $A$ a Schur algebra $S(n,r)$, $r\leq 5$. In another direction, the Humphreys-Verma conjecture on projective indecomposable $G_1$-modules becomes a theorem, valid for all $p$, in a
forced-graded setting \cite{PS12}. At present, this conjecture is only known if $p\geq 2h-2$. This is the main reason it is assumed that $p\geq 2h-2$ in this paper.} The authors intend to return to this topic in a later paper. Also,  another sequel \cite{PS13} obtains some of the Q-Koszul results of this paper under
weaker hypotheses, but still assuming that a version the LCF holds on a given poset of weights.

Many of the main results of this paper assume the validity of the Lusztig character formula (which is presently only known to hold for very large $p$, see footnote 1).  However,  even when the LCF is assumed to hold, many results are established for
dominant weights outside the Jantzen region---giving homological and structural results not covered by the original conjecture or
its immediate consequences. In addition, some results do not assume the LCF. For example, we mention 
again the deep Theorem \ref{DeltaKoszul}(a) which shows that  standard
modules $\Delta(\lambda)$, $\lambda$ $p$-regular, have a natural graded structure for $\fa$. Here we use the (positive) grading on $\fa$ proved by the authors in \cite{PS12}, arising naturally, but non-trivially, from
quantum group considerations when $p\geq 2h-2$ is odd. Other examples include the quite satisfying
identifications of Theorem \ref{nextmainresult} and Theorem \ref{GExt}(b), described above and proved under the same
hypothesis.

\medskip
\subsection{Some Elementary Notation.}

\begin{enumerate}

\item $(K,\sO,k)$: $p$-modular system. Thus, $\sO$ is a DVR with maximal ideal ${\mathfrak m}=(\pi)$,  fraction field $K$, and residue
field $k$. An $\sO$-lattice $\wM$ is, by definition, an $\sO$-module which is free and of finite rank.
A particular $p$-modular system will be required. Let $p>0$ be a fixed odd prime.
$\sO$ will be a DVR with maximal ideal $\mathfrak m=(\pi)$, fraction field $K$ of characteristic 0, and residue field $\sO/{\mathfrak m}
\cong k=\overline{\mathbb F}_p$. We can (and will) assume that $\sO$ is complete and contains a $p$th root $\zeta\not=1$ of unity.
Let $\sA:={\mathbb Z}[v,v^{-1}]_{\mathfrak n}$,
the localization of the ring of integral Laurent polynomials in a indeterminate $v$ at the maximal ideal ${\mathfrak n}:=(v-1,p)$. Regard $\sA$ as a subring of the function field ${\mathbb Q}(v)$.  There is a
natural ring homomorphism $\sA\to\sO$,  $v\mapsto\zeta$.

\item   An $\sO$-order is an $\sO$-algebra $\wA$ which is also a $\sO$-lattice.  If $\wA$ is an $\sO$-order, then an $\wA$-lattice is, by definition, an $\wA$-module $\wM$ which is also a $\sO$-lattice. Let $\wA_K:= K\otimes_{\sO}\wA$ and
$A:=k\otimes_{\sO}\wA$. More generally, if $\wM$ is an $\wA$-module, define $\wM_K:=K\otimes_{\sO}\wM$
and $M= \wM_k:=k\otimes_{\sO}\wM$.

\item For an $\wA$-lattice $\wM$, define $\wrad^n\wM:=\wM\cap \rad^n\wM_K$, where $\rad^n\wM_K$ denotes the $n$th-radical of
the $\wA_K$-module $\wM_K$. Of course, $\rad^n\wM_K=(\rad^n\wA_K)\wM_K$.

Dually, let $\widetilde\soc^{-n}\wM:=\soc^{-n}\wM_K\cap\wM$, $n=0,1,\cdots$, where $\{\soc^{-n}\wM_K\}_{n\geq 0}$
is the socle series of $\wM_K$.

\item If $\wM$ is an $\wA$-lattice, then $\gr \wM:=\bigoplus_{n\geq 0}\wrad^n\wM/\wrad^{n+1}\wM$ is a positively graded lattice for the $\sO$-order
\begin{equation}\label{graded1}\gr\wA:=\bigoplus_{n\geq 0}\wrad^n\wA/\wrad^{n+1}\wA.\end{equation}

\item A $\wA$-lattice $\wM$ is called {\it $\wA$-tight} (or just {\it tight}, if $\wA$ is clear from context) if
\begin{equation}\label{rad}(\wrad^n\wA)\wM=\wrad^n\wM,\quad\forall n\geq 0.
\end{equation}
Clearly, if $\wM$ is $\wA$-projective, then it is tight. (Many other $\wA$-lattices are tight.)

%Looking ahead to \S2.3 (and using the notation there), when $\wB=\wA_\Gamma$ and
%$\wM$ is an $\wA_\Gamma$-lattice, then $\wM$ may also be viewed as a $\wfa$-lattice, whenever $\wfa$
%is an $\sO$-order and $\wfa\to\wA_\Gamma$ is an $\sO$-algebra homomorphism.  Fortunately, in the
%situation of \S4, the lattice $\wM$ is $\wA_\Gamma$-tight if and only if it is $\wfa$-tight, using
%\cite[Cor. 3.8]{PS11} and \cite[Cor. 4.16]{PS10}. In fact, the argument in \cite[Cor. 3.8]{PS11} shows that $%%\wrad^n\wA_\Gamma=(\wrad^n\wfa)\wA_\Gamma
%=\wA_\Gamma(\wrad^n\wfa)$, provided $p\geq 2h-2$ is odd and $\Gamma$ is non-empty ideal of $p$-%regular
%weights. Also, $\wrad^n\wM$ is physically the same submodule of $\wM$, whether
% $\wM$ is viewed as an
%$\wA_\Gamma$-module or an $\wfa$-module.

\item Now let $\wfa$ be an $\sO$-subalgebra of $\wA$. (More generally, we can assume that $\wfa$ is an
order and $\wfa\to\wA$ is a homomorphism.)  Then items (2)--(5) all make perfectly good sense
using $\wfa$ in place of $\wA$. If $\wM$ is an $\wA$-lattice, then it is an $\wfa$-lattice.
 In latter contexts (see, e.~g., \S2.3), it will usually be the case that
$(\rad^n\wfa_K)\wA_K=\rad^n\wA_K$, for all $n\geq 0$. In that case, if $\wM$ is an $\wA$-lattice,
then $\wrad^n\wM$ can be constructed viewing $\wM$ as an $\wA$-lattice or as an $\wfa$-lattice. Both
constructions lead to identical $\sO$-modules.
Ambiguities of a formal nature may still arise as to whether it is more appropriate to use $\wfa$ or $\wA$, but are
generally resolved by context. Similar remarks apply for $\gr\wM$. Often the $\wA$-tightness
of $\wM$ is the same as its $\wfa$-tightness; see \cite[Cor. 3.8]{PS11} and its elaboration at the
end of \S2.5 below.

\item Finally, suppose that $\wfa\to\wA$ is a homomorphism of $\sO$-orders. Assume that the image of
$\wfa$ is normal in $\wA$. Let $A=\wA_k$, $\fa:=\wfa_k$, and consider
an $A$-module $M$.   Define
\begin{equation}\label{gradedzoo}\begin{cases}(1)\,\gr M:=\bigoplus_{n\geq 0}(\rad^nA) M/(\rad^{n+1}A)M;\\
(2)\,\gr_\fa M:=\bigoplus_{n\geq 0}(\rad^n\fa) M/(\rad^{n+1}\fa) M;\\
(3)\,\wgr M:=\bigoplus_{n\geq 0}(\wrad^n\wfa) M/(\wrad^{n+1}\wfa) M.\end{cases}
\end{equation}
Each of these is graded modules for $\gr A$ and $\wgr A$. Though it will not often be used, (3) makes
sense when $M$ is replaced any $\wA$-lattice $\wM$, i.~e., we put $\wgr\wM:=\bigoplus_{n\geq 0}
(\wrad^n\wfa)\wM/(\wrad^{n+1}\wfa)\wM$. It will often be the case that $\wgr\wM\cong\gr\wM$, which 
implies also $\wgr M\cong(\gr\wM)_k$ if $M=\wM_k$. A necessary and sufficient condition for either of these
natural isomorphisms in the context of \S2.3 is the $\wfa$-tightness of $\wM$; see \cite[Lem. 3.5]{PS11}.
\end{enumerate}

\medskip
 For a finite dimensional algebra $A$ (over some field), let $\Amod$ be the category of all finite
 dimensional $A$-modules.
 In the rest of this paragraph assume that $A=\bigoplus_{n\geq 0}A_n$ is a positively graded algebra. Let $A$-grmod be the category of $\mathbb Z$-graded (finite dimensional) $A$-modules. Given graded $A$-modules $M,N$ and $n\in{\mathbb N}$,
 $\grExt^n_A(M,N)$ denotes the space of $n$-fold extensions computed in the category $A$-grmod. See Remark \ref{discussion} for some elementary comments on
 the existence of projective covers in $A$-grmod. When $n=0$, the space of
 homomorphisms $M\to N$ preserving grades is denoted $\grHom_A(M,N)= \grExt^0_A(M,N)$. For
 $M,N\in\Amod$ (not necessarily graded modules) and $n\in{\mathbb N}$, the space of
 $n$-fold extensions is denoted $\Ext^n_A(M,N)$. The bifunctors $\grExt^\bullet$ and $\Ext^\bullet$
 are related as follows.  If $M,N\in A$--grmod,
 then is a natural isomorphism
 \begin{equation}\label{gradedungraded}\Ext^n_A(M,N)\cong\bigoplus_{r\in\mathbb Z}\grExt^n_A(M,N\langle r\rangle ),\quad
 \forall n\in{\mathbb N}.\end{equation}
 In this expression, $N\langle r\rangle \in A$--grmod is the $r$th shift of $N$, i.~e., $N\langle r\rangle _i:=N_{i-r}$.

\section{Varia} This section collects together some useful material on several topics treated in this paper.

\subsection{Algebraic groups.}
Let $G$ be a simple, simply connected algebraic group defined and split over ${\mathbb F}_p$, where $p$ is a prime integer.\footnote{The case when $G$ is semisimple, or even reductive, is easily reduced to the case when $G$ is simple.}  Let $R$ be the root system of $G$ relative to a fixed
maximal split torus $T$. Fix a Borel subgroup $B\supset T$ with opposite Borel subgroup $B^+$ determining
a set $R^+$ of positive roots.  Given
$\lambda\in X(T)_+$ (the set of dominant weights), $\Delta(\lambda)$ (resp., $\nabla(\lambda)$) will denote the Weyl module (resp., dual Weyl module) of highest weight $\lambda$. We generally follow the
standard notation for $G$ and its representation theory as listed in \cite[pp. 569--572]{JanB} (except
that $\Delta(\lambda)$ is denoted $V(\lambda)$ and $\nabla(\lambda)$ is
 denoted $\opH^0(\lambda)$ there).\footnote{Sometimes, in the context of quasi-hereditary algebras, $\Delta(\lambda)$ and $\nabla(\lambda)$ are called the ``standard" and ``costandard" modules, respectively, of highest weight
 $\lambda$.} If $\lambda\in X(T)_+$ and $\lambda^\star:=-w_0\lambda$ (where $w_0$ is the longest
 word in the Weyl group $W$ of $G$),  then $\Delta(\lambda)$ has linear dual $\Delta(\lambda)^*\cong
 \nabla(\lambda^\star)$.

 For any affine algebraic group scheme $H$, let
 $H$--mod be the category of finite dimensional rational (left) $H$-modules. The category $H$--mod fully
 embeds into the category of finite dimensional modules for the distribution algebra $\Dist(H)$ of $H$. See \cite[Chps.
 7,8]{JanB}. In addition, if $H=G$, this embedding is an equivalence of categories; see \cite[p. 171]{JanB}.  In this case, the classical
 Kostant $\mathbb Z$-form (an ``order" of infinite rank) $\Dist_{\mathbb Z}(G):=\sU_{\mathbb Z}({\mathfrak g}_{\mathbb C})$ \cite[Ch. 7]{Hump} provides an
 integral form for $\Dist(G)$, i.~e., $\Dist(G)\cong k\otimes_{\mathbb Z}\Dist_{\mathbb Z}(G)$ (as Hopf algebras).
 For any commutative algebra $\sO$, write $\Dist_\sO(G):=\sO\otimes_{\mathbb Z}\Dist_{\mathbb Z}(G)$.
 In particular, if $\sO=K$ is a field of characteristic 0, then $\Dist_K(G)$ is the universal enveloping algebra
 of the split semisimple Lie algebra ${\mathfrak g}_K$ over $K$, having the same root system as $G$.

  For a positive integer $r$ and a rational $G$-module $V$,  $V^{[r]}$ denotes the pull-back of $V$ through the
  $r$th power $F^r$ of the Frobenius morphism $F:G\to G$.  Let $G_r$ be the scheme-theoretic  kernel of $F^r$,
 and let $G_rT$ be the pull-back of $T$ through $F^r$. For
 $\lambda\in X(T)$,  $\whQ_r(\lambda)$ denotes the injective envelope of the irreducible
$G_rT$-module $\widehat L_r(\lambda)$ of highest weight $\lambda$.

 Throughout this paper, we usually make the assumption that $p\geq 2h-2$. This means that, if $\lambda_0\in X_r(T)$ (the set of $r$-restricted dominant weights), then the $G_rT$-module structure on $\whQ_r(\lambda_0)$ extends uniquely to a rational $G$-module structure. In the special case in
which $r=1$, this rational $G$-module will be denoted by
$Q^\flat(\lambda_0)$;
 it the projective cover of $L(\lambda_0)$ in the subcategory of $\Gmod$ generated by
$L(\gamma)$, with
$\gamma\leq\lambda_0':=2(p-1)\rho+w_0\lambda_0$;  see \cite[Ch. 11]{JanB} for details. We also generally
assume that $p$ is odd, so that previous results can be easily quoted. When $p=2\geq 2h-2$, then $G=SL_2$,
which is usually easy to treat directly.

Given $\lambda\in X(T)_+$, write $\lambda=\lambda_0+p\lambda_1\in X(T)_+$,  where $\lambda_0\in X_1(T)$ and $\lambda_1\in X(T)_+$. The
indecomposable rational $G$-modules
\begin{equation}\label{generalizedQsandPs}
\begin{cases}
Q^\sharp(\lambda):=Q^\sharp(\lambda_0)\otimes\nabla(\lambda_1)^{[1]}\\
P^\sharp(\lambda):=Q^\sharp(\lambda_0)\otimes\Delta(\lambda_1)^{[1]}.
\end{cases}
\end{equation}
will play an important role.
Of course, the restrictions $Q^\sharp(\lambda)|_{G_1T}$ and $P^\sharp(\lambda)|_{G_1T}$ are injective and projective (but not
indecomposable, unless $\lambda_1=0$). By \cite[Prop. 2.3]{PS11}, $Q^\sharp(\lambda)$ (resp.,
$P^\sharp(\lambda)$) has a $\nabla$-filtration (resp., $\Delta$-filtration), namely, a filtration with
sections of the form $\nabla(\gamma)$ (resp., $\Delta(\gamma)$), for $\gamma\in X(T)_+$.

\subsection{Quantum enveloping algebras.}
Let $\widetilde U'_\zeta$ be the (Lusztig) $\sA$-form
of the quantum enveloping algebra ${\mathbb U}_v$ associated to the Cartan matrix of the root system $R$ over the function
field ${\mathbb Q}(v)$.
 Put
$$\widetilde U_\zeta=\sO\otimes_\sA U'_\zeta/\langle K_1^p-1,\cdots, K_n^p-1\rangle.$$
Finally, set $U_\zeta=K\otimes_{\sO}\widetilde U_\zeta$, so that $\widetilde U_\zeta$ is an integral
$\sO$-form of the quantum enveloping algebra $U_\zeta$ at a $p$th root of unity.
Put $\overline U_\zeta=\widetilde
U_\zeta/\pi\widetilde U_\zeta$, and let $I$ be the ideal in
$\overline U_\zeta$ generated by the images of the elements $K_i-1$,
$1\leq i\leq n$. By \cite[(8.15)]{L1},
\begin{equation}\label{hyperalgebra}
\overline U_\zeta/I\cong \Dist(G).\end{equation}
 A rational $G$-module $M$ is said to {\it lift} if there is a
$U_\zeta$-
lattice $\wM$ such that $M\cong \wM/\pi\wM$ as rational $G$-modules.

The category $U_\zeta$--mod of
finite dimensional and integrable type 1 modules is a highest weight category (in the sense
of \cite{CPS-1}) with
irreducible (resp. standard, costandard) modules $L_\zeta(\lambda)$
(resp., $\Delta_\zeta(\lambda)$, $\nabla_\zeta(\lambda)$),
$\lambda\in X(T)_+$. For $\mu\in X(T)_+$, $\text{\rm
ch}\,\Delta_\zeta(\mu)=\text{\rm ch}\,\nabla_\zeta(\mu)=\chi(\mu)$ (Weyl's character formula).

There is a surjective (Hopf) algebra homomorphism
\begin{equation}\label{twistedFrobenius}
\wF:\widetilde U_\zeta\twoheadrightarrow \Dist_\sO (G),\end{equation}
which, after base change to $K$, defines the Frobenius morphism
\begin{equation}\label{Frobenius}F:U_\zeta\twoheadrightarrow \Dist_K(G).\end{equation}
If $M$ is a module for $\Dist_K(G)$, let $M^{[1]}$ be the $U_\zeta$-module obtained by making $U_\zeta$
act through $F$. Similarly, if $\wM$ is a $\Dist_\sO(G)$-module, let $\wM^{[1]}$ be the $\wU_\zeta$-module
obtained by making $\wU_\zeta$ act through $\wF$. In particular, if $\lambda\in X(T)_+$, $\wDelta(\lambda)^{[1]}$
(resp., $\wnabla(\lambda)^{[1]}$) is the $\sU_\zeta$-module obtained from the Weyl (resp., dual Weyl)
lattice $\wDelta(\lambda)$ (resp., $\wnabla(\lambda)$) of the irreducible ${\mathfrak g}_{\mathbb C}$-module
$L_{\mathbb C}(\lambda)$ of highest weight $\lambda$.\footnote{Thus, for example, $\wDelta(\lambda):=
\Dist_\sO(G)\cdot v^+$, if $v^+\in L_K(\lambda)$ is a highest weight vector.}

The rational $G$-modules $P^\sharp(\gamma)$ and $Q^\sharp(\gamma)$ defined
in (\ref{generalizedQsandPs}) lift to $\wU_\zeta$-lattices, denoted $\wP^\sharp(\gamma)$ and
$\wQ^\sharp(\gamma)$, respectively. If $\gamma=\gamma_0$ is restricted, these modules may be
defined as the unique (up to isomorphism) $\wU_\zeta$-lattices lifting $P^\sharp(\gamma_0)$
and $Q^\sharp(\gamma_0)$. We refer ahead to the discussion following display (\ref{Lambda}) for more details. In general, for $\gamma=\gamma_0+p\gamma_1$, with $\gamma_0\in X_1(T)_+$ and $\gamma_1\in X(T)_+$, we have $\wP^\sharp(\gamma)=\wP^\sharp(\gamma_0)\otimes
\wDelta(\gamma_1)^{[1]}$ and $\wQ^\sharp(\gamma)
=\wQ^\sharp(\gamma_0)\otimes\wnabla(\gamma_1)^{[1]}$.

Let $u_\zeta$ be the small quantum enveloping algebra. It is a Hopf subalgebra of $U_\zeta$ and admits
an integral form $\widetilde u_\zeta$, which is a subalgebra of $\wU_\zeta$. As such $\widetilde u_\zeta$
is a lattice of rank $p^{\dim\mathfrak g}$. Let $u'_\zeta$ be the product of the $p$-regular blocks of $u_\zeta$
and define $\widetilde u'_\zeta:=\widetilde u_\zeta\cap u'_\zeta$. Then $\widetilde u'_\zeta$ is a direct
factor of $\widetilde u_\zeta$. In addition, $u':=k\otimes \widetilde u_\zeta'$ is the direct product of the regular
blocks in the restricted enveloping algebra $u$ of $G$.

\subsection{Finite dimensional algebras.} A dominant weight $\lambda$ is $p$-regular if $(\lambda+\rho,\alpha^\vee)\not\equiv 0$ mod$\,p$,  for all roots $\alpha\in R$. The set $X_{\text{\rm reg}}(T)_+$ of
$p$-regular dominant weights is
a poset, setting $\lambda\leq\mu\iff \mu-\lambda\in {\mathbb N}R^+$.  (There is a similar partial order
on entire set $X(T)_+$ of dominant weights, though this paper focuses
on the $p$-regular weights.)  A subset $\Gamma$ of a poset $\Lambda$ is called an ideal if $\Gamma\not=\emptyset$
and, given $\lambda\in\Lambda$ and $\gamma\in\Gamma$,  if $\lambda\leq \gamma$, then $\lambda\in\Gamma$.
Write $\Gamma\trianglelefteq\Lambda$ in this case.

To a finite ideal $\Gamma$ in $X_{\text{\rm reg}}(T)_+$, there is attached two finite dimensional algebras; the
first, denoted $A_\Gamma$, is over $k$, and the second, denoted $A_{\zeta,\Gamma}$, is over $K$. These algebras capture
some of the representation theory of $G$ and $U_\zeta$, respectively. Furthermore, $A_\Gamma$ and
$A_{\zeta,\Gamma}$ are related by an $\sO$-order
$\wA_\Gamma$ with the properties that $\wA_{\Gamma,k}=\wA_\Gamma/\pi\wA_\Gamma\cong A_\Gamma$ and $(\wA_\Gamma)_K\cong A_{\zeta,\Gamma}$.  The ``deforma tion theory" relating the representation theory
of these algebras (and their graded versions) provides
a major theme in earlier work, see \cite{PS10} and \cite{PS11}, and it is continued
in this paper. In the remainder of this subsection,
we will sketch a few details.

Given $\Gamma\subseteq X(T)_+$, let $(\Gmod)[\Gamma]$ be the full subcategory of $\Gmod$ generated by the irreducible modules $L(\gamma)$
 having highest weight $\gamma\in\Gamma$. In particular, if $\Gamma$ is a finite ideal in $X_{\text{\rm reg}}(T)_+$ (or, more generally, of $X(T)_+$), $(\Gmod)[\Gamma]$ is a highest weight category (in the sense of \cite{CPS-1}) with weight poset
 $\Gamma$.
 The category $(\Gmod)[\Gamma]$ identifies with the category $A_\Gamma$-mod of finite dimensional modules for a certain finite dimensional algebra $A_\Gamma$. Specifically, let  $I_\Gamma\trianglelefteq\Dist(G)$ be the annihilator ideal of all the modules  $V\in
 (\Gmod)[\Gamma]$. Then
 $(\Gmod)[\Gamma]\cong A_\Gamma$--mod, the category of finite dimensional $A_\Gamma$-modules, putting  $A_\Gamma:=\Dist(G)/I_\Gamma$.

 There is a similarly constructed algebra $A_{\zeta,\Gamma}$ over $K$. It has the property that $A_{\zeta,\Gamma}$--mod is isomorphic to the full subcategory of $U_\zeta$--mod generated by
 the irreducible modules $L_\zeta(\gamma)$, $\gamma\in \Gamma$. The algebras $A_\Gamma$ and
 $A_{\zeta,\Gamma}$ are related by an $\sO$-order $\wA_\Gamma$ which is defined to
 be the image of $\widetilde U_\zeta$ in $A_{\zeta,\Gamma}$. Necessarily, $\wA_\Gamma/\pi\wA_\Gamma
 \cong A_\Gamma$ and $(\wA_{\Gamma})_K\cong A_{\zeta,\Gamma}$.

 The algebras $A_\Gamma$, $A_{\zeta,\Gamma}$, and $\wA_\Gamma$ are all quasi-hereditary algebras (over
 $k$, $K$ and $\sO$, respectively) with poset $\Gamma$. For more details and properties of these algebras (as well as  of $\wfa$), see \cite{PS9}, \cite{PS10} and
 \cite{PS11}, as well as the earlier papers \cite{CPS-1} and \cite{CPS1a}.  If $\Gamma$ is an ideal in finite ideal $\Lambda$ in the poset $X_{\text{\rm reg}}(T)_+$, then there are surjective
homomorphisms $A_\Lambda\twoheadrightarrow A_\Gamma$, $A_{\zeta,\Lambda}\twoheadrightarrow A_{\zeta,\Gamma}$, and $\wA_\Lambda\twoheadrightarrow \wA_\Gamma$. This induce full embeddings
$i_*:A_\Gamma{\text{\rm --mod}}\to A_{\Lambda}{\text{\rm --mod}}$, etc. which preserve $\Ext^\bullet$-groups
(i.~e., they induce full embeddings at the level of the bounded derived categories).

 Let $\Lambda$ be any finite ideal of $p$-regular weights which contains all restricted $p$-regular weights. Assume also that if $\gamma$ is $p$-regular and restricted, then $2(p-1)\rho+w_0\gamma\in\Lambda$. Then the PIMs for $u'_\zeta$ are all $A_{\zeta,\Lambda}$-modules, so that the natural map
 $u'_{\zeta}\to A_{\zeta,\Lambda}$ is injective.  Similarly, $u'$ maps isomorphically onto its image
 in $A_\Lambda$. It follows that the (isomorphic) image $\wfa$ of $\widetilde u_\zeta'$ in $\wA_\Lambda$ is
pure in $\wA_\Lambda$.

Of course, any poset ideal $\Gamma\subseteq X_{\text{\rm reg}}(T)_+$ is contained in a poset $\Lambda$ as above, and this gives a natural map $\wfa\to\wA_\Gamma$.
By \cite[Cor. 3.9]{PS11}, $\wA_\Gamma$ is a $\wfa$-tight in the sense of (1.1)(5),  so that
$\gr\wA_\Gamma=\wgr \wA_\Gamma$; see \cite[Lem. 3.5]{PS11}. By \cite[Thm. 6.3]{PS10}, the
algebra $\gr\wA_\Gamma$ is quasi-hereditary over $\sO$. It has weight poset $\Gamma$ and standard
objects $\gr\widetilde \Delta(\gamma)=\wgr\widetilde\Delta(\gamma)$. Thus, $\wgr A_\Gamma=(\gr\wA_\Gamma)_k
=(\gr\wA_\Gamma)_k$ is also quasi-hereditary. It is important to observe that $\wgr A_\Gamma$
need {\it not} be the graded algebra $\gr A_\Gamma$ defined in (\ref{gradedzoo})(1). However, see
 Lemma \ref{ideals} below.

If $\Gamma\subseteq\Lambda$ are any finite ideals in $X_{\text{\rm reg}}(T)_+$, the surjective homomorphism
$\wA_\Lambda\twoheadrightarrow\wA_\Gamma$ above induces a surjective homomorphism
$\gr\wA_\Lambda\twoheadrightarrow\gr\wA_\Gamma$. In addition, the corresponding map
$\gr\wA_\Gamma{{\text{\rm --mod}}}\to\gr\wA_{\Lambda}{\text{\rm --mod}}$ induces a full embedding
on the corresponding derived category (and the resulting equality of $\Ext^\bullet$-groups, just as
described above in the ungraded cases. See \cite[Cor. 3.16]{PS9} for more discussion.

Another (more elementary) variant on the deformation theory described above  also will be useful,
replacing the triple $(A_{\zeta,\Gamma},\wA_\Gamma,A_\Gamma)$ by a triple $(A^\heartsuit_\Gamma,
\wA^\heartsuit_\Gamma, A_\Gamma)$.  In fact, define $A_{K,\Gamma}^\heartsuit:=\Dist_K(G)/I^\heartsuit_\Gamma$, where $I^\heartsuit_\Gamma$ is the
 annihilator in $\Dist_K(G)=U({\mathfrak g}_K)$ of the irreducible modules for ${\mathfrak g}_K$ having
 highest weights in $\Gamma$. Thus, $A^{K,\heartsuit}_\Gamma$ is a semisimple algebra over $K$ (in contrast
 to the fact that $A_{\zeta, \Gamma}$ is usually not semisimple). The image $\Dist_\sO(G)$ in $A^{K,\heartsuit}_{\Gamma,K}$ is denoted $\wA^\heartsuit_\Gamma$. It is an order over $\sO$ having the property that
 $(\wA^\heartsuit_\Gamma)/\pi\wA^\heartsuit_\Gamma\cong A_\Gamma$.

 The terminology of \S2.2 also applies in case of $\wA^\heartsuit_\Gamma$ and $A_{K,\Gamma}^\heartsuit$.
 For example, if $\wM$ is an $\wA^\heartsuit_\Gamma$-module, it is also a module for
 $\Dist_\sO(G)$-module, and then, using (\ref{twistedFrobenius}), as a module for $\wU_\zeta$, which is denoted
 $\wM^{[1]}$.

  \subsection{The Lusztig conjecture.}
For $\lambda\in X(T)_+$, the irreducible $U_\zeta$-module $L_\zeta(\lambda)$ has two important ``reductions
mod $p$" from admissible $\wU_\zeta$-lattices $\wrDelta(\lambda)$ and $\wrnabla(\lambda)$. Thus,
$\rDelta(\lambda):=\wrDelta(\lambda)/\pi\wrDelta(\lambda)$ and $\rnabla(\lambda):=\wrnabla(\lambda)/\pi
\wrnabla(\lambda)$. Both $\rDelta(\lambda)$ and $\rnabla(\lambda)$ are finite dimensional rational
$G$-modules. Rather than defining these modules explicitly, see \cite{Lin}, \cite{CPS7}, \cite{PS10}, and \cite{PS11}
for an extensive treatment. (Of course, there are other possible admissible lattices, leading to
other rational $G$-modules, but $\rDelta(\lambda)$ and $\rnabla(\lambda)$ will only be used in this
paper.) If $\lambda=\lambda_0+ p\lambda_1$, $\lambda_0\in X_1(T)$ and
$\lambda_1\in X(T)_+$, then
\begin{equation}\label{Lin}\rDelta(\lambda)\cong\rDelta(\lambda_0)\otimes\Delta(\lambda_1)^{[1]},\,\,
\rnabla(\lambda)\cong\rnabla(\lambda_0)\otimes\nabla(\lambda_1)^{[1]}.\end{equation}
See \cite[Thm. 2.7]{Lin} or \cite[Prop. 1.7]{CPS7}.

In addition, consider the rational $G$-modules $\Delta^p(\lambda):=L(\lambda_0)\otimes\Delta(\lambda_1)^{[1]}$ and $\nabla_p(\lambda):=L(\lambda_0)\otimes\nabla(\lambda_1)^{[1]}$. There are natural surjective
(resp., injective) module homomorphisms $\rDelta(\lambda)\twoheadrightarrow \Delta^p(\lambda)$ (resp.,
$\nabla_p(\lambda)\hookrightarrow \nabla_p(\lambda)$).

The following result indicates the importance of these modules to the representation theory of $G$, and,
in particular, to the validity of the Lusztig modular character formula---a specific formula conjectured to
hold for dominant weights in the Jantzen region. We do not repeat this formula here, but instead refer to \cite{L} and \cite{T}. Recall the Jantzen region is defined
\begin{equation}\label{Jregion}\Jan:=\{\lambda\in X(T)_+\,|\,(\lambda+\rho,\alpha_0^\vee)\leq p(p-h+2)\}\trianglelefteq X(T)_+.
\end{equation}

\begin{prop}\label{LC} If $p\geq 2h-3$ , then the validity of the Lusztig
modular character formula of $G$ for $p$-regular weights $\lambda\in \Jan$  is equivalent to requiring that
\begin{equation}\label{LCF} \rDelta(\lambda)\cong\Delta^p(\lambda),\quad\forall\lambda\in X_{\text{\rm reg}}(T)_+.\end{equation}
\end{prop}

  See
\cite[Cor. 2.5]{PS11} for the proof. It should be remarked that (\ref{LCF}) holds for all $p$-regular weights if and only if it holds for $p$-regular weights in $\Jan$.  (In addition, if (\ref{LCF}) holds then it also holds, for all $\lambda\in X(T)_+$, not
just at the $p$-regular weights by an elementary translation functor argument.)

The lemma below will be important. It is a consequence of some basic Kazhdan-Lusztig theory \cite{CPS1a}
and homological properties of the modules $\rDelta(\gamma)$ and $\rnabla(\gamma)$, $\gamma\in X_{\text{\rm reg}}(T)_+$.  Write $\gamma=w\cdot \gamma'$ where $w$ belongs to the affine Weyl group $W_p
=W\ltimes p{\mathbb Z} R$ of $G$, and $\gamma'$ belongs to the anti-dominant alcove $C_p^-$ containing $-2\rho$. Then
put $l(\gamma):=l(w)$ (Coxeter length).  It will be convenient to work inside the bounded derived category $\sD:=D^b(\Gmod)$
of $\Gmod$.  Let $[1]$ be shifting functor on $\sD$. If $m>0$, $[m]:=\underbrace{[1]\circ\cdots \circ [1]}_{m}$ (with the standard convention if $m<0$).
The category contains $\Gmod$ as a fully embedded subcategory. For $M,N\in\Gmod$,
$\Ext^n_G(M,N)=\Hom_\sD^n(M,N)=\Hom_\sD(M,N[n])$. We also need the full subcategories $\sE^R$
and $\sE^L$ of $\sD$.  For example, let $\sE_0^R$ be the full subcategory of $\sE$ consisting of
objects which are isomorphic to direct sums $\nabla(\gamma)[r]$, with $r\equiv l(\gamma)$ mod$\,\,2$. Having
defined $\sE^R_i$, define $\sE_{i+1}^R$ to be the full, strict subcategory of $\sD$ consisting of objects $X$
for which there is a distinguished triangle $Y\to X\to Z\to$, with $Y,Z\in \sE^R_i$. Let $\sE^R:=\bigcup_{i\geq 0}
\sE^R_i$. The dual subcategory $\sE^L$ is defined analogously, replacing the $\nabla(\gamma)$ by
$\Delta(\gamma)$.

\begin{lem}\label{ERlemma} Assume that $p\geq 2h-3$ and that condition (\ref{LCF}) holds. Let
$M,N\in \Gmod$. Assume that $M$ or $M[1]$ belongs to $\sE^R$, and that $N$ or $N[1]$ belongs
to $\sE^L$. (Thus, the composition factors of $M$ and $N$ all have $p$-regular highest weights.) For any $\lambda\in X_{\text{\rm reg}}(T)_+$, the natural maps
 \begin{equation}\label{ER}
 \begin{cases}(1)\quad\Ext^n_G(\rDelta(\lambda),M)\longrightarrow\Ext^n_G(\Delta(\lambda),M)\\
(2)\quad \Ext^n_G(N,\rnabla(\lambda))\longrightarrow\Ext^n_G(N,\nabla(\lambda))\end{cases}\end{equation}
 are surjective, for all $n\geq 0$.\end{lem}

 \begin{proof} First, consider statement (\ref{ER})(1).
  It is more convenient to prove (\ref{ER})(1)
 allowing $M$ to be an arbitrary object in $\sE^R$ or $\sE^R[1]$ (rather than just a rational $G$-module).
 The condition $p\geq 2h-3$ means that the restricted dominant weights are contained in the Jantzen
 region $\Jan$. Thus, since (\ref{LCF})  holds,  \cite[Thm. 6.8(a)]{CPS7} implies
 that $\rDelta(\lambda)[-l(\lambda)]\in{\mathscr E}^R$ and $\rnabla(\lambda)[-l(\lambda)]\in\sE^R$,
 for all $\lambda\in X_{\text{\rm reg}}(T)_+$.  Also, (\ref{ER})(1) holds trivially (using \cite[Lem. 2.2]{CPS7}) in case $M=\nabla(\xi)[r]$, for
 some integer $r$. Thus, (\ref{ER}(1)) is valid for $M$ or $M[1]$ in $\sE^R_1$. Now assume that $M$
 or $M[1]$ belongs to $\sE^R_{i+1}$ and the surjectivity of (\ref{ER})(1) holds with $M$ replaced by objects in $\sE^R_i$ or $\sE^R[1]$, $i\geq 1$.
 But there is a distinguished triangle $X\to M\to Y\to$ in which $X$ or $X[1]$ (resp., $Y$ or $Y[1]$)] belongs to $\sE^R_i$, so that surjectivity holds with $M$ replaced
 by $X$ or $Y$. Now a standard long exact sequence argument (see the proof of \cite[Thm. 4.3]{CPS1a})
 completes the argument for (\ref{ER})(1).

 The argument for the dual statement (\ref{ER}(2)) is similar
 and is left to the reader.
   \end{proof}

\subsection{Graded structures.}\label{gradedstructures}
  Suppose $B=\bigoplus_{n\geq 0}B_n$ is a positively graded finite dimensional algebra over a field. Let $M$ be in
the category $B$-grmod of $\mathbb Z$-graded $B$-modules. A resolution\footnote{\label{gradeconvention} In this resolution, the
graded $B$-module
$R_i$ has cohomological degree $-i$. For an integer $j$, $R_{i,j}$ denotes the $j$th grade of $R_i$; thus,
$R_i=\bigoplus R_{i,j}$.}
$$\cdots\longrightarrow R_2\longrightarrow R_1\longrightarrow R_0\longrightarrow M\longrightarrow 0$$
in $B$-grmod  is
called {\it $B$-linear} (or just linear, if $B$ is understood) if, for each nonnegative integer $n$,  the graded
$B$-module $R_n$  generated by its term $R_{n,n}$ in grade $n$. (In particular, $M$ is generated by its
grade 0-component $M_0$.)
Call the graded $B$-module $M$ {\it resolution linear}, or just {\it linear}, if it has a linear projective resolution.\footnote{It is possible to define other useful notions of linearity, e.~g., using graded Ext groups.
While such Ext considerations play a role in this paper, there is no need here for a special terminology for
them. 
  In the Koszul case, these notions all coincide. See the next footnote. } (For the structure of
projective objects in $B$-grmod, see Remark \ref{discussion} in \S8 (Appendix I).)  We remark that every such
linear projective resolution is automatically linear and thus uniquely determined. The algebra $B$ is a (finite dimensional) {\it Koszul
algebra} provided every irreducible $B$-module (regarded as a graded module concentrated in grade 0) is
resolution linear.  In this case, the subalgebra $B_0$ is necessarily semisimple.\footnote{When $B$ is a 
Koszul algebra, a graded $B$-module $M$ is resolution linear if and only if 
$\grExt^n_{B}(M,L\langle r\rangle)\not=0\implies n=r$, for all irreducible $B$-modules $L$ (concentrated
in grade $0$) and all $n\in\mathbb N$, $r\in\mathbb Z$. }

Finally, we mention that the  definitions above of linear resolutions and  modules easily carry
 over
to graded lattices over a graded order (such as $\wfa$ defined below). We leave further details to the reader.

If $p>h$, the sum $\wfa_K=u'_{\zeta,K}$ of the regular blocks in the small quantum group $u_\zeta$ is known to be Koszul \cite{AJS}. Let
$\wfa_K=\bigoplus_{i\geq 0}\wfa_{K,i}$ be the associated Koszul grading. By \cite[\S8]{PS10}, the $\sO$-algebra
$\wfa$ has a positive grading $\wfa=\bigoplus_{i\geq 0} \wfa_i$ such that, for any $i\geq 0$, $K\wfa_i=\wfa_{K,i}$.  Notice this implies that $\wfa\cong\wgr\wfa$.
Putting $\fa_i=k\otimes \wfa_i$,
\begin{equation}\label{grading} \fa=\bigoplus_{r\geq 0} \fa_i\end{equation}
provides a positive grading of the $p$-regular part $u'$ of the restricted enveloping algebra of $G$, for all
$p>h$. Also, $\fa\cong\wgr\fa$. In case (\ref{LCF}) holds for $G$ with $p>h$, then, by \cite{AJS}, the algebras $u'$ and $u'_\zeta$ are Koszul.

Given a finite  ideal $\Gamma$ in $X_{\text{\rm reg}}(T)_+$, any
projective $\wA=\wA_\Gamma$-module is $\wfa$-tight in the sense of (1.1)(5). In particular, $\wA$ is itself  $\wfa$-tight, as is
any projective $\wA$-lattice. See \cite[Cor. 3.9]{PS11}. If $\wX$ is a $\wA$-lattice, it is $\wfa$-tight if and only if it is $\wA$-tight, by
\cite[Cor. 3.8]{PS11}. 
(The quoted result, as stated, applies to $\wA_\Lambda$ for a poset $\Lambda$, which may be assumed
to contain $\Gamma$. In particular, $\wA$ is $\wA_\Lambda$-tight, and now the definitions show that $\wA$-tightness of $\wX$ is equivalent to $\wA_\Lambda$-tightness, and thus to $\wfa$-tightness.)
Thus, in this case, $\gr\wX=\wgr\wX$. {\it In particular, $\wgr\wA=\gr\wA$.} A similar argument, varying the poset, gives the tightness of $\wDelta(\gamma)$, $\gamma\in\Gamma$, and $\wgr\wDelta(\lambda)=
\gr\wDelta(\lambda)$.

\subsection{The Jantzen region.} The following result concerns the quasi-hereditary algebras
$\wgr A_\Gamma$.

\begin{lem}\label{ideals}Assume that $p\geq 2h-2$ is odd, and that (\ref{LCF}) holds. Let $\Gamma\trianglelefteq
\Jan$ consist of $p$-regular weights. Then (in the notation of (\ref{gradedzoo}))
$$\gr A_\Gamma=\gr_\fa A_\Gamma=\wgr A_\Gamma$$
and
$$\gr \Delta(\gamma)=\gr_\fa\Delta(\gamma)=\wgr \Delta(\gamma),\quad\forall\gamma\in\Gamma.$$
\end{lem}
\begin{proof}
 We first claim that
 \begin{equation} \label{equality} (\rad\fa) A_\Gamma =\rad A_\Gamma.\end{equation}
  Observe $A_\Gamma$-modules are the same as finite dimensional rational $G$-modules which have composition factors
 $L(\gamma)$, for $\gamma\in\Gamma$.  Thus, to prove (\ref{equality}), it's enough to
 show that, given $M$ in $\Gmod$, $M$ is completely reducible for $G$ if and only if it is completely
 reducible for $u'$.   Because irreducible $G$-modules are completely reducible
 for the restricted enveloping algebra $u$ (or equivalently, for $G_1$), the ``$\implies$" direction is obvious. Conversely, assume that $M$ is completely reducible
for $G_1$. Let $L:=\bigoplus L(\lambda_i)$ be the direct sum of the distinct irreducible $G$-modules having
restricted highest weights which, as $G_1$-modules, appear with nonzero
multiplicity in $M|_{G_1}$. Then
$$\Hom_{G_1}(L,M)\otimes L\overset\sim\longrightarrow M, \quad f\otimes x\mapsto f(x)$$
is an isomorphism of rational $G$-modules. Also, $\Hom_{G_1}(L,M)\cong N^{[1]}$, for a
rational $G$-module $N$. (See \cite[3.16(1)]{JanB}.) Thus, if $L(\tau)$ is a $G$-composition factor
of $N$, then $L(\lambda_i\otimes p\tau)$ is a composition factor of $M$. Thus, by hypothesis, $\lambda_i\otimes p\tau\in\Jan$.  A easy calculation shows that $(\tau+\rho,\alpha_0^\vee)\leq p$, i.~e., $\tau$ belongs to the
closure of the bottom $p$-alcove $C_p$ of $G$. Thus, $L(\tau)\cong\Delta(\tau)\cong\nabla(\tau)$, so that
$N$ is a completely reducible $G$-module because $\Ext^1_G(\Delta(\tau),\nabla(\sigma))=0$
for any $\tau,\sigma\in X(T)_+$.  This proves our claim.

By (\ref{equality}), $(\rad^n\fa) A_\Gamma =\rad^nA_\Gamma $, for all nonnegative integers $n$. This implies that
$\gr A_\Gamma =\gr_\fa A_\Gamma $. On the other hand, $\gr_\fa A_\Gamma =\wgr A_\Gamma $ by \cite[Cor. 5.6]{PS11}. This proves
the first assertion of the lemma. For the second assertion, $\rad^n\Delta(\gamma):=(\rad^n A_\Gamma)\Delta(\gamma)=(\rad^n\fa)\Delta(\gamma)$, so that $\gr\Delta(\gamma)=\gr_\fa\Delta(\gamma)=\wgr\Delta(\gamma)$, as before.
\end{proof}

\section{Q-Koszulity.} Q-Koszul  algebras are
introduced in Definition \ref{QKoszul} of this section. Let $\Lambda$ be an arbitrary finite ideal of $p$-regular
dominant weights, and let $B=\wgr A_\Lambda$ be the algebra defined in \S2.3. Then, under favorable circumstances---which, for the present,
means that $p\geq 2h-2$ is odd and the LCF condition (\ref{LCF}) holds---Theorem \ref{QKoszultheorem} states that $B$ is Q-Koszul.
Its proof is postponed to \S5.
Next,
Definition \ref{standard}  formulates the notion of a ``standard" Q-Koszul algebra, while Theorem \ref{standardtheorem} proves that the algebras $B$ are also standard Q-Koszul
algebras. When $\Lambda$ is contained in the Jantzen region, Corollary \ref{KLtheory} states that
$B$-mod has a graded Kazhdan-Lusztig theory (in the sense of \cite[\S3]{CPS1}). The proofs of these
last two results are presented at the end of \S6. Thus, when $p\geq 2h-2$ is odd, and when (\ref{LCF}) holds, the following picture emerges: the graded algebras $B$  which "model" the representation
theory of $G$ (on $p$-regular weights) are (standard) Koszul inside the Jantzen region $\Jan$, but then become (standard)
Q-Koszul as the weight poset $\Lambda$ expands outside $\Jan$. Ultimately, we expect something similar to hold for small primes, and also for
$p$-singular weights.

   Suppose that $B=\bigoplus_{n\geq 0}B_n$ is a positively graded quasi-hereditary algebra with poset $\Lambda$.    Since $B$ is quasi-hereditary, there is an increasing (``defining") sequence
$0=J_0\subseteq J_1\subseteq J_2\subseteq \cdots\subseteq J_n=B$ of idempotent ideals of
$B$ with the following property: for $1\leq i\leq n$, $J_i/J_{i-1}$ is a heredity ideal in the
algebra $B/J_{i-1}$.\footnote{See \cite{CPS-1}, \cite{CPS1a} and
\cite[\S C.1]{DDPW} for further details. Recall that an idempotent ideal $J$ in a finite dimensional algebra $A$ (over
the field $k$) is heredity provided that, writing $J=AeA$, for an idempotent $e$, the centralizer
algebra $eAe$ is semisimple and multiplication $Ae\otimes_{eAe}eA\longrightarrow AeA=J$
is an isomorphism (of vector spaces).} Because $B$ is graded, \cite[Prop. 4.2]{CPS1a} says that
the idempotent ideals $J_i$ are homogeneous; in fact, $J_i=Be_iB$ for some idempotent $e_i\in B_0$.

Each standard module
$\Delta(\lambda)$, $\lambda\in\Lambda$, has a natural positive grading, described as follows.  $\Delta(\lambda)$ is a projective
(ungraded) module for an appropriate quotient algebra $B/J_i$---it identifies with the projective cover
of $L(\lambda)$ in $B/J_{i-1}$-mod. By the previous paragraph, $B/J_i$ is also a graded quasi-hereditary algebra. Therefore, $\Delta(\lambda)$ is the projective cover in the
$B/J_{i-1}$-grmod of the irreducible module $L(\lambda)$ (viewed as a graded $B/J_{i-1}$-module having pure
grade 0).  See Remark \ref{discussion} in \S8 (Appendix I) for more discussion of PIMs in $B$-grmod.

We have the following elementary result. See also \cite[Cor. 3.2]{PS11}.

\begin{prop}\label{B0}
(a) Suppose $B=\bigoplus_{n\geq 0}B_n$ is a positively graded quasi-hereditary algebra with poset $\Lambda$. Then
the subalgebra $B_0$ is quasi-hereditary with poset $\Lambda$.

(b)  In the special case that $\Lambda$ is a finite ideal of $p$-regular dominant weights, put $B:=\wgr A_\Lambda$. Then the modules $\rDelta(\lambda)$ (resp., $\rnabla(\lambda)$), $\lambda\in\Lambda$, are the standard (resp., costandard) modules
for the quasi-hereditary algebra $B_0$. In particular,  the rational $G$-modules $\rDelta(\lambda)$ and $\rnabla(\lambda)$  are naturally modules for all three algebra $B$, $B_0$ and $A_\Lambda$, all acting through
the common quotient algebra $B_0$. \end{prop}

\begin{proof} Let $0=J_0\subseteq J_1\subseteq J_2\subseteq\cdots \subseteq J_n=B$ be a
defining sequence of idempotent ideals in $B$ as described above.   Each $J_i=Be_iB$, for an idempotent
$e_i\in B_0$. Necessarily (by the axioms for a quasi-hereditary algebra), $e_1Be_1$ is a semisimple algebra, so that necessarily $e_1B_0e_1=e_1Be_1$
is semisimple. In addition, multiplication $Be\otimes_{e_1Be_1}eB\longrightarrow BeB=J_1$ is an
isomorphism of $k$-vector spaces. Taking the gradings into account, it follows that multiplication
$B_0e_1\otimes_{e_1B_0e_1}e_1B_0\to e_1B_0e_1$ is an isomorphism. Therefore, $J_{1,0}=
B_0e_1B_0$ is a heredity ideal in $B_0$. Continuing, we find that $0\subseteq J_{1,0}\subseteq J_{2,0}\subseteq
\cdots \subseteq J_{n,0}$ is a defining sequence of ideals in $B_0$. It follows that $B_0$ is
quasi-hereditary with poset $\Lambda$, as required for (a).

Finally, to see (b), apply \cite[Cor. 3.2]{PS11}, with standard (resp., costandard) modules the
$\rDelta(\lambda)$ (resp., $\rnabla(\lambda)$), $\lambda\in\Lambda$.  Notice that, for any $n>0$,
$(\wrad^n\wfa)\rDelta(\lambda)=(\wrad^n\wfa)\rnabla=0$ because $\rDelta(\lambda)$ and $\rnabla(\lambda)$
are obtained by reductions mod $p$ of lattices in an irreducible $U_\zeta$-module. Hence, $\rDelta(\lambda)$
and $\rnabla(\lambda)$ are indeed $B_0=\wgr A/\wrad A$-modules. \end{proof}

\begin{rem}\label{abovediscussion}The above discussion extends to the $\sO$-algebras $\wA_\Lambda$. In fact, since $\sO$ is assumed to be complete, $\wA:=\wA_\Lambda$ is a semi-perfect
algebra (see \cite{CPS1a}).  In view of \cite[Prop. 4.2]{CPS1a}, the idempotent ideals
$\widetilde J_i$ making up a defining sequence of $\wA_\Lambda$ are all homogeneous and have the
form $\widetilde J_i=\wA e_i\wA$, for some idempotent $e_i$. The argument is then completed
as before. In particular, we note that $\wrDelta(\lambda)$ and $\wrnabla(\lambda)$ are modules for
$\wA_\Lambda$, $\wgr\wA_\Lambda$, $(\wgr\wA_\Lambda)_0$, with the first two algebras acting
through their common quotient algebra $(\wgr\wA_\Lambda)_0$.\end{rem}

%A finite dimensional graded algebra $B=\bigoplus_{n\geq 0}B_n$ is Koszul provided that
%\smallskip
%\begin{enumerate}
%\item[(1)] the algebra $B_0$ is semisimple; and
%\item[(2)]  if $L,L'$ are irreducible $B$-modules given pure grade 0, then
% $$\grExt^n_B(L,L'\langle r\rangle )\not=0\implies
%r=n.$$
 %\end{enumerate}\smallskip

We propose the following generalization of a Koszul algebra.

\begin{defn}\label{QKoszul} A finite dimensional, positively graded algebra $B=\bigoplus_{n\geq 0}B_n$ is called a {\it Q-Koszul algebra} provided the following conditions hold:

\begin{enumerate}
\item[(i)] the subalgebra $B_0$ is quasi-hereditary, with poset $\Lambda$ and standard (resp., costandard)
modules denoted $\Delta^0(\lambda)$ (resp., $\nabla_0(\lambda)$), $\lambda\in\Lambda$; and

\item[(ii)] if $\Delta^0(\lambda)$ and $\nabla_0(\lambda)$ are given pure grade 0 as graded $B$-modules
(through the homomorphism $B\twoheadrightarrow B/B_{\geq 1}\cong B_0$), then
$$\grExt^n_B(\Delta^0(\lambda),\nabla_0(\mu)\langle r\rangle )\not=0\implies n=r, \quad\forall\lambda,\mu\in\Lambda, n\in\mathbb N, r\in\mathbb Z.$$
\end{enumerate}\end{defn}

In the above definition, the algebra $B$ can be taken over any field, not necessarily our algebraically closed
field $k$ of positive characteristic $p$.

\begin{rems}\label{trivial} (a)
 A similar
generalization---in the abstract---of Koszul algebras, using ``tilting modules"  has been proposed
by  Madsen \cite{Madsen}.

(b)  Koszul
algebras and quasi-hereditary algebras provide rather trivial examples of Q-Koszul algebras. In the case
in which
$B$ is Koszul, the subalgebra $B_0$ is semisimple and hence it is quasi-hereditary. In this situation, $\Delta^0(\lambda)\cong\nabla_0(\lambda)$, $\lambda\in\Lambda$, are irreducible. View them as graded $B$-modules having pure
grade 0, condition (ii) is automatic from the definition of a Koszul algebra. Thus, $B$ is Q-Koszul. On the
other hand, suppose that $B$ is an (ungraded) quasi-hereditary algebra. View $B$ as positively graded by setting $B_0:=B$. Then
$B$ is Q-Koszul using the well-known fact that $\dim\Ext^n_B(\Delta(\lambda),\nabla(\mu))=\delta_{\lambda,\mu}\delta_{n,0}$ \cite[Lem. 2.2]{CPS1}.
\end{rems}

Now return to the group $G$. The next result shows that there are more interesting examples of Q-Koszul algebras than those considered in Remark \ref{trivial}(b).  The proof will be given in \S5,
immediately after the proof of Theorem \ref{lasttheorem}.

\begin{thm} \label{QKoszultheorem} Assume that $p\geq 2h-2$ is odd, and that condition (\ref{LCF})
holds.   Let $\Lambda$ be a finite ideal of $p$-regular dominant weights and form the graded algebra $B:=\wgr A_\Lambda$. Then $B$ is a Q-Koszul algebra with
poset $\Lambda$, setting $\Delta^0(\lambda)=\rDelta(\lambda)$ and $\nabla_0(\lambda)=\rnabla(\lambda)$,
$\lambda\in\Lambda$.\end{thm}

Finally, there is the following notion of a standard Q-Koszul algebra. It is modeled on the notion of a
standard Koszul algebras as used by Mazorchuk \cite{Maz}.\footnote{Mazorchuk quotes a paper
\cite{ADR} for the name standard Koszul, though the notion is not quite the same. In any case, the notion (but not the name)  goes back to earlier work of Irving \cite{Irving}.}

\begin{defn}\label{standard} A positively gradded algebra $B=\bigoplus_{n\geq 0}B_n$ is called a {\it standard Q-Koszul algebra} provided it is Q-Koszul\footnote{It seems likely that the requirement that $B$ be Q-Koszul is  already implied by conditions (i) and (ii) and thus is redundant. We intend to discuss this issue further elsewhere.} the following conditions are satisfied:
\begin{enumerate}
\item[(i)] $B$ graded quasi-hereditary algebra with weight poset $\Lambda$, and with standard (resp., costandard, irreducible) modules $\Delta^B(\lambda)$
(resp., $\nabla_B(\lambda)$, $L_B(\lambda)$), for $\lambda\in\Lambda$; and
\item[(ii)]  given
 $\lambda,\mu\in\Lambda$, and positive integers $r,n$,
$$\begin{cases}\grExt^n_B(\Delta^B(\lambda),\nabla_0(\mu)\langle r\rangle )\not=0\implies n=r;\\
\grExt^n_B(\Delta^0(\mu),\nabla_B(\lambda)\langle r\rangle )\not=0\implies n=r.\end{cases}
$$
\end{enumerate}
In (ii), $\Delta^0(\mu)$ (resp., $\nabla_0(\mu)$), $\lambda,\mu\in\Lambda$, are the standard (resp.,
costandard) modules for the quasi-hereditary algebra $B_0$. They are viewed as graded $B$-modules
(concentrated in grade 0) through the homomorphism $B\twoheadrightarrow B/B_{\geq 1}\cong B_0$.
\end{defn}

  The complete proof of the theorem below is postponed to \S6. The theorem requires that there is, by \cite[8.4]{PS9},  a natural duality
$\mathfrak d$ on the module categories $\wgr A_\Lambda$-mod and $\wgr A_\Lambda$-grmod. It arises from an anti-automorphism
of the order $\wA_\Lambda$ and so induces an anti-automorphism on $A_\Lambda$ and a graded
anti-automorphism on $\wgr A_\Lambda$.
Thus, it induces a duality on $A_\Lambda$-mod, $\wgr A_\Lambda$-mod, and $\wgr A_\Lambda$-grmod. This duality fixes irreducible modules and interchanges standard and costandard modules.

\begin{thm} \label{standardtheorem} Assume that $p\geq 2h-2$ is odd, and that (\ref{LCF}) holds. Let $\Lambda$ be a finite ideal of $p$-regular dominant weights and form the graded algebra $B:=\wgr A_\Lambda$. Then $B$ is a standard Q-Koszul algebra (in the sense of Definition \ref{standard}) with
poset $\Lambda$, setting
$$\begin{cases} \Delta^B(\lambda)=\wgr \Delta(\lambda),\\
 \Delta^0(\lambda)=\rDelta(\lambda), \\
 \nabla_B(\lambda)={\mathfrak d}\Delta^B(\lambda),\\
 \nabla_0(\lambda)=\rnabla(\lambda),\end{cases}
 $$
for $\lambda\in\Lambda$. \end{thm}

A graded quasi-hereditary algebra $B$ has, by definition, a graded Kazhdan-Lusztig theory
provided there is a length function $l:\Lambda\to{\mathbb Z}$ such that, for $\lambda,\mu\in\Lambda$,
$r,n\in\mathbb Z$, the non-vanishing of either $\grExt^n_B(\Delta^B(\lambda),L_B(\mu)\langle r\rangle )$ or
of $\grExt^n_B(L_B(\mu)\langle r\rangle ,\nabla_B(\lambda))$ implies that $n=r\equiv l(\lambda)-l(\mu)$ mod $2$.
See \cite[\S3]{CPS1a} and \cite[\S2.1]{CPS5}.

The usual length function $l$ on the (affine) Coxeter group $W_p$ of $G$ leads to a length function $l:X_{\text{\rm reg}}(T)_+\to\mathbb N$ as follows. For a $p$-regular
dominant weight $\lambda$, write $\lambda=w\cdot\lambda^-$, where $\lambda^-\in C^-_p$ (the
unique alcove containing $-2\rho$) and $w\in W_p$. Then put
$l(\lambda):=l(w)$.

The following corollary was promised in \cite[Rem. 10.7(a)]{PS9}. The proof is
postponed to \S6.

 \begin{cor}\label{KLtheory} Assume that $p\geq 2h-2$ is odd, and that (\ref{LCF}) holds (for $p$ and $G$). Let $\Lambda$ is a finite ideal of $p$-regular dominant weights contained in $\Jan$. Then  $\wgr \Delta(\lambda)$ is a linear module over $\wgr A$. Also, the graded quasi-hereditary algebra $\wgr A$-mod has a graded Kazhdan-Lusztig theory. In particular, $\wgr A$ is Koszul. \end{cor}

\section{$(\Gamma,\fa)$-Resolutions.}  This section begins the study of resolutions necessary for most of the main
results of this paper. The detailed information obtained on filtrations of the syzygies in these resolutions are important in their own right. 

We continue the notation of \S\S1,2. We will not quote any results
from \S3.  Let $\Gamma$ denote a finite  ideal in $X_{\text{\rm reg}}(T)_+$ and let $A=A_\Gamma$. The reader should keep in mind that $A$-mod consists of finite dimensional rational $G$-modules whose composition factors have the form $L(\gamma)$ for $\gamma\in\Gamma$. Let $M$ be a graded $\wgr A$-module.The main result of this section, given in Theorem \ref{maintheorem}, constructs
a key specific resolution $\Xi_\bullet\twoheadrightarrow M$.
It is required that the (\ref{LCF}) condition holds, that $M|_\fa$ be linear in the sense of \S2.5, and that each graded component
$M_s$, when regarded as a $(\wgr A)_0$-module has a $\rDelta$-filtration. This resolution will play a central
role in \S\S5,6 in, for example, explaining the structure of rational $G$-modules of the form $\Ext^n_\fa(\rDelta(\lambda),\rnabla(\mu))=\Ext^n_{G_1}(\rDelta(\lambda),\rnabla(\mu))$ (resp., $\Ext^n_\fa(\Delta(\lambda),\nabla(\mu))=\Ext^n_{G_1}(\Delta(\lambda),\nabla(\mu))$  for $p$-regular dominant weights
$\lambda,\mu$; see Theorem \ref{nextmainresult} (resp., Theorem \ref{Jantzentheorem2})).

\begin{defn}\label{basicdefn}
Let $M$ be a graded $\wgr A_\Gamma$-module. A $(\Gamma,\fa)$-{\it projective resolution} of $M$ is an exact complex
\begin{equation}\label{Gammaresolution}
\cdots \longrightarrow \Xi_i\longrightarrow \cdots \longrightarrow \Xi_1\longrightarrow \Xi_0\to M\to 0
\end{equation}
of graded vector spaces and graded maps with the following properties:

\begin{itemize}
\item[(i)] there is an increasing chain $\Gamma=\Gamma_{-1}\subseteq\Gamma_0\subseteq \Gamma_1
\subseteq\cdots$ of finite ideals in $X_{\text{\rm reg}}(T)_+$, such that, for $i\geq 0$, $\Xi_i\in \wgr A_{\Gamma_i}$--mod;

\item[(ii)]  the maps $\Xi_i\to \Xi_{i-1}$ are morphisms in the category $\wgr A_{\Gamma_i}$-grmod. (Set $\Xi_{-1}:= M$.)  In this statement, the graded $\wgr A_{\Gamma_{i-1}}$-module $\Xi_{i-1}$
is regarded as a graded $\wgr A_{\Gamma_i}$-module through the algebra surjection $\wgr A_{\Gamma_i}
\twoheadrightarrow \wgr A_{\Gamma_{i-1}}$. See \cite[Rem. 3.8]{PS10}.

\item[(iii)] for $i\geq 0$, the $\wgr A_{\Gamma_i}$-module $\Xi_i$ has a graded filtration with sections of
the form $\wgr P^\sharp(\gamma)\langle j \rangle$, $\gamma\in \Gamma_{i-1}$, $j\in\mathbb N$.  (The module $P^\sharp(\gamma)$ is
defined in (\ref{generalizedQsandPs}).)

\end{itemize}

Similarly, at level of orders and lattices over $\sO$, there is an analogous notion of a $(\Gamma,\wfa)$-projective resolution
$\wXi_\bullet\twoheadrightarrow \wM$
of a $\wgr \wA_\Gamma$-lattice $\wM$. Setting $\Xi_i:=k\otimes_\sO\wXi_i$ and $M:=k\otimes_\sO\wM$,
it follows that $\Xi_i\twoheadrightarrow M$ is a $(\Gamma,\fa)$-projective resolution of $M$. (Use
$\gr\wP^\sharp(\gamma)=\wgr\wP^\sharp(\gamma)$ in place of $\wgr P^\sharp(\gamma)$.)
\end{defn}

Continue in the context of Defn. \ref{basicdefn}. Suppose that $j>0$, and let $\Omega_j:=\ker(\Xi_{j-1}\to \Xi_{j-2})$. Recall that $\Xi_{-1}:=M$. Define the  $j$-truncated complex
\begin{equation}\label{truncated}
\Xi_\bullet^{\dagger}=\Xi_\bullet^{\dagger_j}: \,\,0\to\Omega_{j}\to \Xi_{j-1}\to \cdots\to \Xi_0\to 0\end{equation}
 in the category $\wgr A_{\Gamma_j}$-grmod.  Observe that $(\Xi^\dagger)_{j}=\Omega_{j}$ and
 $\Xi^\dagger_{-1}=0$. By
 definition,  $\Xi^\dagger_\bullet\twoheadrightarrow M$ is a resolution of $M$. The syzygies $\Omega_j$ will
 play a role below. Similar considerations apply in the integral case (over $\sO$).

Now assume that $p\geq 2h-2$ is odd, and that the LCF condition (\ref{LCF}) holds. In particular, $\fa=u'$ (the direct sum of the regular
blocks of the universal enveloping algebra $u$ of $G$) is a Koszul algebra. A $(\Gamma,\fa)$-projective resolution of $M$ gives a
resolution of $M|_\fa$
by graded and projective $\fa$-modules $\Xi_i|_\fa$. In fact, $\Xi_i|_\fa$ has, by definition, a $\wgr A_{\Gamma_i}$-filtration with sections $\wgr P^\sharp(\gamma)\langle j\rangle $ and each $\wgr P^\sharp(\gamma)\langle j\rangle $ is
a projective graded $\fa$-module. This resolution is $\fa$-linear (in the sense of \S2.5)  if and only if
$j=i$, for all the $\wgr A_{\Gamma_i}$-modules $P^\sharp(\gamma)\langle j\rangle $ which appear as sections (and hence
as $\fa$-summands) of $\Xi_i$ in condition (iii) above.

We will see in \S\S 5, 6 that these resolutions can be used to compute, among other things,  the spaces $\grExt_{\wgr A}^m(M,X)$
and $\Ext^m_{\wgr A}(M,X)$ with  $X=\rnabla(\gamma)$, with $M$ as above. Theorem \ref{maintheorem}
below constructs these resolutions for suitable $M$. Integral versions are also obtained.
In addition, the theorem
 shows that, in the presence of (\ref{LCF}),  {\it 
 the syzygy modules in suitable resolutions of the modules $\rDelta(\lambda)$ (for a $p$-regular dominant weight $\lambda$) have $\rDelta$-filtrations.}  Once Theorem \ref{DeltaKoszul} is established\footnote{Part (a) of Theorem \ref{DeltaKoszul} does not assume the LCF
and depends only on results from \cite{PS10}, while part (b) is derived from Theorem \ref{maintheorem} applied to
$\rDelta(\lambda)$. }, a similar result by be deduced from Theorem \ref{maintheorem} for resolutions of $\wgr\Delta(\lambda)$, expanding a main theme
of \cite{PS11}, which provided a $\rDelta$-filtration of Weyl modules.

\begin{thm}\label{maintheorem}Assume that $p\geq 2h-2$ is odd, and that the LCF condition (\ref{LCF}) holds. Let $\Gamma$ be any finite  ideal in the set
$X_{\text{\rm reg}}(T)_+$ of $p$-regular
dominant weights.  Let $A=A_\Gamma$ and $\wA=\wA_\Gamma$.

(a) Assume that $M$ is a graded $\wgr A$-module such that each grade $M_s$ has a $\rDelta$-filtration.
Assume that $M|_\fa$ is a linear module.  There exists a resolution (\ref{Gammaresolution}) of $M$ which is both $\fa$-linear and $(\Gamma,\fa)$-projective such
that, for $i\geq 0$,   $\Xi_i$ and $\Omega_{i+1}:=\ker(\Xi_i\to \Xi_{i-1})$  have a $\rDelta$-filtration, grade by grade.

(b) Assume that $\wM$ is a graded $\wgr \wA$-module such that each grade $\wM_s$ has a $\wrDelta$-filtration.
Assume that $\wM|_{\wfa}$ is a linear module. There exists an $\wfa$-linear $(\Gamma, \wfa)$-projective resolution of $\wM$, analogous to (\ref{Gammaresolution}),
 such
that, for $i\geq 0$,   $\wXi_i$ and $\wOmega_{i+1}:=\ker(\wXi_i\to \wXi_{i-1})$  have a $\wrDelta$-filtration, grade by grade.\end{thm}

Before proving the theorem, some further notation and a preliminary lemma are required.

For a finite  ideal $\Gamma$ in $X_{\text{\rm reg}}(T)_+$, let $r:=r(\Gamma)$ be the minimal positive
integer such that $\Gamma\subseteq X_r(T)$.   For a positive integer $r$, put
\begin{equation}\label{Lambda}\Lambda_r:=\{\lambda\in X_{\text{\rm reg}}(T)_+\,|\,(\lambda,\alpha_0^\vee)<2p^r(h-1) \}.\end{equation}
Thus,  $\Lambda_r$ in an ideal in the poset of $p$-regular weights. If $r\geq r(\Gamma)$,  then $\Gamma$
is an ideal in $\Lambda_r$.  In addition, if $\gamma\in\Gamma$,  the $G_rT$-projective cover $\widehat Q_r(\gamma)$ of the irreducible $G_rT$-module $\widehat L_r(\gamma)$ of highest weight $\gamma$ has a unique $G$-module
structure with $G$-composition factors
$L(\tau)$, $\tau\in\Lambda_r$. In \cite{JanB}, this $G$-module is denoted by the same symbol
$\widehat Q_r(\gamma)$, but we write it as $P_r(\gamma)$. Given $\gamma\in X_1(T)$, $P_1(\gamma)
=P^\sharp(\gamma)$ in the notation of (2.1.1).

Let $A:=A_{\Lambda_r}$. By \cite[p. 333]{JanB}, $P_r(\gamma)$ is the projective cover of $L(\gamma)$ in the ``$p^r$-bounded category" $A$-mod of rational $G$-modules having composition factors of highest weights in $\Lambda_r$.

Now pass to orders, and let $\wA:=\wA_{\Lambda_r}$, where $r\geq r(\Gamma)$ as before.  Given $\gamma\in\Gamma$, by \cite[Thm. \ref{abovediscussion}, Prop. 2.3 \& p. 159]{DS}, we can lift the projective $A$-module $P_r(\gamma)$ to an $\wA$-lattice
$\wP_r(\gamma)$.  Moreover, any such lifting is projective and unique.

Write $\gamma=\gamma_0+p\gamma_1\in \Gamma$, where $\gamma_0\in X_1(T)$  and $\gamma_1\in X(T)_+$. Then $P_1(\gamma_0)\in A_{\Lambda_1}$-mod lifts to a projective module for $\wA_{\Lambda_1}$ and, thus, to a $\wU_\zeta$-lattice $\wP_1(\gamma_0)$.
The
projective module
$P_{r-1}(\gamma_1)\in A_{\Lambda_{r-1}}$ lifts to a $\Dist_\sO(G)$-lattice $\wP^\heartsuit_{r-1}(\gamma_1)$.
Therefore, pulling back through the Frobenius $\wF$ in (2.2.2), we  obtain the $\wU_\zeta$-lattice $(\wP^\heartsuit_{r-1}(\gamma_1))^{[1]}$, denoted $\wP^\heartsuit_{r-1}(\lambda)^{[1]} $ or simply $\wP_{r-1}(\lambda)^{[1]}$ if it is convenient.
 There is a tensor product decomposition
\begin{equation}\label{heart}
\wP_r(\gamma)\cong\wP_1(\gamma_0)\otimes \wP^\heartsuit_{r-1}(\gamma_1)^{[1]}.\end{equation}
(The reductions mod $\pi$ are isomorphic as rational $G$-modules, so they are integrally isomorphic.) The Hopf algebra
structure on $\wU_\zeta$ is required to view (\ref{heart}) as a $\wU_\zeta$-module.

The proof of the following lemma uses the fact that if $\wX$ is a lattice for an integral
quasi-hereditary algebra $\wB$ with the property that $\wX_k$ has a $\Delta$-filtration for the
quasi-hereditary algebra $B=\wB_k$, then $\wX$ has a $\wDelta$-filtration.  This follows immediately
from \cite[Prop. 6.1]{PS11} and a standard Nakayama's lemma argument.  The integral quasi-hereditary algebra will
be $(\wgr\wA)_0$.

\begin{lem}\label{basic} Assume that $p\geq 2h-2$ is odd, and that (\ref{LCF}) holds.

Let $\Gamma\subset X(T)_+$ be a finite  ideal in the poset of $p$-regular weights. Let $\wB:=\wA_\Gamma$. Suppose that $\widetilde\Omega:=\bigoplus_s\wOmega_s$ is a graded $\wgr\wB$-lattice generated in grade $m$, for some integer $m$. View $\wOmega_m$ as a graded $\wgr\wB$-module concentrated in grade $m$, and
assume that $\widetilde\Omega_m(-m)$ has a $\wrDelta$-filtration. (Any $\wrDelta(\mu)$, $\mu\in\Gamma$, may viewed as a $\wgr\wB$-module concentrated in grade 0; see Remark 3.2).

Let $\Lambda=\Lambda_r$ with $r\geq r(\Gamma)$ and set $\wA:=\wA_\Lambda$. The following statements hold.

(a)  If
\begin{equation}\label{syz}\wOmega':=\ker\left(\wgr\wA\otimes_{(\wgr \wA)_0}\wOmega_m\twoheadrightarrow\wOmega\right),\end{equation}
then $\wOmega'$ is a graded $\wgr\wA$-lattice vanishing in grades $\leq m$.
All composition factors of $\wOmega, \wOmega'$ and $\wgr\wA\otimes_{(\wgr\wA)_0}\wOmega_m$
have highest weights in $\Lambda$.

(b) Moreover,
  $\wgr\wA\otimes_{(\wgr \wA)_0}\wOmega_m$ has a graded filtration with sections of the form $\wgr\wP^\sharp(\lambda)\langle m\rangle $, $\lambda\in\Gamma$. Any such  filtration of $\wgr\wA\otimes_{(\wgr\wA)_0}\wOmega_m$ induces
a filtration of $\wOmega_m$ by modules $\wrDelta(\lambda)\langle m\rangle \cong(\wgr \wP^\sharp(\lambda)\langle m\rangle )_m$. All
filtrations of $\wOmega_m$ with sections $\wrDelta(\lambda)\langle m\rangle $, $\lambda\in\Gamma$, arise this way.

(c) Suppose, for all $s\in\mathbb Z$, that  $\wOmega_s$ has a $\wrDelta$-filtration. Then $\wOmega'_s$ also has a $\wrDelta$-filtration.
 \end{lem}

\begin{proof} We begin by proving (b). Let $\gamma=\gamma_0+p\gamma_1\in \Gamma$, where $\gamma_0
\in X_1(T)$ and $\gamma_1\in X_{r-1}(T)$.
Form the exact sequences
$$\begin{cases} (1)\,\,
0\to\wJ^{[1]}\to \wP_{r-1}^\heartsuit(\gamma_1)^{[1]}\to\wDelta(\gamma_1)^{[1]}\to 0,\\
(2)\,\,0\to \wP_1(\gamma_0)\otimes_\sO \wJ^{[1]}\to\wP_r(\gamma)\to \wP^\sharp(\gamma)\to 0\end{cases}
$$
of $\wU_\zeta$-modules. In (1),
$\wJ^{[1]}$ is defined as the kernel of the natural map $\wP^\heartsuit_{r-1}(\gamma_1)^{[1]}\twoheadrightarrow\wDelta(\gamma_1)^{[1]}$.  Then (2)
is a sequence of $\wA$-modules,
obtained (using the Hopf algebra $\wU_\zeta$)  by applying $\wDelta(\gamma_0)\otimes_\sO -$ to (1).   Also, (2) is $\wfa$-split,  since $\wP^\sharp(\gamma)$ is $\wfa$-projective. Hence, (2) remains an exact sequence
in the category $\wgr\wA$-grmod after $\wgr$ is applied.
Observe from (the dual version of) \cite[Lem. 4.1(c)]{PS11}, which uses (\ref{LCF}), that $(\wgr \wP^\sharp(\gamma))_0\cong\rDelta(\gamma)$ as a $(\wgr\wA)_0$-module.
For convenience, put $\wN:=\wP_1(\gamma_0)\otimes_\sO \wJ^{[1]}$, and form the following commutative diagram:
\begin{equation}\label{train}
\begin{CD}
0 @>>> \wgr\wA\otimes(\wgr\wN)_0 @>>>\wgr\wA\otimes(\wgr
\wP_r(\gamma))_0 @>>>\wgr\wA\otimes\wrDelta(\gamma) @>>> 0\\
@. @VVV @VVV @VVV @.\\
0 @>>> \wgr\wN @>>>\wgr\wP_r(\gamma) @>>>\wgr\wP^\sharp(\gamma) @>>> 0\end{CD}
\end{equation}
where $\otimes=\otimes_{(\wgr \wA)_0}$ in the first row. As noted above, the second row is exact, and
we also claim that
the first row is also exact. This will follow provided that
\begin{equation}\label{tor} \Tor_1^{(\wgr\wA)_0}(\wgr\wA,\wrDelta(\gamma))=0.\end{equation}
First, by \cite[Thm. 6.3 ]{PS10}, $\wgr\wA$ is a quasi-hereditary algebra over $\sO$  with poset $\Lambda$
and with standard {\it right}
modules denoted $\wgr\wDelta(\gamma)^\circ$, $\gamma\in\Gamma$. Now we work with the quasi-hereditary algebra $(\wgr\wA)_0$ which
has right standard (resp., costandard) modules $\wrDelta(\tau)^\circ$ (resp., $\wrnabla(\tau)^\circ$), $\tau\in\Gamma$. By \cite[Thm. 5.1]{PS11}, $\wgr\Delta(\gamma)^\circ$ has a $(\wrDelta)^\circ$-filtration. Therefore, $$\Ext^1_{(\wgr \wA)_0}(\wgr\wDelta(\gamma)^\circ,\rnabla(\mu)^\circ)\cong\Ext^1_{(\wgr A)_0}(\wgr\Delta(\gamma)^\circ,\rnabla(\mu)) =0, \quad\forall\mu\in\Gamma.$$
A standard Nakayama's lemma argument gives that
$\Ext^1_{(\wgr \wA)_0}(\wgr\wrDelta(\gamma)^\circ,\wrnabla(\mu)^\circ)=0,$
 which
means that $\wgr\wA$, viewed as a right $(\wgr\wA)_0$, has a $(\wrDelta)^\circ$-filtration by \cite[Prop. 6.1]{PS11}.
Now (\ref{tor}) follows from Proposition \ref{appendixprop} below (applied to the quasi-hereditary
algebra $(\wgr A)_0=((\wgr \wA)_0)_k$).

 The middle and left vertical maps in (\ref{train}) are both surjective maps, since $\wP_1(\gamma_0)$ and
 $\wP_{r-1}(\gamma)$ are projective $\wfa$-modules. %Therefore, they are both tight for $\wA$ by
% \cite[Cor. 3.8]{PS11}.
 Thus, the right hand vertical map is surjective.

 Next, the middle
vertical map in (\ref{train}) is, in fact, an isomorphism, since $\wP_r(\gamma)$ is $\wA$-projective.
The snake lemma now implies that the two remaining vertical maps are injective, hence they are also isomorphisms. (In particular,  we record
the isomorphism
\begin{equation}\label{rail} \wgr\wA\otimes_{(\wgr \wA)_0}\wrDelta(\gamma)\overset\sim\longrightarrow \wgr\wP^\sharp(\gamma)
\end{equation}
which will be used later.)

Consider the surjection $\wgr\wA\otimes_{(\wgr\wA)_0}\widetilde\Omega_m\twoheadrightarrow\widetilde\Omega$ of graded modules.
A $\wrDelta$-filtration of  $\widetilde\Omega_m(-m)$
gives a filtration of $\wgr\wA\otimes_{(\wgr\wA)_0}\widetilde\Omega_m$ with sections $\wgr\wP^\sharp(\gamma)\langle m\rangle $, using Proposition \ref{appendixprop} again and the right hand vertical isomorphism above. This filtration
is a graded filtration of a graded module. Conversely, any graded filtration of $\wgr\wA\otimes_{(\wgr\wA)_0}\widetilde\Omega_m$ results in a graded filtration of $\wOmega_m\cong(\wgr\wA\otimes_{(\wgr\wA)_0}\wOmega_m)_m$ by modules $\wrDelta(\gamma)\cong \wgr \wP^\sharp(\gamma)\langle m\rangle _m$, proving
(b).

 Since each surjection
$$\wgr\wP^\sharp(\gamma)\twoheadrightarrow\wrDelta(\gamma)$$
has kernel with non-zero grades only in grades 1 or higher, it follows that $\widetilde \Omega'$ in
(\ref{syz}) vanishes in grades $\leq m$.  This establishes the first assertion of (a). The last assertion of (a) is clear, and so (a) is proved.

Finally, consider statement (c).  For any $s\in\mathbb Z$, (\ref{syz}) gives a short exact sequence
$$0\to\wOmega'_s\to \wX:=\left(\wgr\wA\otimes_{(\wgr \wA)_0}\wOmega_m\right)_s\to\wOmega_s\to 0,$$
of $(\wgr\wA)_0$-modules in which
$\wOmega_s$ has a $\wrDelta$-filtration by hypothesis. By (b), $\wX$ has a  filtration with sections $(\wgr\wP^\sharp(\lambda)\langle m\rangle )_s$.  Also, \cite[Thm. 3.1]{PS11} implies that $\wgr \wP^\sharp(\lambda)$ has a (graded) $\wgr\wDelta$-filtration. Thus, $\wgr P^\sharp(\lambda)=\overline{\wgr\wP^\sharp(\lambda)}$ has a (graded) $\wgr\Delta$-filtration, and
therefore, by    \cite[Thm 5.1]{PS11}, each section $(\wgr P^\sharp(\lambda))\langle m\rangle _s$ has a (graded) $\wrDelta$-filtration.  Thus, the (graded) $(\wgr \wA)_0$-module $(\wgr \wP^\sharp(\lambda))\langle m\rangle _s$ (concentrated in
grade $s-m$) has
a $\wrDelta$-filtration. Thus, $\wX$ has a $\wrDelta$-filtration, concentrated in grade $s$. Now the long
exact sequence of cohomology and \cite[Prop. 6.1]{PS11} gives that $\wOmega'_s$ has a $\wrDelta$-filtration,
completing the proof. \end{proof}

\medskip\noindent \underline{Proof of Theorem \ref{maintheorem}.} It suffices to prove part (b) of the
theorem. Then part (a) is obtained by base change to the field $k$. Define $\Gamma_0=\Lambda_{r(\Gamma)}$. Having defined $\Gamma_i$, put $\Gamma_{i+1}=\Lambda_{r(\Gamma_i)}$. The $(\Gamma,\wfa)$-projective resolution $\wXi_\bullet \to \wM$ is constructed recursively. Let $\wXi_0= (\wgr\wA)\otimes_{(\wgr\wA)_0}\wM_0$. Let $\wOmega_1$ to be the kernel of the  natural map $\wXi_0\twoheadrightarrow \wM$. Both $\wXi_0$ and $\wOmega_1$ are graded $\wgr\wA$-modules. Since $\fa$ is a Koszul algebra
and $\Xi_0=\wXi_k$ is a graded projective $\fa$-module generated by its grade $0$-component, $\Omega_1:=(\wOmega_1)_k$
is generated by its grade $1$-component. Therefore, by Nakayama's lemma, the graded $\wfa$-module $\wOmega_1$ is
generated  by its grade $1$-component $\wOmega_{1,1}$. In any given grade $s$, $\wOmega_{1,s}$ has a $\wrDelta$-filtration.

Now assume, for a given $i>0$, that graded $\wgr\wA_{\Gamma_{j-1}}$-modules
$\wOmega_j$ and $\wXi_{j-1}$ have been constructed, for $0<j\leq i$, such that

\begin{itemize}
\item[(i)] there is an exact sequence $0\to\wOmega_j\to\wXi_{j-1}\to \wOmega_{j-1}\to 0$ (with $\wOmega_{-1}=\wM)$);

\item[(ii)] $\wOmega_j|_{\wfa}$ and $\wXi_{j-1}|_{\wfa}$ are generated in grades $j$
and $j-1$, respectively;

\item[(iii)] $\wXi_{j-1}$ is a graded $\wgr\wA_j$-module, filtered by graded lattices
with sections $\wgr\wP^\sharp(\gamma)$, $\gamma\in \Gamma_{j-1}=\Lambda_{r(\Gamma_{j-2})}$;

\item[(iv)] $\wOmega_{j,s}$ has a $\wrDelta$-filtration, for each $s\in\mathbb Z$. (Here $\wOmega_{j,s}$
is the grade $s$-component of $\wOmega_j$.)

\end{itemize}

Define $\wXi_i=\wgr\wA_{\Gamma_i}\otimes_{(\wgr\wA_{\Gamma_i})_0}\wOmega_{i,i}$, and
set
$$\wOmega_{i+1}=\ker\left(\wgr\wA_{\Gamma_i}\otimes_{(\wgr\wA_{\Gamma_i)_0}}(\wOmega_{i,i}\to
\wOmega_i\right).$$
Condition (i),
with $i+1$ replacing $i$, clear from construction. Condition (ii) follows from the Koszulity of $\fa$, together
with (ii) for $j\leq i$ and Nakayama's lemma. Parts (iii) and (iv) follow from Lemma \ref{basic}. This completes
the recursive construction and the proof of the theorem.
\qed

\medskip
As a corollary of the proof of Lemma \ref{basic}, we record the following result.

\begin{cor}\label{torcor} Let $\gamma$  be a $p$-regular weight which is $r$-restricted, for some positive integer $r$. Let $\wA=\wA_{\Lambda_r}$ (see (\ref{Lambda})). Then, in the derived categories $D^-(\wgr\wA)$ and $D^-(\wgr A)$, we have
$$\begin{cases} \wgr\wP^\sharp(\gamma)\cong\wgr\wA\otimes^{\mathbb L}\wrDelta(\gamma);\\
\wgr P^\sharp(\gamma)\cong \wgr A\otimes^{\mathbb L}\rDelta(\gamma).\end{cases}
$$
Here $\otimes=\otimes_{(\wgr\wA)_0}$.
\end{cor}

\begin{proof} This follows from (\ref{rail}) and the proof of (\ref{tor}), for $n\geq 1$ (replacing $\Tor_1$ by $\Tor_n$, and again using Proposition \ref{appendixprop} below). \end{proof}

\section{Filtrations} The main result, Theorem \ref{nextmainresult}, shows, under the hypotheses of
Theorem \ref{maintheorem} that, if $\lambda,\mu$
are $p$-regular weights, the rational $G$-module $$\Ext^n_{G_1}(\rDelta(\lambda),\rnabla(\mu))^{[-1]}$$ has a
$\nabla$-filtration. 

Before beginning the proof of this theorem, we prove the following proposition which has independent
interest and plays a key role in the proof of Theorem \ref{standardtheorem} in \S6. The result is also based on Theorem \ref{maintheorem}, which guarantees the existence of  the resolutions $\Xi_\bullet$
and $\widetilde\Xi_\bullet$ in (a) below.

\begin{prop}\label{acyclic} Assume that $p\geq 2h-2$ is odd and that (\ref{LCF}) holds. Let $\Gamma$ be finite ideal of
$p$-regular weights. Let $M$ be a graded $\wgr A_\Gamma$-module which is $\fa$-linear. Assume that
each grade $M_s$ has a $\rDelta$-filtration and let $\Xi\to M$ be as in display (\ref{Gammaresolution}) as
guaranteed by Theorem \ref{maintheorem}(a).  Similarly, let $\wM$ be a graded $\gr\wA_\Gamma$-module which $\wfa$-linear, and such that each graded $\wM_s$ has $\wrDelta$-filtration. Let $\wXi_\bullet\twoheadrightarrow\wM$ be an integral version of (\ref{Gammaresolution}) as guaranteed by Theorem
\ref{maintheorem}(b).

(a) Then, setting
$A_i=A_{\Gamma_i}$ and $\wA_i=\wA_{\Gamma_i}$,
$$\begin{cases} \Ext^n_{\wgr A_i}(\Xi_i,\rnabla(\gamma))=0;\\
\Ext^n_{\wgr\wA_i}(\wXi_i,\wrnabla(\gamma)=0,\end{cases}$$
for all $i\geq 0$, all positive integers $n$, and all $\gamma\in \Gamma$.

(b) In particular, let $j\geq n$ be nonnegative integers with $j>0$, and let $\Lambda$ be a finite poset of
$p$-regular weights containing $\Gamma_j$. Put $A=A_\Lambda$ and $\wA=\wA_\Lambda$.
 Then
$$\Ext^n_{\wgr A_\Gamma}(M,\rnabla(\lambda))\cong \opH^n(\Hom_{\wgr A}(\Xi^\dagger_\bullet,\rnabla(\lambda)),$$
where $\Xi_\bullet^\dagger=\Xi^{\dagger_j}_\bullet$ is the $j$-truncated resolution (\ref{truncated}).

(c) In addition,
$$\begin{aligned}\grExt^n_{\wgr A_\Gamma}(M,\rnabla(\lambda)\langle r\rangle ) &\cong \opH^n(\grHom_{\wgr A}(\Xi^\dagger_\bullet,\rnabla(\lambda)\langle r\rangle )\\
& \cong \begin{cases} \Hom_{(\wgr A)_0}(\Omega_n/\rad\Omega_n,\rnabla(\lambda)),\quad n=r,\\
0\quad{\text{\rm otherwise.}}\end{cases}\end{aligned}
$$
 \end{prop}

\begin{proof}We first prove (a). Consider the integral case of $\gr \wA_i$. The $\gr \wA_i$-module $\wXi_i$ has a filtration
by the modules $\gr \wP^\sharp(\lambda)$, $\lambda\in\Gamma$. Thus, it suffices to show that
$$\Ext^n_{\gr\wA_i}(\wgr\wP^\sharp(\lambda),\wrnabla(\gamma))=0,\,\,\forall n>0,
$$
 Using
Cor. \ref{torcor},
$$\begin{aligned}\Ext^n_{\gr\wA_i}(\wgr\wP^\sharp(\lambda),\wrnabla(\gamma)) &\cong\Hom^n_{D^-(\wgr\wA_i)}(\gr\wP^\sharp(\lambda),\wrnabla(\gamma)) \\ &\cong\Hom^n_{D^-(\gr\wA_i)}(\gr\wA_i\otimes^{\mathbb L}\wrDelta(\lambda),\wrnabla(\gamma))\\
& \cong
\Ext^n_{(\gr \wA_i)_0}(\wrDelta(\lambda),
\wrnabla(\gamma))=0.\end{aligned}$$
  This proves (a) for $\gr\wA_i$. A similar argument
works for $\gr A_i$.

Finally, (b) and (c) follow from a standard  argument, using the spectral sequences associated to the Cartan-Eilenberg double complex resolution of $\Xi^\dagger_\bullet$; see \cite[Summary 5.7.9]{W}.
\end{proof}

We will need  the following preliminary lemma.

\begin{lem}\label{firstlemma} Assume that $p\geq 2h-2$ is odd, and that (\ref{LCF}) holds. Let $X$ be a finite dimensional $G$-module whose composition factors $L(\gamma)$
satisfy $\gamma\in X_{\text{\rm reg}}(T)_+$. Assume $X$  is completely reducible for $G_1$ and has a $\rDelta$-filtration as a $G$-module.
Then $\Hom_{G_1}(X,\rnabla(\gamma))^{[-1]}$ has a $\nabla$-filtration, for any $\gamma\in X(T)_+$. \end{lem}

\begin{proof} The statement is clearly true if $X=\rDelta(\gamma')$, $\gamma\in X_{\text{\rm reg}}(T)_+$, since
$$\Hom_{G_1}(\Delta^p(\gamma'),\nabla_p(\gamma))\cong\begin{cases} \Hom_k(\Delta(\gamma'_1)^{[1]},\nabla(\gamma_1)^{[1]}),\quad \gamma'_0=\gamma_0\\ 0, \quad{\text{\rm otherwise.}}\end{cases}$$
In general, consider a  short exact
sequence $0\to X'\to X\to X^{\prime\prime}\to 0$ of rational $G$-modules in which $X'$ and $X^{\prime\prime}$ are
nonzero modules having
$\rDelta$-filtrations. Observe this sequence is $G_1$-split. By an evident induction argument, we can assume the conclusion of the lemma holds with $X$ replaced by $X'$ or $X^{\prime\prime}$.
Form the exact sequence
$$0\to\Hom_{G_1}(X^{\prime\prime},\rnabla(\gamma))^{[-1]}\to\!\Hom_{G_1}(X,\rnabla(\gamma))^{[-1]}\to
\Hom_{G_1}(X',\rnabla(\gamma))^{[-1]}\!\to 0$$
of rational $G$-modules. By assumption, the right and left hand sides of this sequence
have $\nabla$-filtrations.  Thus, the middle term has a $\nabla$-filtration, as required. \end{proof}

We now establish the main result of this paper, part (a) of Theorem \ref{nextmainresult}. The first step in its proof identifies $\Ext^n_{G_1}(\rDelta(\lambda),\rnabla(\mu))$ as a {\it vector space}
with a rational $G$-module $\Hom_{G_1}(\Omega_n/\rad_\fa\Omega_n,\rnabla(\mu))^{[-1]}$ (in the 
notation of Theorem \ref{maintheorem}), which can be easily
shown to have a $\nabla$-filtration. Thus, it is necessary to show that this identification is an isomorphism of
$G$-modules. This final step, which is delicate, requires the abstract setting of \S8 (Appendix I).

Later, in \S6, part (a) of the theorem below will be extended to
the case in which $\rDelta(\lambda)$ (resp., $\rnabla(\mu)$) is replaced by $\Delta(\lambda)$
(resp., $\nabla(\mu)$); see Theorem \ref{Jantzen} below. Part (b) will be similarly extended in Theorem 6.5
using $\wgr \Delta(\lambda)$ and a dual construction. We note that part (a) of the theorem below assumes
the LCF condition (\ref{LCF}), while part (b) does not. Parallel remarks hold for their respective extensions
in \S6. 

\begin{thm}\label{nextmainresult} Assume that $p\geq 2h-2$ is odd.  Let $\lambda,\mu\in X_{\text{\rm reg}}(T)_+$.

(a) Suppose that condition (\ref{LCF}) holds.  Then, for any integer $n\geq 0$,  the rational $G$-module $\Ext^n_{G_1}(\rDelta(\lambda),\rnabla(\mu))^{[-1]}$  has a $\nabla$-filtration.

(b)
Let $A=A_\Gamma$, for any finite ideal $\Gamma$ of $p$-regular dominant weights containing
$\lambda,\mu$. For any integer $n\geq 0$, there are natural vector space isomorphisms
\begin{equation}\label{equiv1}\begin{aligned}
\Ext^n_{\wgr A}(\rDelta(\lambda),\rnabla(\mu)& )\cong\Ext^n_A(\rDelta(\lambda),\rnabla(\mu))\\
&\cong\Ext^n_G(\rDelta(\lambda),\rnabla(\mu)).\end{aligned}\end{equation}

\end{thm}

\begin{proof} Let $\Gamma$ be a finite ideal of $p$-regular weights containing $\lambda,\mu$. There is
an algebra homomorphism $\fa\to A:=A_\Gamma$ (which is an inclusion if $\Gamma$ is sufficiently large).

 Using Theorem \ref{maintheorem} and noting that $\rnabla(\mu)|_\fa$ is completely reducible, there
 are natural vector space isomorphisms (labelled for further discussion)
\begin{equation}\label{modules}
\begin{aligned}\Ext^n_{G_1}(\rDelta(\lambda),\rnabla(\mu)) &\overset{(1)}\cong \opH^n(\Hom_{\fa}(\Xi_\bullet,\rnabla(\mu))\\
&\overset{(2)}\cong \Hom_\fa(\Omega_n,\rnabla(\mu))\\ &\overset{(3)}\cong\Hom_{G_1}(\Omega_n/\rad_\fa\Omega_n,\rnabla(\mu)).\end{aligned}\end{equation}
The first term in (\ref{modules}) is obviously a rational $G$-module.  On the other hand, the last term $\Hom_{G_1}(\Omega_n/\rad_\fa\Omega_n,\rnabla(\mu))$ on the right is also a rational $G$-module. To see
this, first observe that
$$\rad_\fa\Omega_n=\fa_{\geq 1}\Omega_n=(\fa_{\geq 1}\wgr A)\Omega_n=
(\wgr A)_{\geq 1}\Omega_n.$$
(The first equality follows since $\fa$ is a Koszul algebra.) Now use the isomorphism
 $$\wgr A/(\wgr A)_{\geq 1}\cong A/A_{\geq 1}$$ to make $\Omega_n/\rad_\fa\Omega_n$ an
 $A$-module, thus, a $\Dist(G)$-module, and, finally, a rational $G$-module.

 Next, $\Omega_n\rad_\fa\Omega_n\cong \Omega_{n,n}$ (the grade $n$-component of $\Omega_n$) has,
 by Theorem \ref{maintheorem}(a), a $\rDelta$-filtration.
Thus,  Lemma \ref{firstlemma} implies that $\Hom_{G_1}(\Omega_n/\rad_\fa\Omega_n,\rnabla(\mu))^{[-1]}$
has a $\nabla$-filtration.

Therefore, (a) will follow provided the composite
\begin{equation}\label{composite}(3)\circ(2)\circ(1):\Ext^n_{G_1}(\rDelta(\lambda),\rnabla(\mu))\to\Hom_{G_1}(\Omega_n/\rad_\fa\Omega_n,
\rnabla(\mu))\end{equation}
of the vector space isomorphisms (1), (2), (3) in (\ref{modules}) is a morphism of
rational $G$-modules. While most readers will expect this to be true, a rigorous proof requires the
constructions of \S8 (Appendix 1) below.

First, general methods imply that the left hand side of (\ref{composite}) can be calculated, {\it as a rational $G$-module}, by {\it any}
a truncated resolution
\begin{equation}\label{partial}
0\to E\longrightarrow R_{n-1}\longrightarrow \cdots\longrightarrow R_0\longrightarrow \rDelta(\lambda)\to 0
\end{equation}
of $\rDelta(\lambda)$ by $A$-modules such that $R_i|_\fa$ is $\fa$-projective, $i=0, \cdots, n-1$. Here
we use the fact that  the category $A$-mod is the same as the category of finite dimensional
rational $G$-modules having composition factors $L(\gamma)$, $\gamma\in\Gamma$. (This statement
holds for any poset ideal $\Gamma$. As we see below, the current $\Gamma$ may need to be enlarged to
for any given $n$, to make the resolution construction possible.)
That is,
\begin{equation}\label{little}
\Ext^n_\fa(\rDelta(\lambda),\rnabla(\mu))\cong\ck(\Hom_\fa(R_{n-1},\rnabla(\mu))\longrightarrow
\Hom_\fa(E,\rnabla(\mu)))\end{equation}
in $G$-mod.\footnote{The isomorphism as vector spaces is elementary. Usually, given rational $G$-modules
$M,N$, the $G$-module structure of the spaces $\Ext_\fa^\bullet(M,N)$ is obtained by computing these
$\Ext$-groups using a $G$-injective resolution $N\to I_\bullet$ of $N$. Necessarily, each $I_j$ is also
$\fa$-injective, defining a rational $G$-module structure on $\Ext^n_\fa(M,N)$, regarded as the $n$-cohomology
of the complex $\Hom_\fa(M,I_\bullet)$. To see this $G$-action agrees with that defined by
the isomorphism (\ref{little}), temporarily denote $E$ by $R_n$ and form the double complex $\Hom(R_p,I_q)$.
Its total complex provides an $\fa$-injective resolution of $\Hom_k(\rDelta(\lambda),\rnabla(\mu))$.  Both
spectral sequences of the double complex collapse, one defining the usual $G$-action and the other
defining the action using (\ref{little}). Thus, the two actions are the same.}

We recursively construct such a resolution (\ref{partial}),  in order to establish that the map (\ref{composite})
is a homomorphism of rational $G$-modules. The argument will use the
results from \S8 (Appendix I). This appendix is written in an abstract framework, though with a recursive construction of (\ref{partial}) in mind. We try to give enough details to enable the reader to make the connection.See also Remark \ref{lastlastremark}. The construction
will be used in the proof of (b) below and, again, in Theorem \ref{GExt}.  Before getting started, we
note the following lemma. Its proof also requires results from Appendix I.

\begin{lem}\label{prelimlemma} Suppose $\Xi_\bullet^\dagger=\Xi^{\dagger_n}$ is a $n$-truncated complex as in (\ref{truncated})
in $\wgr A_{\Gamma_n}$-mod, so that $\Xi^\dagger_\bullet\twoheadrightarrow M$ is a graded resolution of
a $\wgr A_{\Gamma}$-module $M$, and so that each $\Xi^\dagger_j|_\fa$ ($j<n$) is projective. Assume that $M$ is
$\fa$-linear and that $\Xi^\dagger_\bullet |_\fa$ is part of a linear (graded) $\fa$-resolution.

Let $\Xi_\bullet^{\dagger\prime}$ be a second $n$-truncated resolution of $M$ with the same
properties as listed above for $\Xi_\bullet^\dagger$. Then $\Xi_\bullet^\dagger\cong\Xi_\bullet^{\dagger\prime}$
as  graded $\wgr A_{\Gamma_n}$-resolutions of $M$.\end{lem}

\begin{proof} For the proof, simply break each complex into short exact sequences, e.~g.,
$0\to\Omega_j\to\Xi_{j-1}\to\Omega_{j-1}\to 0$ and apply Theorem \ref{mainSecThm}(d)
with $B=\wgr A_{\Gamma_j}$, for various choices of $j$. Observe that, by construction, $\Xi_j$ has a projective
cover in $B$-grmod which remains projective upon restriction to $\fa$. (This observation follows
easily from the discussion involving
\cite[p.333]{JanB} after the statement of Theorem \ref{maintheorem}.)\end{proof}

For the first step of
the construction, put $N=\rDelta(\lambda)$ and replace $A$ by $A_{\Gamma_1}$, where
$\Gamma_1:=\Lambda_1(\Gamma)$ (so that $\fa\subset A$).
Proposition \ref{revised} gives a
short exact sequence $0\to E\to R\to N\to 0$ in $A$-mod and, upon restricting to $\fa$, in  $\fa$-grmod. Here $R|_\fa$ is projective in $\fa$-grmod. Moreover, all the objects
$X$ in this sequence have the property that $X_{\geq s}$, $s\in\mathbb Z$, (as defined by the $\fa$-grading on $X$)
are all $A$-modules, so we may construct a graded $\wgr A$-module
$$\wGr X=\bigoplus_{s\in\mathbb Z} X_{\geq s}/X_{\geq s+1}.$$
This construction guarantees that $\wGr X|_\fa\cong X|_\fa$ in $\fa$-grmod.
Also, according to Proposition \ref{revised}, there
is exact sequence $0\to E' \to R'\to N'\to 0$ in $\wgr A$-mod and in $\fa$-grmod, where $N'$ (at the moment) is just $N$, $E'$ is
a certain quotient of $\wGr E$  (in $A$-mod or in $\fa$-grmod) which is $\fa$-linear of degree 1 (in fact, the
maximal such linear quotient). The conditions in Proposition \ref{revised}  guarantee the hypotheses of Lemma \ref{prelimlemma}
hold, for $m=n=1$.    In particular, $\Omega_1\cong E'$ in $\wgr A$-grmod (and in $\fa$-grmod).

Now enlarge $A$ to $A_{\Gamma_2}$, where $\Gamma_2=\Lambda_1(\Gamma_1)$. Repeat
the argument with $N$ replaced by $E$. The new $N'$ will not be the same as $E$, but will be
$E'$. Continuing in this way, we obtain the sequence (\ref{partial}) in $A$-mod, for $A=A_\Lambda$
(for some large $\Lambda$). It is also an exact sequence in $\fa$-grmod.  The top row of the commutative diagram
$$\begin{CD}
0 @>>> \wGr E\ @>>> \wGr R_{n-1} @>>> \cdots @>>> \wGr R_0 @>>>\rDelta(\lambda) @>>> 0\\
@. @VVV @VVV @. @VVV @| \\
0 @>>> \Omega_n @>>> \Xi_{n-1} @>>> \cdots @>>> \Xi_0 @>>> \rDelta(\lambda) @>>>  0.
\end{CD}
$$
is exact in $\wgr A$-grmod.  The bottom row is just $\Xi_\bullet^\dagger$ (obtained by repeatedly applying
Lemma \ref{prelimlemma}).  Notice $\Omega_n\cong E'$ in Theorem \ref{secondmain}(b), by its recursive construction.  By Theorem \ref{secondmain}, there is a natural  isomorphism
$$\ck(\Hom_\fa(R_{n-1},\rnabla(\mu))\to\Hom_\fa(E,\rnabla(\mu)))\cong
\Hom_\fa(\Omega_n,\rnabla(\mu)),$$
easily seen to preserve the $G$-action on both sides. The key
point is that $\wGr E$ and $E$, as well as $\wGr R_{n-1}$ and $R_{n-1}$ share a (large) common
quotient $\hd^{\flat}E$ in $A/A_{\geq 1}$-mod. This gives the isomorphism (\ref{composite}). This proves (a).

The proof of (b) relies on  a similar construction, and uses Proposition \ref{revised}(a) and Theorem \ref{secondmain}(a).
As is well-known, the identification of $\Ext$-groups as cokernels such as those appearing in Theorem \ref{secondmain}(a) works equally
using projective resolutions or resolutions acyclic for $\Ext^\bullet_{\wgr A}(-,\rnabla(\mu))$ or
$\Ext^\bullet_A(-,\rnabla(\mu))$.   Since we are dealing with quasi-hereditary
algebras, it is enough that each $R_j$ be projective for $A_{\Gamma_j}$ for an ideal $\Gamma_j$
in $\Gamma$ with $\mu\in\Gamma_j$.
 \end{proof}

\begin{rem} As proved in \cite[\S5]{CPS7}, for $\lambda,\mu\in\Jan$, $\dim\Ext^n_G(\rDelta(\lambda),\rnabla(\mu))$ can be computed as a
appropriate coefficient of a (parabolic) Kazhdan-Lusztig polynomial when (\ref{LCF}) holds. The dimension agrees with the corresponding
$\dim\Ext^n_{U_\zeta}(L_\zeta(\lambda),L_\zeta(\mu))$ for the quantum enveloping algebra. See also
\S7 below, for a related study of costandard module multiplicities in $\dim\Ext^n_{G_1}(\rDelta(\lambda),\rnabla(\mu)).$ \end{rem}

\begin{thm}\label{lasttheorem}Assume that $p\geq 2h-2$ is odd, and that (\ref{LCF}) holds. Let $\lambda,\mu\in X_{\text{\rm reg}}(T)_+$, and let $A=A_\Gamma$ for some  be finite ideal
$\Gamma$ of $p$-regular
weights containing $\lambda,\mu$. Then, for any nonnegative integer $n$
and any integer $r$,
$$\grExt^n_{\wgr A}(\rDelta(\lambda),\rnabla(\mu)\langle r\rangle )\not=0\implies r=n.$$ \end{thm}

\begin{proof} As in the proof of the previous theorem, for any integer $r$,
$$\grExt^n_{\wgr A}(\rDelta(\lambda),\rnabla(\mu)\langle r\rangle )\cong{\text{\rm hom}}_{(\wgr A)_0}(\Omega_n/\rad\Omega_n,
\rnabla(\mu)\langle r\rangle )).$$
But $\Omega_n/\rad\Omega_n$ is pure of grade $n$, so if $\grExt^n_{\wgr A}(\rDelta(\lambda),\rnabla(\mu)\langle r\rangle )\not=0$, then $r=n$.\end{proof}

\medskip\noindent\underline{Proof of Theorem \ref{QKoszultheorem}:} First, by Proposition \ref{B0}
applied to $B=\wgr A$,
$(\wgr A)_0$ is quasi-hereditary with standard (resp., costandard) modules $\Delta^0(\lambda)=
\rDelta(\lambda)$ ($\nabla_0(\lambda)=\rnabla(\lambda)$), $\lambda\in\Lambda$. Thus, condition
(i) follows in Definition \ref{QKoszul}. Finally, condition (ii) is implied by Theorem \ref{lasttheorem}.
This completes the proof.  $\Box$

\medskip

When $n=r$ in the theorem, the value of $\dim\grExt^n_{\wgr A}(\rDelta(\lambda),
\rnabla(\mu)\langle r\rangle )$ can thus be calculated in terms coefficients of Kazhdan-Lusztig polynomials;
 see \cite[Thm. 5.4]{CPS7}, which gives the corresponding calculation of $\Ext^n$.

\begin{proof} This follows from Theorem \ref{lasttheorem} and the fact that $\rDelta(\lambda)$ is irreducible,
for $\lambda\in\Jan$.\end{proof}

\section{Further filtrations} This section gives certain variations on the results of \S5. Explicitly,
Theorem \ref{Jantzen} shows that if $\lambda,\mu$ are $p$-regular dominant weights, then the
$G$-modules
$\Ext^m_{G_1}(\Delta(\lambda),\rnabla(\mu))^{[-1]}$  and $\Ext^m_{G_1}(\rDelta(\lambda),\nabla(\mu))^{[-1]}$ have  $\nabla$-filtrations, for all  $m\geq 0$.  We also present proofs of Theorem \ref{standardtheorem} and its Corollary \ref{KLtheory}. This result requires Theorem \ref{DeltaKoszul} which shows that
each $\nabla(\nu)$ can be naturally viewed as a graded $\fa$-module, and that, as such, it is $\fa$-linear.

 In the following lemma, $B$ is a quasi-hereditary algebra with weight poset $\Lambda$, standard (resp., costandard) modules $\Delta(\lambda)=\Delta^B(\lambda)$ (resp., $\nabla(\lambda)=\nabla_B(\lambda)$), $\lambda\in\Lambda$. This lemma will be applied to the representation theory of $G$ in Theorem \ref{Jantzen}.

\begin{lem}\label{preplemma}Let $M\to N$ be a homomorphism of $B$-modules. Assume that $M$ has a $\nabla$-filtration,
and that
\begin{equation}\label{surjection} \Hom_B(\Delta(\sigma),M)\to\Hom_B(\Delta(\sigma),N)\quad{\text{\rm
is surjective}}\,\,\forall\sigma\in\Lambda.\end{equation}
Then $N$ has a $\nabla$-filtration, and the map $M\to N$ is surjective.\end{lem}

\begin{proof}Let $\lambda\in\Lambda$ be maximal, and put $\Gamma:=\Lambda\backslash\{\lambda\}$.
Let $M^\Gamma,N^\Gamma$ be the largest submodules of $M,N$, respectively, with all composition
factors in $L(\gamma)$, $\gamma\in\Gamma$. By induction, we may assume the result is true for quasi-hereditary algebras having posets
of smaller cardinality that that of $\Lambda$. In particular, if $J$ is the annihilator in $B$ of all modules
with composition factors $L(\gamma)$, $\gamma\in \Gamma$, then the lemma holds for the quasi-hereditary algebra $B':=B/J$. Thus,  $N^\Gamma\in B'$- has a $\nabla_{B'}$-filtration and the map $M^\Gamma\to N^\Gamma$ is surjective.  However, standard and costandard modules in $B'$-mod inflate to standard and
costandard modules in $B$-mod, so $N^\Gamma$ has a $\nabla$-filtration in $B$-mod, as well.
Now form the commutative diagram:

\[
\begin{CD}
\Hom_B(\Delta(\lambda),M) @>>a> \Hom_B(\Delta(\lambda),N)\\
@VVbV                               @VVcV\\
\Hom_B(\Delta(\lambda), M/M^\Gamma) @>>d>\Hom_B(\Delta(\lambda),N/N^\Gamma).
\end{CD}
\]
By hypothesis (\ref{surjection}), the map $a$ is surjective. Since $\lambda$ is maximal, $\Delta(\lambda)$
is projective in $B$-module, so that maps $b$ and $c$ are both surjective. Next, for $\sigma,\tau\in\Lambda$, $\Ext^1_B(\nabla(\sigma),\nabla(\tau))\not=0$ implies that $\sigma>\tau$. Thus, because $M$ has a
$\nabla$-filtration,  $M/M^\Gamma$
is a direct sum of copies of the injective module $\nabla(\lambda)$ which has socle $L(\lambda)$. On the
other hand, clearly $N/N^\Gamma$ has socle which is  a direct sum of copies of $L(\lambda)$. It follows
the socle of $M/M^\Gamma$ maps surjectively onto the socle of $N/N^\Gamma$. Thus, we can choose
a direct summand $X$ of $M/M^\Gamma$ which
maps isomorphically onto a submodule of $N/N^\Gamma$ containing the socle of $N/N^\Gamma$.
Since $X$ is injective, $X\cong N/N^\Gamma$. It follows that $N/N^\Gamma$ is isomorphic to
a direct sum of copies of $\nabla(\lambda)$. Since $N^\Gamma$ has a $\nabla$-filtration,
it follows that $N$ has a $\nabla$-filtration.

Finally, since we have shown that $M^\Gamma\to N^\Gamma$ and $M/M^\Gamma\to N/N^\Gamma$ are
surjective maps, it follows that $M\to N$ is surjective. This completes the proof.
\end{proof}

We are now ready to prove the following result.

\begin{thm}\label{Jantzen} Assume that $p\geq 2h-2$ is odd and that (\ref{LCF}) holds.
Let $\nu,\mu\in X_{\text{\rm reg}}(T)_+$ and $m\geq 0$.

  (a) The rational $G$-module $\Ext^m_{G_1}(\Delta(\nu),\rnabla(\mu))^{[-1]}
$ has a $\nabla$-filtration and the natural map
\begin{equation}\label{Jantzentheorem}
\Ext^m_{G_1}(\rDelta(\nu),\rnabla(\mu))\to\Ext^m_{G_1}(\Delta(\nu),\rnabla(\mu))\end{equation}
induced by $\Delta(\nu)\twoheadrightarrow\rDelta(\nu)$  is surjective.

(b)  Dually, the rational $G$-module $\Ext^m_{G_1}(\rDelta(\mu),\nabla(\nu))^{[-1]}$ has a $\nabla$-filtration
 and the natural map
 \begin{equation}\label{Jantzentheorem2}
 \Ext^m_{G_1}(\rDelta(\mu),\rnabla(\nu))\to\Ext^m_{G_1}(\rDelta(\mu),\nabla(\nu))\end{equation}
 induced by $\rnabla(\nu)\hookrightarrow\nabla(\nu)$ is surjective.
\end{thm}

\begin{proof}We will only prove part (a), leaving the dual assertion (b) to
the reader. We proceed by induction on $m$.

First, consider the $m=0$ case. By \cite[Prop. 2.3(b)]{PS11}, the $G_1$-head of $\Delta(\nu)$ is isomorphic to $\rDelta(\nu)\cong \Delta^p(\nu)$. Because $\rnabla(\mu)|_{G_1}$ is completely
reducible, it follows the map
$\Hom_{G_1}(\rDelta(\nu),\rnabla(\mu))\to\Hom_{G_1}(\Delta(\nu),\rnabla(\mu))$ is trivially an
isomorphism (and so, in particular, a surjection). Write $\nu=\nu_0+p\nu_1$ and $\mu=\mu_0+p\mu_1$ as
usual. If $\nu_0\not=\mu_0$, then $0=\Hom_{G_1}(\rDelta(\nu),\rnabla(\mu))^{[-1]}=\Hom_{G_1}(\Delta(\nu),\rnabla(\mu))^{[-1]}$, which has
a $\nabla$-filtration. Thus, suppose that $\nu_0=\mu_0$, so that the rational $G$-module
$$M^0:=\Hom_{G_1}(\rDelta(\nu),\rnabla(\mu))^{[-1]}\cong\Hom_k(\Delta(\nu_1),\nabla(\mu_1))\cong \nabla(\nu^\star_1)\otimes \nabla(\mu_1)$$
is isomorphic to a tensor product of two costandard modules; thus, it has a $\nabla$-filtration.
Therefore,
 $N_0:=\Hom_{G_1}(\Delta(\nu),\rnabla(\mu))^{[-1]}\cong M^0$ has a $\nabla$-filtration. This
completes the proof in the $m=0$ case.

Next, assume that assertion (a) is valid for
smaller values of some fixed integer $m>0$.
Let $\lambda\in X(T)_+$. Consider two Hochschild-Serre spectral sequences
$$\begin{aligned} & E^{a,b}_2=\Ext^a_{G/G_1}(\Delta(\lambda)^{[1]},\Ext^b_{G_1}(\Delta(\nu),\rnabla(\mu))
\Rightarrow \Ext_G^{a+b}(\Delta(\lambda^{[1]})\otimes\Delta(\nu),\rnabla(\mu)),\\
& {^\prime E}_2^{a,b}=\Ext^a_{G/G_1}(\Delta(\lambda)^{[1]},\Ext^b_{G_1}(\rDelta(\nu),\rnabla(\mu))
\Rightarrow \Ext_G^{a+b}(\Delta(\lambda)^{[1]}\otimes\rDelta(\nu),\rnabla(\mu)).\end{aligned}$$
For $a>0$ and $0\leq b<m$,
$$\Ext^a_{G/G_1}(\Delta(\lambda)^{[1]},\Ext^b_{G_1}(\Delta(\nu),\rnabla(\mu))\cong\Ext^a_G(\Delta(\lambda), \Ext^b_{G_1}(\Delta(\nu),\rnabla(\mu))^{[-1]})=0,$$
since, by induction, $\Ext^b_{G_1}(\Delta(\nu),\rnabla(\mu))^{[-1]}$ has a $\nabla$-filtration. In other words,
$E_2^{a,b}=0$ for $a>0$ and $0\leq b<m$, so that the edge map
$$\begin{aligned} E_\infty^m\!=\!\Ext^m_G(\Delta(\lambda)^{[1]}\otimes\!\Delta(\nu),& \rnabla(\mu))\\&\overset\sim\longrightarrow E_2^{m,0}=\Hom_{G/G_1}(\Delta(\lambda)^{[1]},
\Ext_{G_1}^m(\Delta(\nu),\rnabla(\mu))\end{aligned}$$
is an isomorphism. For the same reason, but now using Theorem \ref{nextmainresult}(a), the edge map
$$\begin{aligned}{^\prime E}_\infty^m=\Ext^m_G(\Delta(\lambda)^{[1]}\otimes &\rDelta(\nu),\rnabla(\mu))\\       &\overset\sim\longrightarrow{^\prime E}_2^{m,0} =\Hom_{G/G_1}(\Delta(\lambda)^{[1]},
\Ext_{G_1}^m(\rDelta(\nu),\rnabla(\mu))\end{aligned}$$
is also an isomorphism.

The natural surjection $\Delta(\nu)\twoheadrightarrow\rDelta(\nu)$ induces a map $^\prime E_\bullet\to E_\bullet$ of spectral sequences.   This gives a commutative diagram
\[\begin{CD} \Ext^m_G(\Delta(\nu),\rnabla(\mu)\otimes\nabla(\lambda^\star)^{[1]}) @>^\alpha>>
\Hom_{G}(\Delta(\lambda),\Ext^m_{G_1}(\Delta(\nu),\rnabla(\mu))^{[-1]})
\\
@A^\delta AA   @A^\beta AA \\
\Ext^m_G(\rDelta(\nu),\rnabla(\mu)\otimes\nabla(\lambda^\star)^{[1]}) @>^\epsilon>>
 \Hom_G(\Delta(\lambda),\Ext^m_{G_1}(\rDelta(\nu),\rnabla(\mu))^{[-1]})\end{CD}\]
in which the maps $\alpha$ and $\epsilon$ are isomorphisms (and are induced from the above
edge maps, after identifying $G/G_1$ with $G$ and untwisting the appropriate modules).

By Theorem \ref{nextmainresult}(a), $M^m:=\Ext^m_{G_1}(\rDelta(\nu),\rnabla(\mu))^{[-1]}$ has a $\nabla$-filtration. Let
$\Lambda$ be a large poset ideal in $X(T)_+$ containing all the dominant weights $\gamma$ such
that $L(\gamma)$ appears as a composition factor of the $G$-modules $M^m$ and $N^m:= \Ext^m_{G_1}(\Delta(\nu), \rnabla(\mu))^{[-1]}$, and let $B:=A_\Lambda$. Then Lemma \ref{preplemma} will imply that $N^m$ has
a $\nabla$-filtration and that (\ref{Jantzentheorem}) is surjective, provided that $\delta$ is surjective.

Equivalently, it suffices to show that the map $\beta$ is surjective.
   First, observe that $$\rnabla(\mu)\otimes
\nabla(\lambda^\star)^{[1]}\cong L(\mu_0)\otimes(\nabla(\mu_1)\otimes\nabla(\lambda^\star))^{[1]}.$$
Also, $\nabla(\mu_1)\otimes\nabla(\lambda^\star)$ has a $\nabla$-filtration in which the sections
$\nabla(\tau)$ satisfy $\tau\leq \mu_1+\lambda^\star$. Therefore, $\rnabla(\mu)\otimes\nabla(\lambda^*)^{[1]}$
has a $\rnabla$-filtration with sections $\rnabla(\xi)$ in which $\mu+p\lambda^\star-\xi\in p{\mathbb Z}R$.
In particular, $\xi$ is $p$-regular and has the same parity as $\mu+p\lambda^\star$, i.~e., $l(\xi)\equiv l(\mu+p\lambda^*)$
mod$\,\,2$.  Since $\rnabla(\xi)[-l(\xi)]\in \sE^R$ by \cite[Thm. 6.8]{CPS7},  $\rnabla(\mu)\otimes\nabla(\lambda^\star)^{[1]}$ belongs to $\sE^R$ or to $\sE^R[1]$ (depending on whether this parity is even or odd). Now Lemma \ref{ERlemma} implies that $\beta$ is a
surjection.   \end{proof}

Next, recall from (\ref{grading}) that $\fa$ has a positive grading induced from a grading on $\wfa$, as long
as $p>h$.  We show, in part (a) of the theorem below,  that the standard modules $\Delta(\nu)$, $\nu\in X_{\text{\rm reg}}(T)_+$, have a natural
$\fa$-grading, and as graded modules they satisfy
$\Delta(\nu)\cong\wgr\Delta(\nu)$.  This part of the theorem does not use the assumption (\ref{LCF}) that
the Lusztig character formula holds. However, if (\ref{LCF}) is assumed to hold, then we show also that  $\Delta(\nu)$ is linear over the Koszul algebra $\fa$.

\begin{thm}\label{DeltaKoszul} Assume that $p\geq 2h-2$ is odd.

(a) For $\lambda\in X_{\text{\rm reg}}(T)_+$, the standard module $\Delta(\lambda)$ has a graded $\fa$-module
structure,
 isomorphic to $\wgr\Delta(\lambda)$
over $\wgr\fa\cong\fa$.

(b) Assume that (\ref{LCF}) holds. With the graded structure given in (a), $\Delta(\lambda)$ is linear over $\fa$.
\end{thm}

\begin{proof} We first prove (a). In fact, we will prove a stronger statement, namely, that the grading
on $\Delta(\lambda)$ comes (via base change) from an $\wfa$-grading on $\wDelta(\lambda)$. (The proof makes heavy, though
implicit, use of a main result in \cite[Thm. 6.4]{PS9} which establishes
that, at the quantum enveloping algebra level,  $\Delta_K(\lambda)$ has a $\wfa_K$-grading.)

First, \cite[Thm. 6.3]{PS10} verifies the hypotheses of \cite[Thm. 5.3(ii)]{PS10}
in our context (ignoring the case $p=h=2$).\footnote{We take the opportunity to correct here several typos/omissions in
\cite{PS10}. p. 257, l. 17 down: replace this line by ``$\wR_\lambda\bigoplus(\wrad^{j_\lambda+1}\wN)_\lambda.$"
p. 266, l. 1 up: $\wP(\lambda)^\dagger=\wP(\lambda)_0^\dagger +\sum_{i\geq 1}\wfa_i\wP(\lambda)^\dagger=\wP(\lambda)^\dagger_0+\sum_{i\geq 1}\wfa_i\wP(\lambda)_0^\dagger+\sum_{i\geq 2}\wfa_i\wP(\lambda)^\dagger.$
p. 271, l. 14 down: Insert the sentence: ``Note that
$\wP^(\lambda)^\dagger$ inherits the structure of a $\wfa$-graded module from $\wQ(\lambda_0)$." before
the expression ``In general"
p. 271, l. 24 down: $K\wP(\lambda)^\dagger_0\subseteq \wA_{K,0}v$. p. 271, l. 26 down: ... as an $\wA_{K,0}$-module ...}  The module $P_K(\lambda)$ in \cite[Thm. 5.3]{PS10} is $\wP^\sharp(\lambda)_K$ in this
paper (see \S2.2). The verification in \cite[Thm. 6.3]{PS10} produces a lattice $\wP(\lambda)^\dagger
$ in $\wP(\lambda)_K$ with certain properties, including an $\wfa$-grading.  (In fact, $\wP(\lambda)^\dagger=
\wP^\sharp(\lambda_0)\otimes\wDelta(\lambda)^{[1]}$, where $\lambda=\lambda_0+p\lambda_1$ with
$\lambda_0\in X_1(T)$ and $\lambda_1\in X(T)_+$. The grading of $\wP(\lambda)^\dagger$ is inherited
from that
of $\wP^\sharp(\lambda_0)$.) The surjective map $\phi:P_K(\lambda)\to\Delta_K(\lambda)$ appearing
in the proof of \cite[Thm. 5.3]{PS10} is shown to satisfy:

\begin{itemize} \item[(i)] $\phi(\wP(\lambda)^\dagger)\cong \wDelta(\lambda)$---see the last line of the proof;

\item[(ii)] $\phi(\wP(\lambda)^\dagger)=\bigoplus_{i\geq 0}\wfa_i\phi(\wP(\lambda)^\dagger_0)$---see
the second and third displays on \cite[p. 269]{PS10}.
\end{itemize}
Thus, $\wDelta(\lambda)$ is a $\wfa$-graded module. On the other hand, $\wDelta(\lambda)$ is shown
in \cite[Thms. 5.3 \& 6.3]{PS10} to be $\wfa$-tight; see also \cite[Cor. 3.9]{PS10}. Hence, $\Delta(\lambda)=\bigoplus_{i\geq 0} \fa_i\Delta(\lambda)\cong \wgr\Delta(\lambda)$ as graded $\fa$-modules.

Next, we prove (b).  It suffices to prove that if $\grExt^n_\fa(\Delta(\lambda),\rnabla(\mu)\langle r\rangle )\not=0$, then $n=r$. However, the
surjection
(\ref{Jantzentheorem})
induces a surjection
$$\grExt^n_{\fa}(\rDelta(\lambda),\rnabla(\mu)\langle r\rangle )\twoheadrightarrow \grExt^n_\fa(\Delta(\lambda),
\rnabla(\mu)\langle r\rangle ).$$
 Thus, $\grExt^n_\fa(\rDelta(\lambda),\rnabla(\mu)\langle r\rangle )\not=0$, so $r=n$.
\end{proof}

\begin{rem} We emphasize again that Theorem \ref{DeltaKoszul}(a) does {\it not} require that the Lusztig
modular conjecture (equivalent to (\ref{LCF})) hold. Also, under the hypothesis of (a), it is proved in
 \cite[Cor. 3.2]{PS11} that $\wgr\Delta(\lambda)_0\cong\rDelta(\lambda)$
as a rational $G$-module. In \cite[Thm. 5.1]{PS11}, it proved under the hypothesis of (b) that $\wgr\Delta(\lambda)$ has a $\rDelta$-filtration, section by section.
\end{rem}

 Suppose that $\Gamma$ is a finite non-empty ideal of regular weights and let
$A=A_{\Gamma}$. For $\lambda\in\Gamma$, $\Delta(\lambda)\cong\widetilde
{\text{\textrm{gr}}}\,\Delta(\lambda)$ as $\fa\cong\widetilde
{\text{\textrm{gr}}}\,{\mathfrak{a}}$-modules by Theorem \ref{DeltaKoszul}(a).
On the other hand, $\wgr\Delta(\lambda)$ is a graded
$\wgr A$-module.  It follows easily that, for each
nonegative integer i, the ${\mathfrak{a}}$-submodule $\Delta(\lambda)_{\geq i}$ is
$A$-stable. In the sense of Definition \ref{hybrid} below and its discussion,
$\Delta(\lambda),$ together with its ${\mathfrak{a}}$-grading, has the structure
of an admissible hybrid $A$-module. Each admissible hybrid $A-$module $N$ has
an associated graded $\wgr A$-module

\[ \wGr\,N =\bigoplus_{j\in\mathbb{Z}}N_{\geq
j}/N_{\geq j+1}
\]
as defined  below Definition \ref{hybrid}. It is important for our discussion
to note that $N$ and\ $\wGr\,N$ have obviously
isomorphic restrictions to $\fa$-grmod, and that
$A
/\fa_{\geq1}A=(\wgr A)_{0}
\cong(\wgr A)/\fa_{\geq1}(\wgr A)$\ acts isomorphically on $N/\fa_{\geq
1}N\cong(\wGr\,N)/{\mathfrak{a}}_{\geq1}
(\wGr\,N)$. \ This latter isomorphism is a natural
transformation of functors on the admissible hybrid $A$-module category. Next observe
Corollary \ref{corlast1} can be applied after enlarging the poset $\Gamma$, using
$\Delta(\lambda)$ as the module $N$ there. Then $N$ can be replaced by the
admissible hybrid module $E$ obtained in that result$.$ Once again, the weight
poset can be enlarged and the process repeated. This process results in a
resolution $R_{\bullet}\twoheadrightarrow \Delta(\lambda)$ by modules which are all projective
over (various) quasi-hereditary algebras $A_{\Lambda\text{ }}$with
$\Gamma\subseteq\Lambda.$and $(A_{\Lambda})_{\Gamma}=A_{\Gamma}$. All the
differentials are maps of admissible hybrid $A_{\Lambda}$-modules for one of these posets
$\Lambda.$ In addition, $\wGr R_{\bullet}$ is
a graded resolution of $\wGr\Delta(\lambda)$ by
modules projective over\ the (various) associated quasi-hereditary graded
algebras $\wgr A_{\Lambda}$ with $(\wgr A_{\Lambda})_{\Gamma}=\wgr
A_{\Gamma}$. Consequently, for any $\mu\in\Gamma$, the resolution
$R_{\bullet}|_{A_{\Gamma}}$ is by objects which are acyclic for the functor
Hom$_{A_{\Gamma}}(-,$ $\rnabla(\mu))$. Similarly,
$(\wGr\,R_\bullet)|_{\wgr A_\Gamma}$
 is a resolution by objects
acyclic for the functor $\Hom_{\wgr A_{\Gamma}}(-,
\wnabla(\mu))$. Finally, using the isomorphisms
$R_{\bullet}/{\fa_{\geq1}}R_{\bullet}\cong(\wGr\,R_{\bullet})/\fa_{\geq1}(\wGr\,R_{\bullet})$, it follows that the respective
application of each of the two Hom  functors to these respective resolutions
by acyclic objects results, after making natural identifications, in exactly
the same complex!  This gives the first half of the following important
result. The proof of the second half, dual to the first, is left to the
reader. The conclusions, of course, should be compared with Theorem
\ref{nextmainresult}(b). Observe that the LCF assumption (\ref{LCF}) is {\it not}
required in the proof.

\begin{thm}
\label{GExt} Assume that $p\geq2h-2$ is odd. Let
$\lambda,\mu\in X_{\text{\textrm{reg}}}(T)_+$.  Let $A=A_{\Gamma}$, for any
finite ideal $\Gamma$ of $p$-regular dominant weights containing $\lambda,\mu
$. For any integer $n\geq0$, there are natural vector space isomorphisms
\begin{equation}
\label{equiv2}\begin{aligned} {\text{\rm Ext}}^n_{\wgr A}(\wgr\Delta(\lambda),\rnabla
(\mu))&\cong\Ext^n_A(\Delta(\lambda),\rnabla(\mu))\\ &\cong\Ext^n_G(\Delta(\lambda),\rnabla(\mu))\end{aligned}
\end{equation}
and
\begin{equation}
\label{equiv3}\begin{aligned} \Ext^n_{\wgr A}(\rDelta(\lambda),\widetilde{\text{\rm gr}}\,^\diamond\nabla(\mu))&\cong{\text{\rm Ext}}^n_A(\rDelta(\lambda),\nabla(\mu))\\ &\cong\Ext^n_G(\rDelta(\lambda),\nabla(\mu)).\end{aligned}
\end{equation}
\end{thm}

Now we are ready to complete the proof of several results from \S3.

\medskip\noindent{\underline{Proof of Theorem \ref{standardtheorem}:} We use the notation
of Theorem \ref{standardtheorem}. By Theorem \ref{QKoszultheorem}, $B=\wgr A$ is a Q-Koszul algebra
with weight $\Lambda$. As discussed in \S2.3, the graded algebra $B$ is also quasi-hereditary with weight poset $\Lambda$ and with
standard (resp., costandard) modules as indicated (in the statement of Theorem \ref{standardtheorem}). In
particular, condition (i) in Definition \ref{standard} holds.

 It therefore remains to check condition (ii) in Definition \ref{standard},
which is really two conditions. 
We will prove the first of these; the second follows by duality. Given $\lambda\in\Lambda$,
$\Delta^0(\lambda)=\rDelta(\lambda)$, again by Theorem \ref{QKoszultheorem}. Also, $\Delta^B(\lambda)=\wgr\Delta(\lambda)$, as noted above. In turn,
Theorem \ref{DeltaKoszul} implies (using both parts (a) and (b)) that $\Delta^B(\lambda)|_\fa$ is linear. Also, by the main result
\cite[Thm. 5.1]{PS11}, each section $\Delta^B(\lambda)_s=(\wgr \Delta(\lambda))_s$ has a $\rDelta$-filtration. 
Thus, $M:=\Delta^B(\lambda)$ satisfies the hypothesis of Theorem \ref{maintheorem}(a), using $\Gamma=\Lambda$. These
hypotheses appear again in Proposition \ref{acyclic} (which applies the construction of Theorem \ref{maintheorem}(a)). The vanishing in the conclusion of Proposition \ref{acyclic}(c) now gives the desired result.
\qed

 \medskip\noindent\underline{Proof of Corollary \ref{KLtheory}:} First, suppose that $\grExt^n_{\wgr A}(\wgr \Delta(\lambda),L(\mu)\langle r\rangle )\not=0$. Since $\Lambda\subset \Jan$, $L(\mu)=\rnabla(\mu)$, for all $\mu\in\Lambda$. Then by Theorem \ref{standardtheorem}, $n=r$. Also, Theorem \ref{GExt} implies that $\Ext^n_A(\Delta(\lambda),
 L(\mu))\not=0$. Therefore, using \cite{CPS1a}, we obtain that $l(\lambda)\equiv l(\mu)$ mod 2. It follows
 that $\wgr A$-mod has a graded Kazhdan-Lusztig theory (and so is Koszul). In particular, $\wgr \Delta(\lambda)$, $\lambda\in
 \Lambda$, is $\wgr A$-linear.  \qed

 \begin{rem} Theorem \ref{GExt} is really quite general, and would hold with $\Delta(\lambda)$ replaced by any other admissible hybrid $A$-module.  A dual statement holds for $\nabla(\mu)$.\end{rem}

\section{Calculations}

In this section,  assume that $p\geq 2h-2$ is odd, and that the Lusztig character formula holds (or, equivalently, that the isomorphisms (\ref{LCF}) hold). If $V$ is a (finite dimensional) rational $G$-module
having a $\nabla$-filtration $\mathscr F$, then the number of times  that a given module $\nabla(\gamma)$ appears as a
section in $\mathscr F$  depends only on $V$ (and not on $\mathscr F$); this multiplicity equals
$\dim\Hom_G(\Delta(\gamma),V)$. This well-known observation is immediate since
the functor $\Hom_G(\Delta(\gamma),-)$ is exact on the category of modules with a $\nabla$-filtration and
since $\dim\Hom_G(\Delta(\gamma),\nabla(\mu))=\delta_{\gamma,\mu}$. If $V$ has a $\nabla$-filtration, let $[V:\nabla(\gamma)]$ denote the multiplicity of $\nabla(\gamma)$ as a section of $V$ in a $\nabla$-filtration.

Recall that Theorem \ref{nextmainresult}(a) established that, if $\lambda,\mu\in X_{\text{\rm reg}}(T)_+$ and $n$ is a nonnegative integer, then the rational $G$-module $\Ext^n_{G_1}(\rDelta(\lambda),\rnabla(\mu))^{[-1]}$ has a $\nabla$-filtration.  Also, Theorem \ref{Jantzen} showed that both $\Ext^n_{G_1}(\Delta(\lambda),\rnabla(\mu))^{[-1]}$
and $\Ext^n_{G_1}(\rDelta(\mu),\nabla(\lambda))^{[-1]}$ have $\nabla$-filtrations.

This section describes
how the multiplicity of a $\nabla(\tau)$, $\tau\in X(T)_+$, in a $\nabla$-filtration of $\Ext^n_{G_1}(\Delta(\lambda),\rnabla(\mu))^{[-1]}$
 can be combinatorially determined in terms of the coefficients of certain Kazhdan-Lusztig polynomials $P_{x,y}$ for the
the affine Weyl group $W_p$ of $G$, plus a well-known multiplicity result of Steinberg. We regard $P_{x,y}$ as a polynomial in $t:=\sqrt{q}$---in fact, it is
a polynomial in $t^2$. Also, in the formulas below, $\overline{P}_{x,y}$ is obtained from $P_{x,y}$ by replacing $t$ by $t^{-1}$.

First, consider  $\Ext^n_{G_1}(\rDelta(\lambda),\rnabla(\mu))^{[1]}$.  Write $\lambda=\lambda_0+p\lambda_1$
and $\mu=\mu_0+p\mu_1$, where $\lambda_0,\mu_0\in X_1(T)$, $\lambda_1,\mu_1\in X(T)_+$.
Hence, $\rDelta(\lambda)\cong L(\lambda_0)\otimes\Delta(\lambda_1)^{[1]}$ and $\rnabla(\mu)\cong L(\mu_0)
\otimes\nabla(\mu_1)^{[1]}$. Thus,
$$\begin{aligned}\Ext^n_{G_1}(\rDelta(\lambda), \rnabla(\mu))^{[-1]} &\cong \Hom_k(\Delta(\lambda_1),\nabla(\mu_1))\otimes
\Ext^n_{G_1}(L(\lambda_0),L(\mu_0))\\ &\cong \nabla(\lambda_1^\star)\otimes\nabla(\mu_1)\otimes \Ext^n_{G_1}(L(\lambda_0),L(\mu_0)).\end{aligned}$$
It is well-known (and has been already used several times in this paper)
that the tensor product of modules of the form $\nabla(\tau)$, $\tau\in X(T)_+$, has a $\nabla$-filtration,
the terms of which can be determined by character-theoretic calculations, using Steinberg's theorem
\cite[24.2]{Hump}. Thus, it suffices to
determine the multiplicities of $\nabla$-sections in $\Ext^n_{G_1}(L(\lambda_0),L(\mu_0))^{[-1]}$.
Observe that  $L(\lambda_0)\cong\rDelta(\lambda_0)$ and $L(\mu_0)=\rnabla(\mu_0)$.

Thus, we can assume from the start that $\lambda=\lambda_0$ and $\mu=\mu_0$ are restricted dominant weights. Then,
if $\tau\in X(T)_+$, the multiplicity of $\nabla(\tau)$ as a section in a $\nabla$-filtration of
$\Ext^n_{G_1}(\rDelta(\lambda_0),\rnabla(\mu_0))^{[-1]}$ is
$$\begin{aligned}
 \left[(\Ext^n_{G_1}(\rDelta(\lambda_0),\rnabla(\mu_0))^{[-1]}:\nabla(\tau)\right]& =\dim\Hom_G(\Delta(\tau)^{[1]},
\Ext^n_{G_1}(\rDelta(\lambda_0),\rnabla(\mu_0))\\
&=\dim\Ext^n_{G_1}(\rDelta(\lambda_0)\otimes\Delta(\tau)^{[1]},\rnabla(\mu_0))^G\\
&=\dim{\left(\Ext^n_{G_1}(\rDelta(\lambda_0+p\tau),\rnabla(\mu_0))^{[-1]}\right)}^{G}\\
&=\dim\Ext^n_G(\rDelta(\lambda_0+p\tau),\rnabla(\mu_0)).\end{aligned}
$$
The last equality holds because the Hochschild-Serre spectral sequence (using $G_1$ as the normal
subgroup scheme) for computing
$\Ext^n_G(\rDelta(\lambda+p\tau),\rnabla(\mu_0))$ has $E_2^{a,b}$-term ($a+b=n$) given by
$$E_2^{a,b}=\opH^a(G,\Ext^b_{G_1}(\rDelta(\lambda_0+p\tau),\rnabla(\mu_0))^{[-1]}).$$
However, $E_2^{a,b}=0$ if $a>0$, since $\opH^a(G,V)=0$, for $a>0$ and any rational $G$-module $V$
having a $\nabla$-filtration.

Write $\lambda':=\lambda_0+p\tau=x\cdot\lambda^-$ and $\mu_0=y\cdot\mu^-$, where $\lambda^-,\mu^-$ belong to the $p$-alcove $C^-_p$ containing $-2\rho$, and $x,y$ are (uniquely determined) elements of $W_p$. We can assume that $\lambda^-=\mu^-$, otherwise all the $\Ext$ groups are 0 by the linkage principle. Then since $p\geq 2h-2$ is odd and since (\ref{LCF}) is assumed to hold, \cite[Thms. 5.4 \& 6.7]{CPS7} implies that

$$\begin{aligned}
\dim\Ext_{G}^n(\rDelta(\lambda'),& \rnabla(\mu_0))\\
& =\sum_{m=0}^n\sum_\nu \dim\Ext_{G}^m(\rDelta(\lambda'),\nabla(\nu))
\cdot\dim\Ext_{G}^{n-m}(\Delta(\nu),\rnabla(\mu_0)).\end{aligned}
$$
The dimensions of the $\Ext$-groups appearing in the sum are all coefficients of Kazhdan-Lusztig
polynomials, as shown in \cite[\S5]{CPS7}. More precisely,
for a given $\nu$, above Ext groups are 0, unless $\nu=z\cdot\lambda^-$, for some $z\in W_p$.
Then
\begin{equation}\begin{aligned}\label{KLcom} t^{l(x)-l(z)} \overline{P}_{z,x}
&=\sum_{n\geq 0}\dim\,\Ext^n_G(\rDelta(\lambda'),\nabla(z\cdot
\lambda^-))t^n\\
&=\sum_{n\geq 0}\dim\,\Ext^n_G(\Delta(z\cdot\lambda^-),\rnabla
(\lambda'))t^n.\end{aligned}
\end{equation}

Thus, the multiplicity of $\nabla(\tau)$ can be combinatorially calculated in terms of Kazhdan-Lusztig polynomial coefficients. We give the formula explicitly below, up to Steinberg's formula for multiplicities
in tensor products mentioned above, which calculates the multiplicities $\left[\nabla(\lambda^\star)\otimes\nabla(\mu_1)\otimes\nabla(\tau):\nabla(\omega)\right]$ in (\ref{formula}). Given $u,v\in W_p$ and $s\in {\mathbb Z}$,  $c(u,v,s)$ denotes the coefficient
of $t^s$ in $P_{u,v}$.
Thus,
\begin{equation}\label{BobSteinberg}
P_{u,v}=\sum_{s\geq 0}c(u,v,s)t^s.\end{equation}
For $p$-regular dominant
weights $\lambda,\mu$, write  $\lambda=x\cdot\lambda^-$ and $\mu=y\cdot\mu^-$, for unique $x,y\in W_p$,
and unique $\lambda^-,\mu^-\in C^-$. Using (\ref{BobSteinberg}), put
$$C(\lambda,\mu,n):=\begin{cases} 0, \quad {\text{\rm when}}\quad\lambda^-\not=\mu^-;\\
\sum_{z}\sum_{m=0}^n
c(z,x,l(x)-l(z)-m)\cdot c(z,y,l(y)-l(z)-n+m),\\ \quad {\text{\rm when}}\quad\lambda^-=\mu^-, \end{cases}$$
where $\sum_z$ is the sum over all $z\in W_p$ satisfying $z\cdot\lambda^-\in X(T)_+$.

Now we can state
\begin{thm}\label{calcthm}Let $\lambda,\mu\in X_{\text{\rm reg}}(T)$ and let $n$ be a nonnegative integer.
For any $\omega\in X(T)_+$,
\begin{equation}\label{formula}\begin{aligned}
\rule{0pt}{0pt}[ \Ext^n_{G_1}(\rDelta(\lambda),& \rnabla(\mu))^{[-1]}:\nabla(\omega) ] \\
&=\sum_{\tau\in X(T)_+}C(\lambda_0+p\tau,\mu_0,n)
 \left[\nabla(\lambda_1^\star)\otimes\nabla(\mu_1)\otimes\nabla(\tau):\nabla(\omega)\right].\end{aligned}\end{equation}

\end{thm}

For the case of $\Ext^n_{G_1}(\Delta(\lambda),\rnabla(\mu))$, the calculations are easier (but use Theorem \ref{Jantzen}) and are
left to the reader. Given $p$-regular weights  $\lambda=x\cdot\lambda^-$ and $\mu=y\cdot\mu^-$ as
above, define, for $n\in\mathbb Z$,
$$c(\lambda,\mu,n):=\begin{cases} 0, \quad{\text{\rm when}}\quad \lambda^-\not=\mu^-;\\
c(x,y,l(x)-l(y)-n), \quad{\text{\rm when}}\quad  \lambda^-=\mu^-.\end{cases}$$

\begin{thm}\label{calcthm2} Let $\lambda,\mu\in X_{\text{\rm reg}}(T)$ and let $n$ be a nonnegative integer.
For any $\omega\in X(T)_+$,
\begin{equation}\label{formla2}
\begin{aligned} \rule{0pt}{0pt}[\Ext^n_{G_1}(\Delta(\lambda),& \rnabla(\mu))^{[-1]}:\nabla(\omega)] \\
  &=\sum_{\tau\in X(T)_+} c(\lambda,\mu_0+p\tau^\star,n)[\nabla(\tau)\otimes\nabla(\mu_1):\nabla(\omega)].
  \end{aligned}  \end{equation}
  and
\begin{equation}\label{formla3}
\begin{aligned}\rule{0pt}{0pt}[\Ext^n_{G_1}(\rDelta(\lambda),\nabla(\mu))^{[-1]}: & \nabla(\omega)] \\
&=\sum_{\tau\in X(T)_+} c(\mu,\lambda_0+p\tau,n)[\nabla(\lambda_1)\otimes\nabla(\tau):\nabla(\omega)].
\end{aligned}\end{equation}
\end{thm}

\begin{rems}(a) Choose a total ordering $\lambda_0\prec\lambda_1\prec\cdots$ of $X_{\text{\rm reg}}(T)_+$ with the property that
$\lambda\leq \mu\implies \lambda\prec\mu$. Since $\Ext^1_G(\nabla(\lambda),\nabla(\mu))\not=0$ implies
that $\mu<\lambda$, an explicit description (in some sense) of a $\nabla$-filtration of any of the above
$\Ext^n_{G_1}$-groups can be given once the $\nabla$-multiplicities are calculated.

(b) Observe that
$$\begin{cases}\dim\Ext^n_G(\Delta(\lambda),\rnabla(\mu_0+p\tau))=\dim\Ext^n_{U_\zeta}(\Delta_\zeta(\lambda),
L_\zeta(\mu_0+p\tau)),\\
\dim\,\Ext^n_{G}(\rDelta(\lambda),\rnabla(\mu_0))=\dim\,\Ext^n_{U_\zeta}
(L_\zeta(\lambda),L_\zeta(\mu_0))
\end{cases}$$

(c) In (\ref{formla3}), if $\lambda=0$, we find, using \cite[Lem. 4.1(b)]{CPS7} that the total multiplicity
of $\nabla(\tau)$ as a section in a $\nabla$-filtration of $H^\bullet(G_1,\nabla(\mu))^{[-1]}$ equals
the dimension $\dim\Delta(\tau)_\xi$ of the $\xi$-weight space in $\Delta(\tau)$. Here we write
$\mu=w\cdot 0 +p\xi$, $\sigma\in C_p$.
\end{rems}

 \section{Appendix I: syzygies}

 This appendix coordinates the representation theory of a positively graded ``subalgebra" $\fa$ with that of a larger algebra, which is allowed to be graded or ungraded.  In fact, both cases arise, and we will use $B$ for an algebra which is graded, and $A$ for an algebra that may not have a grading. We will assume that $\fa$ is an actual subalgebra of $A$, but, for $B$ we require only that we have only a natural homomorphism $\fa\to B$ of graded algebras, which might well also be injective. In applications, $B$ will arise as the graded algebra $\wgr A$ associated with a filtration of $A$, and the map $\fa\to B$ will occur naturally from this construction. We set this up in reasonable generality in \S8.2, which is aimed at coordinating the representation theory of all three algebras. The first \S8.1 deals with the graded algebras $\fa$ and $B$ only.   We largely have in mind here the case where $\fa$ is a Koszul algebra, though the results are formulated under only the assumption
that  $\fa$ is positively graded. A central issue addressed is how to formulate the notion of a nice resolution in $B$-grmod of a module which, in $\fa$-grmod, has a linear resolution. This leads to the notation of a semilinear
 resolution, formulated below. Another concept in \S8.1 is the notion of the ``flat" radical of a
 (graded or ungraded) module over a graded algebra. When $\fa$ and $B$ are sufficiently closely related
 (see Definition \ref{flathead}), the flat radical $\rad^\flat M$ of any $B$-module $M$, whether taken with respect to $\fa$ or $B$, give the same subspace. In \S8.2, the quotient module $\hd^\flat M:=M/\rad^\flat M$ is also a $A$-module. In this way, the representation theories of $A$, $\fa$ and $B$ can be coordinated. The consequent results---here all cast in an abstract finite dimensional algebra setting---play an
 important role in the algebraic group results in \S5. This is discussed more at the end of this section.

 In this section, all algebras and modules for them will always be finite dimensional over the field $k$.

 \subsection{Syzygies of graded modules.}
	   Let $\fa=\bigoplus_{n\geq 0}\fa_n$ be a positively graded algebra.  Generalizing slightly the terminology of \S2.5, a graded $\fa$-module $M$ is said to be {\it linear of degree}
   $m\in{\mathbb Z}$ if the following conditions hold:
   \begin{enumerate}
   \item[(i)] $M$ is generated by its grade $m$-component $M_m$, and
   \item[(ii)] if $M$ has a graded projective resolution $\cdots\to P_{m+1}\to P_m\to M\to 0$ such that, for each $i\geq m$,
 $\Omega_{i+1}:=\ker(P_i\to P_{i-1})$ is generated by its grade $i+1$-component $\Omega_{i+1,i+1}$.
 (Here $P_{m-1}:=M$.)
 \end{enumerate}

 Clearly, $M$ is linear of degree $m$ if and only if it satisfies condition (i), and condition (ii) holds for
 its minimal graded projective resolution. In this case,
 $\Omega_{i+m}$ is called the $i$th syzygy module of $M$.

 Thus, the usual notion of a linear (or Koszul) module is the same as that of an $\fa$-module which is
 linear of degree 0. The $m$th syzygy of such a module is linear of degree $m$.

 It is useful to have a notion which applies to syzygies in more general resolutions. A graded $\fa$-module $M$ will be called {\it semilinear of degree $m$} if $M$ is a direct sum $M=N\oplus P$, where $N$ is linear of
 degree $m$ and $P$ is projective and is generated by its components in grades $< m$, i.~e., $P=\fa(P_{<m})$, where $P_{<m}:=
 \bigoplus_{i<m}P_i$.  Many important resolutions that we encounter of linear modules have $m$th syzygies
 which are semilinear of degree $m$. We are able to show this by proving that semilinearity is ``inherited" in the short exact sequences building the resolutions we require, and it provides considerable  structure for these
 resolutions. Before stating the main theorem in this direction, we introduce more notation and give some
   general preliminary results.

   \begin{defn}\label{flathead} Let $E$ be any graded  or ungraded $\fa$-module. Define the ``flat radical" of $E$ to be
   $$\rad^\flat E:=\fa_{\geq 1}E:=\sum_{i\geq 1}\fa_iE.$$
   Also, the ``flat head" of $E$ is
   $$\hd^\flat E:=E/\rad^\flat E.$$
   \end{defn}

   Observe that $\rad^\flat E=(\rad^\flat\fa)E\subseteq (\rad\fa)E=\rad E,$
   since $\fa_{\geq 1}=\rad^\flat\fa$ is a nilpotent ideal of $\fa$.

   Now suppose that $E$ is graded $\fa$-module. Both $\rad^\flat E$ and $\hd^\flat E$ are graded
   $\fa$-modules, and $\hd^\flat E$ decomposes as an $\fa$- (or $\fa_0$-) module as
   $\hd^\flat E=\bigoplus_{i\in{\mathbb Z}}(\hd^\flat E)_i$. There is also a natural identification
   $(\hd^\flat E)_i=E_i/\sum_{j>0}\fa_jE_{i-j}$, for each $i\in\mathbb Z$.

   For any graded $\fa$-module $E$, and $s\in{\mathbb Z}$, define graded $\fa$-submodules
   \begin{equation}\label{syzy1.5}\begin{cases} E^s:=\fa E_s,\\ E^{\leq s}:=\sum_{j\leq s}E^j, \\ E^{<s}:=E^{\leq s-1}, \\
   E^{\# s}=E^{\leq s}/E^{<s}.\end{cases}
  \end{equation}
   There is a natural filtration
   \begin{equation}\label{syzy1} \cdots \subseteq E^{\leq s}\subseteq E^{\leq s+1}\subseteq\cdots
   \end{equation}
   with, of course, only finitely many distinct terms. There is a corresponding filtration of the graded quotient
   module $\hd^\flat E$ of $E$, and we have, for each $s\in{\mathbb Z}$, natural isomorphisms

  \begin{equation}\label{syzy2}\begin{cases}
  \hd^\flat E^{\leq s}\cong (\hd^\flat E)^{\leq s},\\ \hd^\flat E^{\# s}\cong(\hd^\flat E)^{\#s}\cong(\hd^\flat E)_s.
  \end{cases}
  \end{equation}
  Any homomorphism $E\to F$ of graded $\fa$-modules  induces maps $E^{\leq s}\to F^{\leq s}$ and
  $E^{\# s}\to F^{\# s}$, both surjections whenever the original map is a surjection.

  \begin{defn}\label{tighty} If $\fa \to B$ is a morphism of graded algebras, we say that $B$ is (left) {\it tight
  over} $\fa$ if $\fa B_0=B$. (There is, of course, a corresponding right hand notion.\footnote{The word
  ``tight" in this paper is an adjective applying in many not necessarily related contexts. In particular, $B=\fa$ is always tight over $\fa$, but $\fa$ is not
  necessarily a tightly graded algebra---which means that it is generated by $\fa_0$ and $\fa_1$.})\end{defn}

  When $B$ is tight over $\fa$, and $E=E'|_\fa$, for a graded $B$-module $E'$, then all the graded
  $\fa$-modules listed in (\ref{syzy1.5}) inherit natural graded $B$-module structures from $E'$, for
  any $s\in{\mathbb Z}$. In fact, $E^s=E^{\prime s}|_\fa$, etc.

  \begin{lem} \label{syzyLem} Suppose that $M$ is a graded semilinear $\fa$-module of degree $m$.

  (a)  All the inclusions
  in the filtration (\ref{syzy2}) are split as graded $\fa$-modules, and there is a direct sum decomposition
  $M\cong\bigoplus_{s\in{\mathbb Z}} M^{\# s}$ in which $M^{\# m}$ is linear of degree $m$, $M^{\# s}$
  is projective (and generated in grade $s$) for $s\not=m$, and $M^{\# s}=0$ for
  $s> m$.

  (b) Moreover, $M^{\# m}$ naturally inherits a $B$-module structure $M^{\prime\# m}$, whenever
  $B$ is a graded algebra which is tight over $\fa$ and $M=M'|_\fa$, for a graded $B$-module $M'$.
  Also, the natural surjection $M\twoheadrightarrow M^{\# m}$ agrees by restriction with the natural
  surjection $M'\twoheadrightarrow M^{\prime \# m}$.
  \end{lem}

  \begin{proof} By definition, $M\cong N\oplus P$, where $N$ is linear of degree $m$ and $P$ is a graded
  projective $\fa$-module generated in grades $<m$.  Of course, $M^{<m}$ is also generated in grades $<m$,
  hence projects to 0 in $N$, which has $N_s=0$ for $s<m$. Therefore, $M^{<m}=P$, $N\cong M^{\# m}$,
  and $M\cong M^{\# m}\oplus M^{<m}$ with $M^{<m}=P$. The projective module $P$ qualifies as a graded
  semilinear $\fa$-module of degree $m-1$, so the process can be repeated, obtaining $M^{<n}\cong
  M^{\# m-1}\oplus M^{< m-1}$ with
  $M^{< m-1}$ projective, etc. This proves (a).

  Finally, (b) follows from the discussion preceding the statement of the lemma. \end{proof}

  \begin{rems}\label{discussion} We have implicitly assumed that projective covers exist in the
  category of graded $B$-modules, for any positively graded algebra $B$. We will elaborate on this
  a little.

(a)  First, consider the case of the category $B$-mod of ungraded $B$-modules.
   Observe that the exact restriction functor $B{\text{\rm --mod}}\longrightarrow B_0$--mod
  has a right exact left adjoint $B\otimes_{B_0}-$. Thus, if
  $P$ be any projective
  $B_0$-module, then $B\otimes_{B_0}P$ is a  projective $B$--module.
  Every projective $B$-module has this form. In fact, the irreducible $B$-modules naturally identify with
  the irreducible $B_0$-modules. If $L$ is an irreducible $B_0$-module with projective cover $P$ in $B_0$--mod,
  then $B\otimes_{B_0}P$ is the projective cover of $L$ regarded as an $B$-module.

 (b) Second, a similar construction
  works at the graded level. First, regard $B_0$ as a positively graded algebra concentrated in grade 0. The
  graded projective $B_0$-modules are just projective $B_0$-modules $P$ equipped with a direct sum
  decomposition $P=\bigoplus_{i\in\mathbb Z} P_i$ in $B_0$-mod, with $P_i$ viewed as a graded $B_0$-module
  concentrated in grade $i$. In this way, $P$ is a graded $B_0$-module. Again, the exact restriction
  functor $B{\text{\rm --grmod}}\longrightarrow B_0$-grmod has right exact left adjoint $B\otimes_{B_0}-$.
  In fact, if $X=\bigoplus_{i\in \mathbb Z}X_i\in B_0$--grmod, then $(B\otimes_{B_0} X)_j :=
  \bigoplus_{i\in\mathbb Z}B_i\otimes_{B_0}X_{j-i}$, for each $j\in\mathbb Z$, defines $B\otimes_{B_0}X$
  as a graded $B$-module. If $X=P$ is projective in $B_0$-grmod, then
  $R:=B\otimes_{B_0}P$ is projective in $B$-grmod. We have
  $$R=\bigoplus_{s\in\mathbb Z} R^{\# s}\quad{\text{\rm and}}\quad R^{\# s}\cong {B\otimes_{B_0}P_s,\quad (s\in\mathbb Z}).$$
  If $R\to N$ is a homomorphism in $B$-grmod, then the image of $R^{\# s}$ is contained in $N^{\leq s}$.
  Moreover, $R\to N$ is surjective if and only if all the composite maps $R^{\#s}\to N^{\leq s}\to N^{\# s}$ are
  surjective.   If $P$ is the projective cover in $B_0$-grmod of $\hd^\flat N=
  \bigoplus_{s\in\mathbb Z}\hd^\flat N^{\# s}$, then $R=B\otimes_{B_0}P$ is the projective cover of $N$ in
  $B$-grmod. So each $R^{\# s}\to N^{\# s}$ is surjective in this case. While $R^{\# s}$ is a direct summand of $R$, the module $N^{\# s}$ is, in general,
  only a section of $N$. Finally, forgetting the gradings, $R$ is the projective cover of $N$ in $B$--mod.
   \end{rems}

 We now state the main theorem of this subsection.
 \begin{thm}\label{mainSecThm} Suppose that $\fa\to B$ is morphism of positively graded algebras such that
 $B$ is tight over $\fa$. Let $N$ be a graded $B$-module such that $N|_\fa$ is semilinear of degree $m$. Suppose
 there is given a projective $B$-module $P$ such that $P|_\fa$ is also projective and such that there is a
 surjection
 $P\twoheadrightarrow N$ in $B$-mod. Then the following statements hold:

 (a)  Let $R\twoheadrightarrow N$ be the projective cover in the category $B$-grmod. Then
  $R|_\fa$
 projective in $\fa$-grmod.

 (b) In the short exact sequence $0\to E\to R\to N$ in $B$-grmod,  $E:=\ker(R\twoheadrightarrow N)$ is semilinear of degree $m+1$ in $\fa$-grmod.

 (c) The graded $B$-modules $E^{\# m+1}$ and $N^{\# m}$, when restricted to $\fa$, are linear of degrees
 $m+1$ and $m$, respectively.

 (d) There is, up to isomorphism, a unique graded $B$-module $P'$ for which there is a graded $B$-module
 homomorphism $P'\to N^{\# m}$ becoming a projective cover upon restriction to $\fa$. The kernel of this map
 is isomorphic to $E^{\# m+1}$, and the resulting short exact sequence $0\to E^{\# m+1}\to P'\to N^{\# m}\to 0$
 in $B$-grmod is unique up to isomorphism (assuming $P'|_\fa$ is projective).

 (e) The short exact sequences in (b) and (d) (in $B$-grmod) fit into a commutative diagram with exact rows and
 natural surjective vertical maps:
 $$\begin{CD}
 0 @>>> E^{\# m+1} @>>> P' @>>> N^{\# m} @>>> 0\\
 @.  @AAA @AAA @|| @.  \\
 0 @>>> X @>>> R^{\# m} @>>> N^{\# m} @>>> 0\\
 @. @AAA @AAA  @AAA @.\\
 0 @>>> E @>>> R @>>> N @>>> 0.\end{CD}
 $$
 \end{thm}
 \begin{proof} Consider (a). First,   $R$ is the projective cover of $N$ in $B$--mod, so $R$ is a
 $B$-direct summand of $P$. Since $P|_\fa$ is projective, we conclude that $R|_\fa$ is projective
 in $\fa$-mod. Hence, it is projective as a graded $\fa$-module.

 For parts (b)---(d), by Remark \ref{discussion}, $R=\bigoplus_{s\in\mathbb Z}R^{\# s}$, and,
 in this case, $R^{\# s}=0$ if $s> m$ (the semilinearity degree of $N$). In addition to (a), this is the main property of $R$ that
 will be needed.

 Now we prove (c).  Observe, by Remark \ref{discussion}, the surjection $R\twoheadrightarrow N$
 induces surjections $R^{\# s}\twoheadrightarrow N^{\# s}$, for all $s$. Also, because $N$ is semilinear
 of degree $m$, $N|_\fa=(N|_\fa)^{\# m}\oplus (N|_\fa)^{< m}$.  Also, $(N|_\fa)^{< m}=\bigoplus_{s<m}
 N^{\# s}$.

   First, let $X:=
 \ker(R^{\# m}\twoheadrightarrow N^{\# m})$. The module $N^{\# m}|_\fa$ is linear by Lemma
 \ref{syzyLem} and the map from $R^{\# m}$ is surjective. So $X|_\fa$ must be the direct sum of a linear
 module of degree $m+1$ and a graded projective $\fa$-module, the latter a summand of $R^{\# m}$. (This is
 a standard argument using minimal projective covers in $\fa$-grmod, and it is left to the reader.) All
 summands of $R^{\# m}$ are generated in grade $m$, so that $X|_\fa\cong(X|_\fa)^{\# m+1}
 \oplus (X|_\fa)^m$, and $(X|_\fa)^m$ is projective in $\fa$-grmod.

 Second, let $Y=\ker(R^{< m}\twoheadrightarrow N^{<m})$. As noted above,
 $N|_\fa=(N|_\fa)^{\# m}\oplus (N|_\fa)^{< m}$. Clearly, $Y|_\fa$ is a direct sum of projective modules
 generated in grades $<m$.

 However, the given surjection $R\twoheadrightarrow N$ in $B$-grmod need not be the direct sum of
 above surjections $R^{\# m}\twoheadrightarrow N^{\# m}$ and $R^{<m}\twoheadrightarrow N^{<m}$, i.~e., there is a (possibly) different graded $\fa$-module
 surjection $R|_\fa\twoheadrightarrow N|_\fa$. But  Schanuel's lemma and the Krull-Schmidt theorem in
 $\fa$-grmod, $X|_\fa\oplus Y|_\fa\cong E|_\fa$. Consequently, $E|_\fa$ is semilinear of degree $m+1$.
 This proves (b). We also obtain that $X|_\fa$ is semilinear of degree $m+1$.

 By the above decomposition of $E|_\fa$ and of $X|_\fa$, together with Lemma \ref{syzyLem},
 $E^{\# m+1}|_\fa\cong X^{\# m+1}|_\fa$ is linear of degree $m+1$, and that $N^{\# m}|_\fa\cong
 (N|_\fa)^{\# m}$ is linear of degree $m$. This proves (c).

 To prove (d), observe that $X^m$ is a $B$-grmod submodule of $R^{\# m}$, and, as noted above (with the proof
 left to the reader), an $\fa$-summand of $R^{\# m}|_\fa$. In fact, the same analysis shows that the inclusion
 $X^m\subseteq R^{\# m}$ is split upon restriction to $\fa$, with $P':=R^{\# m}/X^m$ projective upon
 restriction to $\fa$. This gives the existence of an exact sequence $0\to E^{\# m+1}\to P'\to N^{\# m}\to 0$
 as required in the existence part of (d). Here we have used the identifications of graded $B$-modules
 $$\begin{aligned}
 E^{\# m+1}=E/E^{\leq m} &\cong (E/E^{< m})/(E^{\leq m}/E^{< m})\\
 & \cong (E/E^{<m})/(E/E^{< m})^{\leq m} \\ & \cong
 X/X^{\leq m} \\
 & =X^{\# m+1}.\end{aligned} $$
 Next, suppose that $P^\dagger\twoheadrightarrow N^{\# m}$ is any surjection in $B$-grmod with
 $P^\dagger|_\fa\twoheadrightarrow N^{\# m}|_\fa$ a projective cover. Then $P^\dagger|_\fa$ is generated in
 grade $m$, so there a commutative diagram (in $B$-grmod)
 $$\begin{CD}
 0 @>>> \Omega @>>> P^\dagger @>>> N^{\# m} @>>> 0 \\
 @.  @AAA  @AAA @AAA  @.\\
 0 @>>> X @>>> R^{\# m} @>>> N^{\#m} @>>> 0\end{CD}
 $$
  with horizontal rows exact. The middle vertical map arises from the projectivity of
  $R\in B$-grmod, the fact that $P^\dagger$ is generated in grade $m$, and the description
  $R^{\# m}=R^{\leq m}/R^{<m}$. This middle vertical map is surjective by Nakayama's lemma. (Note that  $P^\dagger|_\fa\to
  N^{\# m}|_\fa$ is a projective cover as an ungraded map, whether given as a graded or ungraded
  cover, by Remark \ref{discussion}.) The module $\Omega|_\fa$, as a first syzygy, in a minimum
  graded projective resolution of $N^{\# m}|_\fa$, is necessarily linear of degree $m+1$. So the
  vertical map $X|_\fa\to \Omega$ must kill $X^m$. Thus, there is an induced commutative diagram
  $$\begin{CD}
  0 @>>> \Omega @>>> P^\dagger @>>> N^{\# m} @>>> 0\\
  @. @AAA @AAA @|| @.\\
  0 @>>> X^{\# m+1} @>>> P' @>>> N^{\#m} @>>> 0,\end{CD}
  $$
  where $P'=R^{\# m}/X^m$ is as constructed above.  Since
  $P'|_\fa$ and $P^\dagger|_\fa$, as projective covers of $N^{\# m}|_\fa$, both have the same dimension,
  the surjective middle vertical map is an isomorphism. We have already identified $X^{\# m+1}\cong
  E^{\# m+1}$, so $\Omega\cong E^{\# m+1}$. If we are given any exact sequence $0\to E^{\# m+1}\to
  P^{\prime\prime}\to N^{\# m}\to 0$ with $P^{\prime\prime}|_\fa$ projective, then $P^{\prime\prime}$ has the same dimension as $P'$, so the above argument gives both an isomorphism $P^{\prime\prime}\cong P'$,
  and a similar isomorphism of exact sequences with end terms $E^{\# m+1}$ and $N^{\# m}$. This proves
  (d).

  Finally, (e) is easily obtained from the descriptions of $X=\ker(R^{\# m}\to N^{\# m})$, and
  $P'=R^{\# m}/X^m$ in the discussion above.  The map $E\to X$ is surjective by a snake lemma argument.
  (Note that $R^{<m}\to N^{<m}$ is surjective.)
  \end{proof}

  \subsection{Gradings induced by graded subalgebras.} An important case occurs when the
  grading of the algebra $B$ results from a filtration of another algebra $A$, induced by a sufficiently ``normal"
  graded subalgebra $\fa$. More precisely, throughout this subsection, the
  following conditions are in force:

  \begin{enumerate}
  \item [(i)] $\fa$ is positively graded and $\fa\to A$ is a homomorphism of algebras.

  \item [(ii)] For each $j\geq 0$, put $\fa_{\geq j}:=\bigoplus_{i\geq j}\fa_i$. Then $\fa_{\geq j}A$ is
  required to be
  an ideal in $A$. That is, $A\fa_{\geq j}A=\fa_{\geq j}A$. (In applications,
  $A\fa_{\geq j}=\fa_{\geq j}A$.)

  \item [(iii)] Define
  $$B=\wgr A:=\bigoplus_{j\geq 0} \fa_{\geq j}A/\fa_{\geq j+1}A.$$
  \end{enumerate}

  Condition (ii) implies that the algebra $B$ defined above is  positively graded. There is a graded morphism
  $\fa\to B$ such that
   $\fa_jB_0=B_j$, for each $j\geq 0$. That is, $B$ is tight over $\fa$, as per Definition \ref{tighty}. In most applications, the map $\fa\to B$ will be an inclusion.

  Every $A$-module $M$ is naturally an $\fa$-module,  so the $\fa$-modules $\hd^\flat M$ and
  $\rad^\flat M$ are defined, using Definition \ref{flathead}. By (ii), they are also modules for $A/\fa_{\geq 1} A=B_0$.

  \begin{defn} \label{hybrid}
  An $A$-module equipped with a fixed graded $\fa$-module structure
  will be called {\it hybrid}. Morphisms of hybrid $A$-modules are just morphisms of $A$-modules which
  preserve the given $\fa$-gradings.
  \end{defn}

  The hybrid $A$-modules form an abelian category, exactly embedded
  in the category of $A$-modules. A hybrid $A$-module $N$ is {\it admissible} if each subpace $N_{\geq j}:=
  \bigoplus_{i\geq j}N_i$ is an $A$-submodule. The admissible objects form a full abelian subcategory of the category of hybrid $A$-modules. Given an admissible hybrid $A$-module $N$, one can form a
  graded $B$-module
  $$\wGr N:=\bigoplus_{j\in\mathbb Z}N_{\geq j}/N_{\geq j+1}.$$
  Here the capitalizing $\wGr$ is used to help distinguish this module from
  $$\wgr N:=\bigoplus_{j\geq 0}\fa_{\geq j}N/\fa_{\geq j+1}N$$
  defined in (\ref{gradedzoo}).

  The category of hybrid $A$-modules is equipped with natural grade shifting functors $N\mapsto N\langle r\rangle $, for every $r\in{\mathbb Z}$. Recall from \S1.1 that
  $N\langle r\rangle _j:=N_{j-r}$. If $N$ is admissible, so is $N\langle r\rangle $, and
  $$\wGr N\langle r\rangle =(\wGr N)\langle r\rangle , \quad r\in{\mathbb Z}.$$
  Finally,
  $$(\wGr N)|_\fa \cong N|_\fa, \quad{\text{\rm in $\fa$-grmod}}.$$

We now construct some admissible hybrid $A$-modules. Suppose that $R$ is an $A$-module equipped with a
decreasing filtration  by $A$-submodules $\{{^iR}\}$, $i\in\mathbb Z$; thus, $\cdots \supseteq{^{i-1}R}\supseteq {^iR}\supseteq{^{i+1}R}\supseteq\cdots$. Assume that $^iR/{^{i+1}R}$ is projective
as an $\fa$-module, for each $i$. Also, assume that $^iR=R$, for $i$ sufficiently small, and $^iR=0$, for $i$ sufficiently large. In particular, each $^iR$ is projective as an $\fa$-module and has a decomposition
$^iR={^iR}/{^{i+1}R} \oplus {^{i+1}R}$ as a direct sum of projective modules.
Choose any $\fa_0$-stable complement $h_i$ to $^{i+1}R + \rad^\flat({^iR})$ in $^iR$. (Such an $h_i$ exists
because all projective $\fa$-modules $X$ may be given a grading $X\cong \fa\otimes_{\fa_0}\hd^\flat X$
corresponding to any $\fa_0$-grading of $\hd^\flat X$.) Then $h:=\sum_{i\in\mathbb Z}h_i\cong\bigoplus_{i\in\mathbb Z}h_i$ is an $\fa_0$-submodule of $R$ and a complement to $\rad^\flat R$. As an $\fa$-module,
$$R=\fa h\cong\fa\otimes_{\fa_0}h\cong \bigoplus_i \fa\otimes_{\fa_0}h_i.$$
 We now give $R$ an $\fa$-grading by assigning each $h_i$ grade $i$. The resulting hybrid structure on $R$ is admissible, since
$$R_{\geq s}={^sR} + \fa_{\geq 1}({^{s-1}R}) + \fa_{\geq 2}({^{s-2}R}) +\cdots.$$
The flexibility to choose the $\fa_0$-submodules $h_i$ generating the $\fa$-grading is quite useful.

\begin{prop}\label{hybrid} Let the $A$-module $R$ have a decreasing filtration $\{{^iR}\}_{i\in\mathbb Z}$ as above. Suppose that $N$ is an admissible hybrid $A$-module, and $\phi:R\twoheadrightarrow N$ is a surjection of $A$-modules such that $\phi({^iR})
=N_{\geq i}$, for each $i\in \mathbb Z$. Then there is a choice of $\fa_0$-submodules $h_i$ so that, as
above, $h_i$ is an $\fa_0$-stable complement to $^{i+1}R+ \rad^\flat({^iR})$ in $^iR$, and, additionally,
$\phi(h_i)\subseteq N_i$ ($i\in\mathbb Z$). The induced $\fa$-grading on $R$ gives $R$ an admissible hybrid
structure and $\phi$ becomes a surjective homomorphism of admissible hybrid $A$-modules.\end{prop}

\begin{proof} The proposition is trivial if $R=0$, in which case just take all $h_i=0$. We may, thus, proceed by
induction on $\dim R$. Let $m\in \mathbb Z$ be maximal with $^mR=R$. Thus $^{m+1}R\not\subseteq R$.
So, the proposition holds, by induction, when $^{m+1}R$, $N_{\geq m+1}$, and $\phi|_{^{m+1}R}$ play the roles of $R$, $N$, and $\phi$, respectively. This gives $h_{m+1}, h_{m+2}, \cdots$ contained in
$^{m+1}R$, $^{m+2}R, \cdots$, respectively, such that each $h_i$ is an $\fa_0$-stable complement
to $^{j+1}R + \hd^\flat{^jR}$ in $^jR$ ($j\geq m+1)$. We need to find an $h_j$ for
$j=m$ with this property.

Put $S=\phi^{-1}(N_m)$. Then $\phi(S + {^{m+1}R})=N_{\geq m}=\phi({^mR})=\phi(R)$. Since
$\ker\phi\subseteq S\subseteq S+{^{m+1}R}$, we must have $S+{^{m+1}R}=R={^mR}$. The
surjection
$$S\to {^mR}/{^{m+1}R}\twoheadrightarrow\hd^\flat({^mR}/{^{m+1}R})$$
 is $\fa_0$-split, since the projective $\fa$-module
$X:=({^mR}/{^{m+1}R})|_\fa$ has $\hd^\flat X=X/\fa_{\geq 1}X$ as a natural projective $\fa_0\cong
\fa/\fa_{\geq 1}$-quotient module. Let $h_m$ be the image in $S$ of the splitting. By construction, the image
$$(h_m +{^{m+1}R} +\rad^\flat {^mR})/
({^{m+1}R} +\rad^\flat {^mR})$$
of $h_m$ under the map
$$S\subseteq {^mR}\to {^mR}/{^{m+1}R}\to \hd^\flat({^mR}/{^{m+1}R})\cong {^mR}/({^{m+1}R}+
\rad^\flat {^mR})$$
is an $\fa_0$-isomorphic copy of $h_m$, equal to the target of map. That is, $^mR$ is a direct sum
$h_m\oplus ({^{m+1}R} +\rad^\flat {^mR})$. Finally, $\phi(h_m)\subseteq\phi(S)\subseteq N_m$.
This proves the proposition, since $\phi$ now becomes an $\fa$-graded map on the constructed
$\fa$-grading of $R$. (Recall that $R=\fa h \cong \fa\otimes_{\fa_0} h$ as an $\fa$-module, where
$h=\sum_{j\in\mathbb Z}h_j=\bigoplus_{j\in\mathbb Z} h_j$. \end{proof}

\begin{cor}\label{corlast1} Let $N$ be an admissible hybrid $A$-module, and suppose, for each $s\in\mathbb Z$,
there is given a projective $A$-module $^{\#s}R$ and a surjection
$$\phi_s:{^{\#s}R}\twoheadrightarrow  N_{\geq s}/N_{\geq s+1}.$$
Assume that each $^{\#s}R|_\fa$ is projective and that $^{\#s}R=0$, for $|s|\gg 0$.
Lift each $\phi_s$ in any way to a $A$-module homomorphism  $\phi_{\geq 0}:{^{\#s}R}\to N_{\geq s}$. Put
$R:=\bigoplus_{s\in\mathbb Z}{^{\# s}R}$ and let $\phi:R\to N$ denote the sum of the maps $\phi_{\geq s}$.

Then $R$ has the structure of an admissible hybrid $A$-module in such a way that $\phi$ becomes a surjective homomorphism of admissible hybrid $\fa$-modules. In particular, if $E=\ker \phi$, then $E$
is admissible, and
$$0\to E\longrightarrow R\longrightarrow N\to 0$$
remains exact upon applying the functors $-|_\fa$ and $\wGr$, giving graded exact sequences in each case
(in $\fa$-grmod and $B$-grmod, respectively). The modules $R|_\fa$ and $\wGr R$ are
projective in $\fa$-grmod and $B$-grmod, respectively.
\end{cor}

\begin{proof} Put $^jR=\bigoplus_{s\geq j} {^{\# s}R}\subseteq R$, for each $j\in\mathbb Z$. The hypotheses of
Proposition \ref{hybrid} are then satisfied. So there are admissible hybrid $A$-module structures
on the object $R$ and the morphism $\phi$, required in the first assertion. The second
assertion is just a property of all exact sequences in the category of admissible hybrid $A$-modules, and
has been essentially previously noted below Definition \ref{hybrid} (and is obvious, in any case). The final assertion, regarding graded projectivity, follows from
the following isomorphisms of graded $B$-modules:
$$ \wGr({^sR})/\wGr ({^{m+1}R})\cong (\wgr{^{\# s}R})\langle s\rangle ,\quad s\in\mathbb Z,$$
which is easily obtained by inspecting the construction. The right hand side is clearly projective both as
a graded $B$-module and as a graded $\fa$-module. This completes the proof. \end{proof}

\begin{rem}\label{remgraded}
We can sometimes trim some of the terms $^{\# s}R$ from $R$. Let $m$ be an integer such that
$N_{\geq m}=\fa N_m$. (If $\fa$ is tightly graded---that is, if $\fa$ is generated by $\fa_0$ and $\fa_1$---and if $N|_\fa$ is semilinear of degree $m$, then this equality
holds.) In this case, we do not need any $^{\# s}R$ with $s>m$, and we may redefine $^{\#  s}R=0$
and $\phi_s=0$ in the definition of $R$ and $\phi$, ignoring the requirements in the hypothesis
of Corollary \ref{corlast1} and Proposition \ref{hybrid}, so that $\phi({^i}R)=N_{\geq i}$ is assumed only
for $i\leq m$ with $^iR=0$ assumed for $i>m$. The modified analogue of Proposition \ref{hybrid}
is proved essentially as above, but beginning the argument by observing, for $S:=\phi^{-1}(N_m)\cap
{^mR}$,
$$\phi(S+\fa_{\geq 1}{^m R})=N_m+\fa_{\geq 1}N_{\geq m}=
N_{\geq m}=\phi({^mR}).$$
The revised corollary then follows as before from the modified proposition. For the convenience of the
reader, we state these two results as Proposition \ref{revisedproposition} and Corollary \ref{revisedcorollary}, without
further details of their proofs.
\end{rem}

\begin{prop}\label{revisedproposition} Let the $A$-module $R$ have a decreasing filtration $\{{^iR}\}_{i\in\mathbb Z}$ as above the statement of Proposition \ref{hybrid}. Suppose that $N$ is an admissible hybrid $A$-module, and  $m$ is an integer with
$\fa N_m=N_{\geq m}$. Suppose $\phi:R\twoheadrightarrow N$ is a surjection of $A$-modules such that $\phi({^iR})
=N_{\geq i}$, for each $i\in \mathbb Z$ with $i\leq m$, and $^iR=0$, for $i>m$. Then there is a choice of $\fa_0$-submodules $h_i$ so that $h_i$ is an $\fa_0$-stable complement to $^{i+1}R+ \rad^\flat{^iR}$ in $^iR$, and, additionally,
$\phi(h_i)\subseteq N_i$ ($i\in\mathbb Z$). The induced $\fa$-grading on $R$ gives $R$ an admissible hybrid
structure and $\phi$ becomes a surjective homomorphism of admissible hybrid $A$-modules.\end{prop}

\begin{cor}\label{revisedcorollary}  Let $N$ be an admissible hybrid $A$-module, and let $m\in\mathbb Z$
be such that $\fa N_m=N_{\geq m}$. Suppose, for each $s\leq m$,
there is given a projective $A$-module $^{\#s}R$ and a surjection
$$\phi_s:{^{\#s}R}\twoheadrightarrow  N_{\geq s}/N_{\geq s+1}.$$
Assume that $^{\#s}R|_\fa$ is projective and that $^{\#s}R=0$, for $|s|\gg 0$ (or
$s>m$).
Lift each $\phi_s$ in any way to an $A$-module homomorphism  $\phi_{\geq s}:{^{\#s}R}\to N_{\geq s}$. Put
$R:=\bigoplus_{s\in\mathbb Z}{^{\# s}R}$ and let $\phi:R\to N$ denote the sum of the maps $\phi_{\geq s}$.

Then $R$ has the structure of an admissible hybrid $A$-module in such a way that $\phi$ becomes a surjective homomorphism of admissible hybrid $\fa$-modules. In particular, if $E=\ker \phi$, then $E$
is admissible, and
$$0\to E\longrightarrow R\longrightarrow N\to 0$$
remains exact upon applying the functors $-|_\fa$ and $\wGr$, giving graded exact sequences in each case
(in $\fa$-grmod and $B$-grmod, respectively). The modules $R|_\fa$ and $\wGr R$ are
projective in $\fa$-grmod and $B$-grmod, respectively.
\end{cor}

We reformulate these latter conclusions in part (a) of the proposition below. The hypotheses of
Corollary \ref{revisedcorollary} above are assumed, as they are in Theorem \ref{secondmain}.

\begin{prop} \label{revised} Let $N$ be an admissible hybrid $A$-module. Let $\phi:R\twoheadrightarrow N$ be the morphism of admissible hybrid $A$-modules constructed above, with $R=\bigoplus {^{\# s}}R$. (In particular, $R$ is $A$-projective.) Let $E=\ker(\phi)$ and form the exact sequence
\begin{equation}\label{sex}
0\to E\to R\to N\to 0\end{equation}
in the category of admissible hybrid $A$-modules, as in Corollary \ref{revisedcorollary}.

(a) The modules $R$, $R|_\fa$, $\wGr R$
are projective in $A$-mod, $\fa$-grmod, and $\wgr A$-grmod, respectively. Further, (\ref{sex}) gives rise
to three exact sequences
$$\begin{cases} 0\to E\to R\to N\to 0;\\ 0\to E|_\fa\to R|_\fa\to N|_\fa\to 0;
\\ 0\to \wGr E\to \wGr R\to \wGr N\to 0\end{cases}$$
in the categories $A$-mod, $\fa$-grmod, and $B$-grmod, respectively.

(b)  In addition, if $N|_\fa$ is semilinear of degree $m$ and if $\fa$ is Koszul (or just tight---generated by $\fa_0$ and $\fa_1$), then $E|_\fa$ is also semilinear, of degree $m+1$ (as is $\wGr E|_\fa\cong E|_\fa$).  Let $E',N'$ denote the maximal linear quotients of
$\wGr E$ and $\wGr N$ of degrees $m+1$ and $m$, respectively. Then there is an induced exact
sequence
\begin{equation}\label{sex2}0\to E'\to R'\to N'\to 0\end{equation}
in $B$-grmod. Here $R'=\wGr R/X$, where $X$ is the image in $\wGr R$ of $\ker(\wGr E\to E')$. Also, $R'|_\fa$ is projective
in $\fa$-grmod.
\end{prop}

\begin{proof} Everything except the last assertion has been outlined in the Remark \ref{remgraded}. For that last assertion, the argument in the proof of Theorem \ref{mainSecThm}(d) may be used.\end{proof}

Before stating the second main theorem of this section, note that, for any admissible hybrid $A$-module $X$,
there is an obvious ungraded isomorphism
$$\hd^\flat(\wGr X)\cong\hd^\flat X \quad{\text{\rm in}}\,\,A/{\fa_{\geq 1} A}{\text{\rm --mod}}.$$
The algebra $A/\fa_{\geq 1}A$ is $(\wgr A)_0$, by definition. There is even a natural isomorphism
of graded $(\wgr A)_0$-modules
$$\hd^\flat(\wGr X)\cong\wGr(\hd^\flat X).$$

\begin{thm}\label{secondmain} Let $0\to E\to R\to N\to 0$ be as in (\ref{sex}) and let $V$ be any
$B_0=A/\fa_{\geq 1}A$-module.  Then the following statements hold.

(a) There is a natural isomorphism
\begin{equation}\label{natiso}
\coker(\Hom_A(R,V)\to\Hom_A(E,V))\cong\coker(\Hom_{B}(\wGr R,V)\to
\Hom_{B}(\wGr E,V)).\end{equation}

(b) If $\fa$ is tightly graded,  if $N|_\fa$ is semilinear of degree $m$, and if $0\to E'\to R'\to N'\to 0$ is as in
(\ref{sex2}), then there  is a natural isomorphism
$$\Hom_\fa(E',V)\cong\coker(\Hom_\fa(R,V)\to\Hom_\fa(E,V))$$
of vector spaces induced by the quotient maps $\wGr E\to E'$, $\wGr R\to R'$, together with the analogue of (\ref{natiso}) for $\fa$.
\end{thm}

\begin{proof} Assertion (a)  is a consequence of the natural isomorphisms
$$\begin{aligned} \Hom_A(X,V) &\cong \Hom_{A/\fa_{\geq 1} A}(\hd^\flat X,V)\\
&\cong(\Hom_{(\wgr A)_0}(\hd^\flat X,V)\\
&\cong\Hom_{B}(\wGr X,V),\end{aligned}$$
for all admissible hybrid $A$-modules $X$.

For (b), apply Theorem \ref{mainSecThm} (though not with the same notation). First, as above,
$$\ck(\Hom_\fa(R,V)\to\Hom_\fa(E,M))\cong\ck(\Hom_\fa(\wGr R,V)\to\Hom_\fa(\wGr E,M)).$$
Next, we construct a projective cover  $P\twoheadrightarrow \wGr N$ in $\wgr A$-grmod, as in Remark
\ref{discussion}, with $P=P(\hd^\flat\wGr N)$. Note that $P=\bigoplus_{s\in\mathbb Z} P^{\# s}$ by
construction, using the notation of  Remark \ref{discussion}, except that $P^{\# s}$ is
the projective cover of $(\hd^\flat \wGr N)_s\cong(\wGr \hd^\flat N)_s$. Recall that $R$,
constructed as in Remark \ref{remgraded}, is generated in grades $\leq m$ over $\fa$, as is $\wGr R$ (over
$\fa$ or over $\wgr A$). The surjection $\wGr R\twoheadrightarrow \wGr N$ lifts to a split surjection
$\wGr R\twoheadrightarrow P$ in $\wgr A$-grmod. Let $F$ be its kernel. Standard diagram arguments
show that $\wGr E\cong F\oplus \ker(P\to\wGr N)$ in $\wgr A$-grmod (and in $\fa$-grmod, consequently).
By Theorem \ref{mainSecThm}, $\ker(P\to\wGr N)$ is semilinear of degree $m+1$. Clearly, $F|_\fa$ is projective in
$\fa$-grmod and is generated in grades $\leq m$ (properties inherited from $\wGr R$). Therefore,
$\ker(P\to\wGr N)|_\fa$ and $(\wGr E)|_\fa\cong E|_\fa$ share the same maximal quotient $E'$ which is linear
of degree $m+1$. Also, $E'$ carries the same $\wgr A$-module structure from $\wGr R$ as from $\ker(P\to\wGr N)$. Theorem \ref{mainSecThm} guarantees that there is, up to isomorphism, a unique exact sequence
$$0\to E'\to P'\to N'\to 0
\quad {\text{\rm in}}\,\,\wgr A{\text{\rm --grmod}},$$
with $N'$ the degree $m$ maximal linear quotient of $\wGr N$, $P'$ an object
in $\wgr A$-grmod with $P|_\fa$ projective. The commutative diagram in Theorem \ref{mainSecThm}(a) may now
be used to produce a commutative diagram
$$\begin{CD} 0 @>>> E' @>>> P' @ >>> N' @>>> 0\\
@.  @AAA @ AAA @ AAA @.\\
0 @>>> \wGr E @>>> \wGr R @>>> \wGr N @>>> 0\end{CD}$$
 in $\wgr A$-mod, with exact rows and with all vertical maps surjective. Both $\wGr R|_\fa$ and $P'|_\fa$
 are projective in $\fa$-grmod, and $P$ is (consequently) generated in grades $\leq m$. So $P'_{\geq m+1}
 \subseteq\fa_{\geq 1}P'$. Thus,
 $$\begin{aligned}\ck(\Hom_\fa(P',V)\to\Hom_\fa(E',V)) &\cong\ck(\Hom_\fa(\hd^\flat(P',V)\to\Hom_\fa(\hd^\flat E',V))\\
& \cong \Hom_\fa(\hd^\flat E',V).\end{aligned}$$

 Finally, it is clear that the complex consisting of the bottom row is the direct sum of a complex of
 projective modules in $\fa$-grmod. So there is a natural isomorphism
 $$\ck(\Hom_\fa(P',V)\to\Hom_\fa(E',V))\cong\ck(\Hom_\fa(\wGr R,V)\to\Hom_\fa(\wGr E,V)).$$
 Together with the identifications previously noted. This proves the second assertion of the theorem.
 \end{proof}

\begin{rem}\label{lastlastremark} The theory resulting from Proposition \ref{revised} and Theorem \ref{secondmain} is a recursive
one, for the purpose of building resolutions one step at a time.  The recursive design is made necessary by the
hypothesis, appearing in Corollary \ref{revisedcorollary} (and earlier), that requires sufficient projective $A$-modules exist with
projective restrictions to $\fa$ to be able to make the constructions. In the situations we must deal with in 
Section 5, this generality requires enlarging the algebra $A$. We refer the reader to the proofs of Theorem \ref{nextmainresult} and Theorem \ref{GExt} for  cases where this can be done successfully, the latter requiring results
of this section only through Corollary \ref{corlast1}. The algebras $A$ are various $A_\Gamma$'s and $B=\wgr A$. The algebra $\fa$ is introduced in \S2.5. The hypothesis in Theorem \ref{mainSecThm} that $B$ is
tight over $\fa$ follows from the definition of $\wgr A$ and the tightness of $\wA$ over $\wfa$.   \end{rem}

  \section{Appendix II: vanishing of Tor} This appendix proves the following general fact about integral quasi-hereditary
algebras. Let $\wA$ be an
integral quasi-hereditary  over $\sO$ with poset $\Lambda$. Recall that given $\lambda\in\Lambda$, $\wDelta(\lambda)$
is the standard {\it left} module defined by $\lambda$. Let $\wDelta(\lambda)^\circ$ be the standard
{\it right} module defined by $\lambda$.

\begin{prop}\label{appendixprop} Let $\wA$ be a split quasi-hereditary algebra over $\sO$ with weight poset $\Lambda$. For $\lambda,\mu\in
\Lambda$, we have
$$\Tor^\wA_n(\wDelta(\lambda)^\circ,\wDelta(\mu))\cong\begin{cases} \sO\quad {\text{\rm when $n=0$ and $\lambda=\mu$}};\\
0 \quad{\text{\rm otherwise}}.\end{cases} \quad\forall n\in{\mathbb N}.$$
A similar result holds when $\sO$ is replaced by a field.
\end{prop}

\begin{proof}One can reduce easily to the case of a quasi-hereditary algebra $A$ over a field $F$.
(The argument is similar to that used in \cite[p. 5243]{CPS7}.) For
convenience, we assume that the poset $\Lambda$ is linear.

 First, consider the case $n=0$.  We assume
that $\lambda\geq \mu$ (and leave the other case to the reader). Without loss, we can replace
$A$ by a quotient quasi-hereditary $B$ with $\lambda$ maximal in the poset of $B$. Thus, $\Delta(\lambda)^\circ$
is a projective indecomposable module $eB$ with $e$ a primitive idempotent. Then $\Delta(\lambda)^\circ
\otimes_B\Delta(\mu)=eB\otimes_B\Delta(\mu)\cong e\Delta(\mu)$. If $\lambda\not=\mu$, then $e\Delta(\mu)=0$.
If $\lambda=\mu$, then $e\Delta(\mu)=eBe=F$, since $B$ is split. This completes the proof when
$n=0$.

Now assume that $n>0$. Again, we treat only the case $\mu\geq\lambda$, leaving the other
case to the reader. Then
then the projective indecomposable $A$-module $P(\mu)$ has a $\Delta$-filtration with
top section $\Delta(\lambda)$ and lower sections of the form $\Delta(\tau)$, for $\tau>\mu$. Now
an evident induction on $\mu$ completes the proof. (Observe the assertion is trivial if $\mu$ is
maximal.)
\end{proof}

 \end{document}